\documentclass[12pt]{article}
\usepackage{latexsym, amssymb, amsthm}
\usepackage{dsfont}

\DeclareMathAlphabet\gothic{U}{euf}{m}{n}

\makeatletter
\def\eqnarray{\stepcounter{equation}\let\@currentlabel=\theequation
\global\@eqnswtrue
\tabskip\@centering\let\\=\@eqncr
$$\halign to \displaywidth\bgroup\hfil\global\@eqcnt\z@
  $\displaystyle\tabskip\z@{##}$&\global\@eqcnt\@ne
  \hfil$\displaystyle{{}##{}}$\hfil
  &\global\@eqcnt\tw@ $\displaystyle{##}$\hfil
  \tabskip\@centering&\llap{##}\tabskip\z@\cr}

\def\endeqnarray{\@@eqncr\egroup
      \global\advance\c@equation\m@ne$$\global\@ignoretrue}

\def\@yeqncr{\@ifnextchar [{\@xeqncr}{\@xeqncr[5pt]}}
\makeatother

\begin{document}
\bibliographystyle{tom}

\newtheorem{lemma}{Lemma}[section]
\newtheorem{thm}[lemma]{Theorem}
\newtheorem{cor}[lemma]{Corollary}
\newtheorem{prop}[lemma]{Proposition}

\theoremstyle{definition}

\newtheorem{remark}[lemma]{Remark}
\newtheorem{exam}[lemma]{Example}
\newtheorem{definition}[lemma]{Definition}

\newcommand{\gota}{\gothic{a}}
\newcommand{\gotb}{\gothic{b}}
\newcommand{\gotc}{\gothic{c}}
\newcommand{\gote}{\gothic{e}}
\newcommand{\gotf}{\gothic{f}}
\newcommand{\gotg}{\gothic{g}}
\newcommand{\gothh}{\gothic{h}}
\newcommand{\gotk}{\gothic{k}}
\newcommand{\gotm}{\gothic{m}}
\newcommand{\gotn}{\gothic{n}}
\newcommand{\gotp}{\gothic{p}}
\newcommand{\gotq}{\gothic{q}}
\newcommand{\gotr}{\gothic{r}}
\newcommand{\gots}{\gothic{s}}
\newcommand{\gott}{\gothic{t}}
\newcommand{\gotu}{\gothic{u}}
\newcommand{\gotv}{\gothic{v}}
\newcommand{\gotw}{\gothic{w}}
\newcommand{\gotz}{\gothic{z}}
\newcommand{\gotA}{\gothic{A}}
\newcommand{\gotB}{\gothic{B}}
\newcommand{\gotG}{\gothic{G}}
\newcommand{\gotL}{\gothic{L}}
\newcommand{\gotS}{\gothic{S}}
\newcommand{\gotT}{\gothic{T}}

\newcounter{teller}
\renewcommand{\theteller}{(\alph{teller})}
\newenvironment{tabel}{\begin{list}%
{\rm  (\alph{teller})\hfill}{\usecounter{teller} \leftmargin=1.1cm
\labelwidth=1.1cm \labelsep=0cm \parsep=0cm}
                      }{\end{list}}

\newcounter{tellerr}
\renewcommand{\thetellerr}{(\roman{tellerr})}
\newenvironment{tabeleq}{\begin{list}%
{\rm  (\roman{tellerr})\hfill}{\usecounter{tellerr} \leftmargin=1.1cm
\labelwidth=1.1cm \labelsep=0cm \parsep=0cm}
                         }{\end{list}}

\newcounter{tellerrr}
\renewcommand{\thetellerrr}{(\Roman{tellerrr})}
\newenvironment{tabelR}{\begin{list}%
{\rm  (\Roman{tellerrr})\hfill}{\usecounter{tellerrr} \leftmargin=1.1cm
\labelwidth=1.1cm \labelsep=0cm \parsep=0cm}
                         }{\end{list}}

\newcounter{proofstep}
\newcommand{\nextstep}{\refstepcounter{proofstep}\vertspace \par 
          \noindent{\bf Step \theproofstep} \hspace{5pt}}
\newcommand{\firststep}{\setcounter{proofstep}{0}\nextstep}

\newcommand{\Ni}{\mathds{N}}
\newcommand{\Qi}{\mathds{Q}}
\newcommand{\Ri}{\mathds{R}}
\newcommand{\Ci}{\mathds{C}}
\newcommand{\Ti}{\mathds{T}}
\newcommand{\Zi}{\mathds{Z}}
\newcommand{\Fi}{\mathds{F}}
\newcommand{\Ki}{\mathds{K}}
\newcommand{\Di}{\mathds{D}}

\renewcommand{\proofname}{{\bf Proof}}

\newcommand{\vertspace}{\vskip10.0pt plus 4.0pt minus 6.0pt}

\newcommand{\simh}{{\stackrel{{\rm cap}}{\sim}}}
\newcommand{\ad}{{\mathop{\rm ad}}}
\newcommand{\Ad}{{\mathop{\rm Ad}}}
\newcommand{\alg}{{\mathop{\rm alg}}}
\newcommand{\clalg}{{\mathop{\overline{\rm alg}}}}
\newcommand{\Aut}{\mathop{\rm Aut}}
\newcommand{\arccot}{\mathop{\rm arccot}}
\newcommand{\capp}{\mathop{\rm cap}}
\newcommand{\rcapp}{\mathop{\rm rcap}}
\newcommand{\diam}{\mathop{\rm diam}}
\newcommand{\divv}{\mathop{\rm div}}
\newcommand{\dom}{\mathop{\rm dom}}
\newcommand{\graph}{\mathop{\rm graph}}
\newcommand{\codim}{\mathop{\rm codim}}
\newcommand{\RRe}{\mathop{\rm Re}}
\newcommand{\IIm}{\mathop{\rm Im}}
\newcommand{\tr}{{\mathop{\rm tr \,}}}
\newcommand{\Tr}{{\mathop{\rm Tr \,}}}
\newcommand{\Vol}{{\mathop{\rm Vol}}}
\newcommand{\card}{{\mathop{\rm card}}}
\newcommand{\rank}{\mathop{\rm rank}}
\newcommand{\supp}{\mathop{\rm supp}}
\newcommand{\sgn}{\mathop{\rm sgn}}
\newcommand{\essinf}{\mathop{\rm ess\,inf}}
\newcommand{\esssup}{\mathop{\rm ess\,sup}}
\newcommand{\Int}{\mathop{\rm Int}}
\newcommand{\lcm}{\mathop{\rm lcm}}
\newcommand{\loc}{{\rm loc}}
\newcommand{\HS}{{\rm HS}}
\newcommand{\Trn}{{\rm Tr}}
\newcommand{\n}{{\rm N}}
\newcommand{\WOT}{{\rm WOT}}

\newcommand{\at}{@}

\newcommand{\mod}{\mathop{\rm mod}}
\newcommand{\spann}{\mathop{\rm span}}
\newcommand{\one}{\mathds{1}}

\hyphenation{groups}
\hyphenation{unitary}

\newcommand{\tfrac}[2]{{\textstyle \frac{#1}{#2}}}

\newcommand{\ca}{{\cal A}}
\newcommand{\cb}{{\cal B}}
\newcommand{\cc}{{\cal C}}
\newcommand{\cd}{{\cal D}}
\newcommand{\ce}{{\cal E}}
\newcommand{\cf}{{\cal F}}
\newcommand{\ch}{{\cal H}}
\newcommand{\chs}{{\cal HS}}
\newcommand{\ci}{{\cal I}}
\newcommand{\ck}{{\cal K}}
\newcommand{\cl}{{\cal L}}
\newcommand{\cm}{{\cal M}}
\newcommand{\cn}{{\cal N}}
\newcommand{\co}{{\cal O}}
\newcommand{\cp}{{\cal P}}
\newcommand{\cs}{{\cal S}}
\newcommand{\ct}{{\cal T}}
\newcommand{\cx}{{\cal X}}
\newcommand{\cy}{{\cal Y}}
\newcommand{\cz}{{\cal Z}}

\newlength{\hightcharacter}
\newlength{\widthcharacter}
\newcommand{\covsup}[1]{\settowidth{\widthcharacter}{$#1$}\addtolength{\widthcharacter}{-0.15em}\settoheight{\hightcharacter}{$#1$}\addtolength{\hightcharacter}{0.1ex}#1\raisebox{\hightcharacter}[0pt][0pt]{\makebox[0pt]{\hspace{-\widthcharacter}$\scriptstyle\circ$}}}
\newcommand{\cov}[1]{\settowidth{\widthcharacter}{$#1$}\addtolength{\widthcharacter}{-0.3pt}\settoheight{\hightcharacter}{$#1$}\addtolength{\hightcharacter}{0.1ex}#1\raisebox{\hightcharacter}{\makebox[0pt]{\hspace{-\widthcharacter}$\scriptstyle\circ$}}}
\newcommand{\scov}[1]{\settowidth{\widthcharacter}{$#1$}\addtolength{\widthcharacter}{-0.15em}\settoheight{\hightcharacter}{$#1$}\addtolength{\hightcharacter}{0.1ex}#1\raisebox{0.7\hightcharacter}{\makebox[0pt]{\hspace{-\widthcharacter}$\scriptstyle\circ$}}}

\thispagestyle{empty}

\vspace*{1cm}
\begin{center}
{\Large\bf The Perron solution for elliptic equations \\[10pt]
without the maximum principle} \\[4mm]

\large W. Arendt$^1$, A.F.M. ter Elst$^2$ and M. Sauter$^1$

\end{center}

\vspace{4mm}

\begin{center}
{\bf Abstract}
\end{center}

\begin{list}{}{\leftmargin=1.8cm \rightmargin=1.8cm \listparindent=10mm 
   \parsep=0pt}
\item
In this article we consider the Dirichlet problem 
on a bounded domain $\Omega \subset \Ri^d$ with respect to a 
second-order elliptic differential operator in divergence form.
We do not assume a divergence condition as in the 
pioneering work by Stampacchia, but merely assume that $0$ is not 
a Dirichlet eigenvalue.
The purpose of this article is to define and investigate a solution 
of the Dirichlet problem, which we call Perron solution, 
in a setting where no maximum principle is available.
We characterise this solution in different ways: by approximating the domain
by smooth domains from the interior, by variational properties, 
by the pointwise boundary behaviour at regular boundary points and 
by using the approximative trace.
We also investigate for which boundary data the Perron solution 
has finite energy.
Finally we show that the Perron solution is obtained as an 
$H^1_0$-perturbation of a continuous function on~$\overline \Omega$.
This is new even for the Laplacian and solves an open problem.
\end{list}

\vspace{6mm}
\noindent
May 2023.

\vspace{3mm}

\vspace{3mm}

\vspace{6mm}

\noindent
{\bf Home institutions:}    \\[3mm]
\begin{tabular}{@{}cl@{\hspace{10mm}}cl}
1. & Institute of Applied Analysis & 
2. & Department of Mathematics  \\
& Ulm University & 
  & University of Auckland  \\
& Helmholtzstr.\ 18  &
  & Private bag 92019 \\
& 89081 Ulm  &
   & Auckland 1142 \\ 
& Germany  &
  & New Zealand \\[8mm]
\end{tabular}

\newpage

\section{Introduction} \label{Sdat1}

Let $\Omega\subset\Ri^d$ be an open and bounded set with boundary~$\Gamma$.
The Dirichlet problem for the Laplacian asks to find a function 
$u$ such that
\[
\Delta u = 0 
\quad \mbox{and} \quad 
u|_\Gamma = \varphi
\]
for a given $\varphi\in C(\Gamma)$. 
This is a classical and well-studied problem. 
The domain $\Omega$ is called {\bf Wiener regular} if for all 
$\varphi \in C(\Gamma)$ there exists a classical solution 
$u \in C(\overline{\Omega}) \cap C^2(\Omega)$.
While this is a weak regularity condition that can be elegantly characterised, 
it is not satisfied for arbitrary~$\Omega$. 

A generalised solution for the Dirichlet problem was introduced by 
Perron~\cite{Perron} 
and further developed in particular by Wiener~\cite{Wiener2},
\cite{Wiener}, \cite{Wiener3}, Brelot~\cite{Brelot} and 
Keldy\v{s} \cite{Keldys2}, \cite{Keldys}. 
For this Perron solution several different but equivalent descriptions 
were obtained.
Given $\varphi$, the Perron solution is unique and, if possible, 
classical, but in general the boundary condition $u|_\Gamma=\varphi$ 
will not be satisfied at every boundary point.
Perron's construction uses a supremum of subsolutions and infimum of 
supersolutions, and critically depends on the maximum principle 
(see \cite{GT} Section~2.8).

In this paper we define and study a generalised solution for the Dirichlet 
problem with continuous boundary data $\varphi$ on an arbitrary bounded, 
open $\Omega$ and for a general second-order elliptic operator with 
lower-order coefficients. 
We only assume a straightforward spectral condition. 
As the maximum principle does not need to hold for such an operator, 
Perron's construction cannot be used.
In~\cite{AE9} the authors considered the aforementioned operators, 
but only on Wiener regular domains, where it was possible to obtain a
classical solution, that is a harmonic $u \in C(\overline \Omega)$
with $u|_\Gamma = \varphi$.

We give four equivalent descriptions of our generalised solution. 
Since these generalise different characterisations of the Perron solution 
in a more classical setting, it is justified to continue to use the 
name Perron solution. 
Our results extend work by Stampacchia~\cite{Stam2}, 
since we are not dependent on a maximum principle and therefore do not 
need a divergence condition on the lower-order coefficients. 
The four equivalent descriptions are based on the following 
concepts, for simplicity expressed here for the Laplacian:
\begin{enumerate}
\item 
A continuous solution operator $T \colon C(\Gamma) \to C_b(\Omega)$ via $H^1$-extensions and an $L_\infty$-estimate.
\item 
Domain approximation from the interior.
\item 
Proper limit behaviour at regular boundary points. 
\item 
Existence of a function $\Phi \in C(\overline{\Omega})$ 
with $\Phi|_\Gamma=\varphi$ 
such that $\Delta \Phi \in (H^1_0(\Omega))'$. 
Then the Perron solution is the unique harmonic function $u$ on $\Omega$
such that $u - \Phi \in H^1_0(\Omega)$.
\end{enumerate}
The fourth description is completely new, 
even in the special case of the Laplacian, where it solves an open problem. 

Moreover, we characterise those $\varphi$ for which one obtains solutions 
with finite energy, i.e.\ when $u \in H^1(\Omega)$, which turns out 
to be independent of the chosen elliptic operator.
Finally, we investigate under which conditions it suffices to consider 
a trace in $H^1(\Omega)$ with respect to the $(d-1)$-dimensional 
Hausdorff measure on $\Gamma$ in order to uniquely determine solutions.

In order to precisely state the main theorems of this paper, we have to introduce some notation,
which we keep throughout the paper.
Let $\Omega \subset \Ri^d$ be open and bounded with boundary $\Gamma = \partial \Omega$ and $d \geq 2$.
For all $k,l \in \{ 1,\ldots,d \} $ let $a_{kl} \colon \Omega \to \Ri$ 
be a bounded measurable function and suppose that there exists a
$\mu \in (0,1]$ such that 
\[
\RRe \sum_{k,l=1}^d a_{kl}(x) \, \xi_k \, \overline{\xi_l} 
\geq \mu \, |\xi|^2
\]
for all $x \in \Omega$ and $\xi \in \Ci^d$.
Further, for all $k \in \{ 1,\ldots,d \} $ let $b_k,c_0 \colon \Omega \to \Ci$
and $c_k \colon \Omega \to \Ri$
be bounded and measurable.
Define the map $\ca \colon H^1_\loc(\Omega) \to \cd'(\Omega)$ by
\[
\langle \ca u,v \rangle_{\cd'(\Omega) \times \cd(\Omega)}
= \sum_{k,l=1}^d \int_\Omega a_{kl} \, (\partial_l u) \, \overline{\partial_k v}
   + \sum_{k=1}^d \int_\Omega b_k \, u \, \overline{\partial_k v}
   + \sum_{k=1}^d \int_\Omega c_k \, (\partial_k u) \, \overline v
   + \int_\Omega c_0 \, u \, \overline v
\]
for all $u \in H^1_\loc(\Omega)$ and $v \in C_c^\infty(\Omega)$.
Here we denote by $\cd'(\Omega)$ the antidual of 
$\cd(\Omega) = C_c^\infty(\Omega)$, the space of all test functions.
A function $u \colon \Omega \to \Ci$ is called {\bf $\ca$-harmonic}
if $u \in H^1_\loc(\Omega)$ and $\ca u = 0$.
Throughout this paper we assume that {\bf $0$ is not a Dirichlet eigenvalue,
that is, if $u \in H^1_0(\Omega)$ is such that $\ca u = 0$, then $u = 0$.}
This completes the list of conditions that we assume throughout the paper.
Recall that the set of all Dirichlet eigenvalues is discrete.
Stampacchia's divergence condition implies that all Dirichlet eigenvalues
lie in the open right half-plane, whereas under our assumption there may be 
Dirichlet eigenvalues in the left half-plane.
Obviously the latter does occur for certain choices of the coefficients. 
Moreover, the above assumption is natural and in some sense best 
possible since otherwise solutions can be perturbed with a non-trivial 
$\ca$-harmonic element of $H^1_0(\Omega)$ and there is no hope for 
unique solvability.

Given $\varphi \colon \Gamma \to \Ci$ we study the Dirichlet problem
\[
(D_\varphi) 
\quad
\mbox{Find an $\ca$-harmonic function $u$ such that ``$u|_\Gamma = \varphi$'',}
\]
where we will give different meanings to ``$u|_\Gamma = \varphi$''.
The investigation of how solutions may attain the given boundary data
$\varphi$ is the purpose of this paper.

Given a function $\varphi \in C(\Gamma)$,
a {\bf classical solution of the Dirichlet problem $(D_\varphi)$}
is an $\ca$-harmonic $u \in C(\overline \Omega) \cap H^1_\loc(\Omega)$
such that $u|_\Gamma = \varphi$.
We emphasise that a classical solution is defined on $\overline \Omega$.

Define 
\[
\ch_\ca^1(\Omega)
= \{ u \in H^1(\Omega) : \ca u = 0 \}
,  \]
the space of all $\ca$-harmonic functions in $H^1(\Omega)$.
Note that $\ch^1_\ca(\Omega)$ is a closed subspace in $H^1(\Omega)$.
Then 
\begin{equation}
H^1(\Omega) 
= H^1_0(\Omega) \oplus \ch_\ca^1(\Omega)
,
\label{eSdat1;20}
\end{equation}
see Lemma~\ref{ldat260}\ref{ldat260-2}.
As in Stampacchia \cite{Stam2} we use (\ref{eSdat1;20}) to define 
a weak solution.
The difficulty in this paper is to show an $L_\infty$-estimate
for this solution without Stampacchia's divergence condition
\cite{Stam2}~(9.2).

The first main theorem of this paper is as follows.

\begin{thm} \label{tdat120}
There exists a unique bounded linear map $T \colon C(\Gamma) \to C_b(\Omega)$
such that the function $T \varphi$ is $\ca$-harmonic for all $\varphi \in C(\Gamma)$
and $\Phi|_\Omega - T(\Phi|_\Gamma) \in H^1_0(\Omega)$ for all 
$\Phi \in C(\overline \Omega) \cap H^1(\Omega)$.
\end{thm}

We prove Theorem~\ref{tdat120} as a special case of Theorem~\ref{tdat235}.
We write $T^\Omega$ instead of $T$ if we want to make precise the 
underlying domain.
Given $\varphi \in C(\Gamma)$ we call $T \varphi$ the 
{\bf Perron solution of the Dirichlet problem $(D_\varphi)$}.
Note that $T \varphi$ is a function on $\Omega$, not~on~$\overline \Omega$.

Next we describe several properties of the operator $T$ which justify our terminology.
If $\Omega_1,\Omega_2,\ldots \subset \Ri^d$ are open and bounded, then 
we write $\Omega_n \uparrow \Omega$ if $\Omega_n \subset \Omega$ for all $n \in \Ni$
and for all compact $K \subset \Omega$ there exists an $N \in \Ni$ such that 
$K \subset \Omega_n$ for all $n \in \Ni$ with $n \geq N$.

\begin{thm} \label{tdat825}
\begin{tabel}
\item \label{tdat825-3}
Let $\Phi \in C(\overline \Omega)$ and $\Omega_1,\Omega_2,\ldots \subset \Ri^d$ 
be open and bounded with $\Omega_n \uparrow \Omega$.
Then 
\[
\lim_{n \to \infty} T^{\Omega_n} (\Phi|_{\partial \Omega_n}) = T(\Phi|_\Gamma)
\]
uniformly on compact subsets of $\Omega$.
\item \label{tdat825-4}
Let $u \in C(\overline \Omega) \cap H^1_\loc(\Omega)$ and suppose that $\ca u = 0$.
Then $T(u|_\Gamma) = u|_\Omega$.
\end{tabel}
\end{thm}

This theorem is proved in Propositions~\ref{pdat807} and \ref{pdat808}.
Theorem~\ref{tdat825} allows us to give an alternative description of~$T$.
First suppose that $\Omega$ is Wiener regular.
Then by \cite{AE9} Theorem~1.1 (or Corollary~\ref{cdat320} in this paper)
for every $\varphi \in C(\Gamma)$ there exists a unique classical solution
$u \in C(\overline \Omega) \cap H^1_\loc(\Omega)$ of the Dirichlet problem $(D_\varphi)$.
Theorem~\ref{tdat825}\ref{tdat825-4} shows that $T \varphi = u|_\Omega$.
In particular, classical solutions are unique, which is 
not obvious without the maximum principle.
Now suppose that $\Omega$ is arbitrary.
Then we may choose Wiener regular
(or even $C^\infty$-domains) $\Omega_1,\Omega_2,\ldots$ such that $\Omega_n \uparrow \Omega$.
Given $\varphi \in C(\Gamma)$, take as $\Phi \in C(\overline \Omega)$
an extension of $\varphi$ and define $\varphi_n = \Phi|_{\partial \Omega_n}$
for all $n \in \Ni$.
Then Theorem~\ref{tdat825}\ref{tdat825-3} shows that 
\[
T \varphi 
= \lim_{n \to \infty} T^{\Omega_n} \varphi_n
\]
uniformly on compact subsets of $\Omega$.
For the Laplacian, that is if $a_{kl} = \delta_{kl}$ and 
$b_k = c_k = c_0 = 0$, Keldy\v{s} \cite{Keldys} Theorem~II showed that 
this limit coincides with the solution defined by Perron, justifying
our terminology.
Herv\'e \cite{Herve1} generalised this result to 
purely second-order differential operators
(which trivially verify Stampacchia's divergence condition).

\medskip

We next give a `Sobolev'-description of the Perron solution, which is 
new even for the Laplacian.
By $H^{-1}(\Omega)$ we denote the antidual of $H^1_0(\Omega)$.
Since $0$ is not a Dirichlet eigenvalue by assumption, it follows
that $\ca|_{H^1_0(\Omega)} \colon H^1_0(\Omega) \to H^{-1}(\Omega)$ is bijective (see Lemma~\ref{ldat260}\ref{ldat260-1}).
In general, a function $\varphi \in C(\Gamma)$ does not 
have an extension $\Phi \in C(\overline \Omega) \cap H^1(\Omega)$
(just take $\Omega$ Lipschitz and $\varphi \in C(\Gamma) \setminus H^{1/2}(\Gamma)$).
But the following extension is possible.

\begin{thm} \label{tdat121}
Let $\varphi \in C(\Gamma)$.
Then there exists a $\Phi \in C(\overline \Omega) \cap H^1_\loc(\Omega)$ such that 
$\Phi|_\Gamma = \varphi$ and $\ca \Phi \in H^{-1}(\Omega)$.
Hence there exists a unique $v \in H^1_0(\Omega)$ such that 
$\ca v = \ca \Phi$.
It follows that $T \varphi = \Phi|_\Omega - v$.
\end{thm}

For the Laplacian it was proved 
before in \cite{AD2} that $T(\varphi) = \Phi|_\Omega - v$ whenever $\varphi \in C(\Gamma)$, 
$\Phi \in C(\overline \Omega)$,
$\Delta \Phi \in H^{-1}(\Omega)$ as distribution, $\Phi|_\Gamma = \varphi$
and $v \in H^1_0(\Omega)$ is such that 
$\Delta v = \Delta \Phi$.
But it remained an open question whether such an extension $\Phi$ exists for all 
$\varphi \in C(\Gamma)$.
The above theorem gives a positive answer even for arbitrary elliptic operators.
As a consequence, for the Laplacian the Perron solution is always an
$H^1_0(\Omega)$-perturbation of a suitable extension 
$\Phi \in C(\overline \Omega)$ of $\varphi$.
(Remark. The extra condition $\Phi \in H^1_\loc(\Omega)$ in Theorem~\ref{tdat121}
is needed if merely $a_{kl}, b_k \in L_\infty(\Omega)$.
If $a_{kl}, b_k \in W^{1,\infty}(\Omega)$, as is the case for the Laplacian, then this
extra condition is not needed. 
We come back to this in Proposition~\ref{pdat241}.)

As an immediate consequence of the last statement in 
Theorem~\ref{tdat121} one can characterise when 
$T \varphi$ has finite energy.
It is remarkable that this property does not depend on the elliptic operator.

\begin{cor} \label{cdat122}
Let $\varphi \in C(\Gamma)$.
Then $T \varphi \in H^1(\Omega)$ if and only if there exists a 
$\Phi \in C(\overline \Omega) \cap H^1(\Omega)$ such that $\Phi|_\Gamma = \varphi$.
\end{cor}

Next we turn to regular points as in the paper \cite{LSW}
of Littman, Stampacchia and Weinberger.
Let $z \in \Gamma$.
We say that $z$ is {\bf $\ca$-regular} if 
\[
\lim_{x \to z} (T\varphi)(x) = \varphi(z)
\]
for all $\varphi \in C(\Gamma)$.
We say that $\Omega$ is {\bf $\ca$-regular} if $z$ is $\ca$-regular for all $z \in \Gamma$.
Therefore $\Omega$ is $\ca$-regular if and only if for each $\varphi \in C(\Gamma)$ 
the Dirichlet problem $(D_\varphi)$ has a unique classical solution.
If $- \ca$ is the Laplacian, then we simply say that 
$z$ is {\bf regular} if $z$ is $(-\Delta)$-regular.

Extending Stampacchia's result to our more general situation, we shall show in 
Theorem~\ref{tdat831} that a point $z \in \Gamma$ is $\ca$-regular if and only if it is regular.
Hence this notion of regularity is independent of the operator.

Even if $\Omega$ is not Wiener regular, we can characterise the Perron 
solution by the pointwise behaviour at the boundary.
Recall that a subset $P \subset \Ri^d$ is called a {\bf polar set}
if $\capp P = 0$.
Given an assertion $Q(z)$ for all $z \in \Gamma$, then we say that 
$Q(z)$ holds for {\bf quasi every} $z \in \Gamma$ if there exists a polar set $P \subset \Gamma$ 
such that $Q(z)$ is true for all $z \in \Gamma \setminus P$.

\begin{thm} \label{tdat834.6}
Let $\varphi \in C(\Gamma)$ and $u \in C_b(\Omega) \cap H^1_\loc(\Omega)$.
Then the following are equivalent.
\begin{tabeleq}
\item \label{tdat834.6-1}
$u = T \varphi$.
\item \label{tdat834.6-2}
$\ca u = 0$ and $\lim_{x \to z} u(x) = \varphi(z)$ for quasi every $z \in \Gamma$.
\end{tabeleq}
\end{thm}

The implication \ref{tdat834.6-2}$\Rightarrow$\ref{tdat834.6-1} is well known 
for the Laplacian, but it seems to be new for elliptic operators, 
even in the situation considered by Stampacchia.
We prove Theorem~\ref{tdat834.6} in Corollary~\ref{cdat244}.

\smallskip

Finally we wish to consider the notion of approximative trace and 
variational solution of the Dirichlet problem.
For this we provide $\Gamma$ with the $(d-1)$-dimensional Hausdorff measure,
denoted by $\sigma$, and we assume for the remainder of the introduction
that $\sigma(\Gamma) < \infty$.
Variational solutions are based on the decomposition~(\ref{eSdat1;20}).
Let $w = v + u \in H^1_0(\Omega) \oplus \ch_\ca^1(\Omega)$.
If in addition $w \in C(\overline \Omega)$, then by our definition
$T \varphi = u$ is the Perron solution of the Dirichlet problem $(D_\varphi)$,
where $\varphi = w|_\Gamma$ is the restriction of $w$ to the boundary.
Now we want to replace the restriction by a very weak notion of a trace.
Given $u \in H^1(\Omega)$ and $\varphi \in L_2(\Gamma)$, then 
we say that $\varphi$ is an {\bf approximative trace of $u$}, or shortly,
a {\bf trace of $u$} if 
there are $u_1,u_2,\ldots \in H^1(\Omega) \cap C(\overline \Omega)$ 
such that $\lim_{n \to \infty} u_n = u$ in $H^1(\Omega)$ and 
$\lim_{n \to \infty} u_n|_\Gamma = \varphi$ in $L_2(\Gamma)$.
We define 
\[
\Tr(\Omega)
= \{ \varphi \in L_2(\Gamma) : \mbox{there exists a $u \in H^1(\Omega)$ such that 
$\varphi$ is a trace of $u$} \}
,  \]
the space of all traces.
Clearly $0$ is a trace of every $u \in H^1_0(\Omega)$.
For the remainder of this introduction we conversely suppose that 
if $0$ is a trace of $u \in H^1(\Omega)$, then $u \in H^1_0(\Omega)$, that is
\begin{equation}
H^1_0(\Omega)
= \{ u \in H^1(\Omega) : 
0 \mbox{ is a trace of } u \} 
.
\label{eSdat1;21}
\end{equation}
Condition~(\ref{eSdat1;21}) is satisfied if $\Omega$ has a continuous boundary 
in the sense of graphs (see \cite{Sau1} Corollary~7.48).

Now given $\varphi \in \Tr(\Omega)$, there exists a $w \in H^1(\Omega)$ such that 
$\varphi$ is a trace of $w$.
Decompose 
\[
w = v + u \in H^1_0(\Omega) \oplus \ch_\ca^1(\Omega)
.  \]
Since $v \in H^1_0(\Omega)$, the function $\varphi$ is also a trace of $u$.
In virtue of Condition~(\ref{eSdat1;21}), the function $u$ is the unique 
$\ca$-harmonic function in $H^1(\Omega)$ having $\varphi$ as a trace.
We call $u$ {\bf the variational solution of the Dirichlet problem $(D_\varphi)$}.
Therefore for each $\varphi \in \Tr(\Omega)$ there exists a unique variational
solution of the Dirichlet problem $(D_\varphi)$.
Even though the coupling between $w$ and a trace $\varphi$ is 
very weak, the link between variational and Perron solutions is 
as strong as it can be.
Our main result on the variational solution is the following.

\begin{thm} \label{tdat123}
Assume Condition~(\ref{eSdat1;21}).
Let $\varphi \in C(\Gamma)$.
\begin{tabel}
\item \label{tdat123-1}
Suppose in addition $\varphi \in \Tr(\Omega)$ and let $u \in H^1(\Omega)$ be the variational 
solution of the Dirichlet problem $(D_\varphi)$.
Then $T \varphi = u$.
\item \label{tdat123-2}
If $T \varphi \in H^1(\Omega)$, then $\varphi$ is a trace of $T \varphi$.
Consequently, $T \varphi$ is the variational 
solution of the Dirichlet problem $(D_\varphi)$ and $\varphi \in \Tr(\Omega)$.
\end{tabel}
\end{thm}

The proof of Theorem~\ref{tdat123} depends heavily on the previous results 
on possible extensions of $\varphi \in C(\Gamma)$ and it is given 
in Theorems~\ref{tdat910} and \ref{tdat206}.

We end this selection of highlights of the paper with a comment on the 
coefficients.

\begin{remark} \label{rdat117}
Throughout this paper we assume that the $c_k$ are real valued.
All results are also valid if the $c_k$ are complex valued with 
$\IIm c_k \in W^{1,\infty}(\Omega)$ for all $k \in \{ 1,\ldots,d \} $.
The reason is as follows.
Define the map $\widetilde \ca \colon H^1_\loc(\Omega) \to \cd'(\Omega)$ by
\begin{eqnarray*}
\langle \widetilde \ca u,v \rangle_{\cd'(\Omega) \times \cd(\Omega)}
& = & \sum_{k,l=1}^d \int_\Omega a_{kl} \, (\partial_l u) \, \overline{\partial_k v}
   + \sum_{k=1}^d \int_\Omega (b_k - i \IIm c_k) \, u \, \overline{\partial_k v}
\\* 
& & {}
   + \sum_{k=1}^d \int_\Omega (\RRe c_k) \, (\partial_k u) \, \overline v
   + \int_\Omega (c_0 - i \sum_{k=1}^d \partial_k \IIm c_k) \, u \, \overline v
\end{eqnarray*}
for all $u \in H^1_\loc(\Omega)$ and $v \in C_c^\infty(\Omega)$.
Since 
\[
\int_\Omega (\IIm c_k) (\partial_k u) \, \overline v
= - \int_\Omega (\IIm c_k) \, u \, \overline{ \partial_k v }
   - \int_\Omega (\partial_k \IIm c_k)  \, u \, \overline v
\]
for all $u \in H^1_\loc(\Omega)$, $v \in C_c^\infty(\Omega)$ and $k \in \{ 1,\ldots,d \} $,
it follows that $\ca = \widetilde \ca$. 
Now the coefficients of $\widetilde \ca$ satisfy the conditions on the 
coefficients in the paper.

We do not know whether one can allow $c_k \in L_\infty(\Omega)$ complex 
valued throughout the paper.
\end{remark}

The structure of this paper is as follows. 
In Section~\ref{Sdat2} we consider the Perron solution and prove 
Theorems~\ref{tdat120} and \ref{tdat825}.
A few parts are based on techniques developed in \cite{AE9}, but here more
delicate arguments are needed to control the constants, since in 
Theorem~\ref{tdat825}\ref{tdat825-3} a sequence of 
domains is used.
In Section~\ref{Sdat3} we consider the $\ca$-regular points
and show that they are independent of the elliptic operator.
We also give in Theorem~\ref{tdat243} a sufficient condition
(in terms of the pointwise behaviour at the boundary)
for a function to be in $H^1_0(\Omega)$.
As one of the many consequences this gives 
Theorem~\ref{tdat834.6} in Corollary~\ref{cdat244}.
In Section~\ref{Sdat4} we prove Theorem~\ref{tdat121} via a delicate
approximation procedure.
Finally in Section~\ref{Sdat5} we consider the approximative trace, 
the variational solution of the Dirichlet problem, its possible uniqueness
(that is a characterisation of Condition (\ref{eSdat1;21}))
and prove Theorem~\ref{tdat123}.

\section{The Dirichlet problem with continuous data} \label{Sdat2}

In this section we prove Theorems~\ref{tdat120} and \ref{tdat825}.
The proof requires some preparation.
Throughout this section we adopt the notation and assumptions as in 
the first part of the introduction.
We first consider $\varphi \in C(\Gamma)$ 
of the form $\varphi = \Phi|_\Gamma$ with $\Phi \in C(\overline \Omega) \cap H^1(\Omega)$.
Then we define an operator 
$T_0 \colon \{ \Phi|_\Gamma : \Phi \in C(\overline \Omega) \cap H^1(\Omega) \} 
  \to C_b(\Omega) \cap H^1(\Omega)$
with range in the $\ca$-harmonic functions
and show that it is bounded with respect to the $C(\Gamma)$-norm and the 
$C_b(\Omega)$-norm.
Next the Stone--Weierstra\ss\ theorem gives a continuous extension 
$T \colon C(\Gamma) \to C_b(\Omega)$.
We use a technique of 
Stampacchia in order to show that $T \varphi$ is $\ca$-harmonic. 
That gives Theorem~\ref{tdat120}.

Define the form $\gota \colon H^1(\Omega) \times H^1(\Omega) \to \Ci$ by
\begin{equation}
\gota(u,v) 
= \sum_{k,l=1}^d \int_\Omega a_{kl} \, (\partial_l u) \, \overline{\partial_k v}
   + \sum_{k=1}^d \int_\Omega b_k \, u \, \overline{\partial_k v}
   + \sum_{k=1}^d \int_\Omega c_k \, (\partial_k u) \, \overline v
   + \int_\Omega c_0 \, u \, \overline v
.  
\label{eSdat2;31}
\end{equation}
Let $A^D$ be the operator in $L_2(\Omega)$ associated with the 
form $\gota|_{H^1_0(\Omega) \times H^1_0(\Omega)}$, 
that is 
\begin{eqnarray*}
\graph(A^D) 
= \{ (u,f) \in L_2(\Omega) \times L_2(\Omega) 
& : & u \in H^1_0(\Omega) \mbox{ and }  \\*
& &    \gota(u,v) = (f, v)_{L_2(\Omega)} \mbox{ for all } v \in H^1_0(\Omega) \} 
.
\end{eqnarray*}
In other words, $A^D$ is the realisation of the elliptic operator $\ca$ 
in $L_2(\Omega)$ with Dirichlet boundary conditions.
This operator has a compact resolvent.
Moreover, $A^D$ is injective by our assumption that $0$ is not a Dirichlet 
eigenvalue.

We will frequently use the following two lemmas.
We identify $H^{-1}(\Omega)$ with the subspace of $\cd'(\Omega)$ consisting 
of all distributions $F \in \cd'(\Omega)$ for which there exists a $c > 0$ such that 
$|\langle F, v \rangle_{\cd'(\Omega) \times \cd(\Omega)}| \leq c \, \|v\|_{H^1(\Omega)}$
for all $v \in \cd(\Omega)$.

\begin{lemma} \label{ldat260}
\mbox{}
\begin{tabel}
\item  \label{ldat260-1}
$\ca|_{H^1_0(\Omega)}$ is an isomorphism from $H^1_0(\Omega)$ onto $H^{-1}(\Omega)$.
\item \label{ldat260-2}
$H^1(\Omega) = H^1_0(\Omega) \oplus \ch_\ca^1(\Omega)$, where
$\oplus$ means the direct sum of vector spaces.
\item \label{ldat260-3}
If $u \in H^1(\Omega)$, then $\ca u \in H^{-1}(\Omega)$.
\end{tabel}
\end{lemma}
\begin{proof}
`\ref{ldat260-1}'.
Since the form $\gota$ is continuous 
there exists a unique continuous linear $Z \colon H^1_0(\Omega) \to H^{-1}(\Omega)$ such that 
$\langle Z u,v\rangle_{H^{-1}(\Omega) \times H^1_0(\Omega)} = \gota(u,v)$
for all $u,v \in H^1_0(\Omega)$.
Then $Z$ is injective because $\ker A^D = \{ 0 \} $.
Moreover, the inclusion $H^1_0(\Omega) \hookrightarrow L_2(\Omega)$ is compact.
Hence the operator $Z$ is invertible by the Fredholm--Lax--Milgram lemma in
\cite{AEKS} Lemma~4.1.

`\ref{ldat260-2}'.
Let $\Phi \in H^1(\Omega)$.
Clearly the map $w \mapsto \gota(\Phi,w)$ is continuous from 
$H^1_0(\Omega)$ into $\Ci$.
Hence there exists a unique $v \in H^1_0(\Omega)$ such that 
$\gota(v,w) = \gota(\Phi,w)$ for all $w \in H^1_0(\Omega)$.
In particular $\Phi - v \in H^1(\Omega)$ and $\ca(\Phi - v) = 0$.

`\ref{ldat260-3}'.
Define $f_0 = c_0 \, u + \sum_{k=1}^d c_k \, \partial_k u \in L_2(\Omega)$
and $f_k = - b_k \, u - \sum_{l=1}^d a_{kl} \, \partial_l u \in L_2(\Omega)$.
Then $\ca u = f_0 + \sum_{k=1}^d \partial_k f_k \in H^{-1}(\Omega)$.
\end{proof}

If $p \in (1,\infty)$, then define $W^{-1,p}(\Omega) = ( W^{1,q}_0(\Omega) )'$,
where $q$ is the dual exponent of $p$.
Note that $W^{-1,p}(\Omega) \subset \cd'(\Omega)$.
Moreover, $H^{-1}(\Omega) = W^{-1,2}(\Omega)$ if $p = 2$.

\begin{lemma} \label{ldat202.3}
Let $p \in (1,\infty)$ and let $q$ be the dual exponent of $p$.
We define the norm $|||\cdot||| \colon W^{1,q}_0(\Omega) \to [0,\infty)$ by 
$|||u||| = \| \, |\nabla u| \, \|_{L_q(\Omega)}$ and provide 
$W^{-1,p}(\Omega)$ with the norm of $( W^{1,q}_0(\Omega), |||\cdot||| )'$.
\begin{tabel}
\item \label{ldat202.3-1}
If $F \in W^{-1,p}(\Omega)$, then there exist
$f_1,\ldots,f_d \in L_p(\Omega)$ such that 
$F(v) = \sum_{k=1}^d \int_\Omega f_k \, \overline{\partial_k v}$
for all $v \in W^{1,q}_0(\Omega)$ and 
\[
\Big\| \Big( \sum_{k=1}^d |f_k|^2 \Big)^{1/2} \Big\|_{L_p(\Omega)} 
\leq 2 \|F\|_{W^{-1,p}(\Omega)}
.  \]
So $F = - \sum_{k=1}^d \partial_k f_k$ as distribution.
\item \label{ldat202.3-2}
Let $f_0,f_1,\ldots,f_d \in L_p(\Omega)$.
Define $F \colon W^{1,q}_0(\Omega) \to \Ci$ by 
\[
F(v) = \int_\Omega f_0 \, \overline v + \sum_{k=1}^d \int_\Omega f_k \, \overline{\partial_k v}
.  \]
Then $F \in W^{-1,p}(\Omega)$ and 
\[
\|F\|_{W^{-1,p}(\Omega)}
\leq (\diam \Omega) \, \|f_0\|_{L_p(\Omega)}
    + \Big\| \Big( \sum_{k=1}^d |f_k|^2 \Big)^{1/2} \Big\|_{L_p(\Omega)} 
.  \]
\end{tabel}
\end{lemma}
\begin{proof}
`\ref{ldat202.3-1}'.
This follows from \cite{MazED2} Theorem~1.1.15.1.

`\ref{ldat202.3-2}'.
The Dirichlet-type Poincar\'e inequality, see for example \cite{EdE} Theorem~V.3.22,
gives that $\|v\|_q \leq (\diam \Omega) \, |||v|||$ for all $v \in W^{1,q}_0(\Omega)$.
Then the statement follows from the H\"older inequality.
\end{proof}

As in \cite{AE9} we define for all $\lambda \in \Ri$ the forms 
$\gota_\lambda, \gotb_\lambda \colon H^1(\Omega) \times H^1(\Omega) \to \Ci$
by 
\begin{eqnarray*}
\gota_\lambda(u,v) 
& = & \gota(u,v) + \lambda \, (u,v)_{L_2(\Omega)}  \quad \mbox{and}  \\
\gotb_\lambda(u,v) 
& = & \sum_{k,l=1}^d \int_\Omega a_{kl} \, (\partial_l u) \, \overline{\partial_k v}
   + \sum_{k=1}^d \int_\Omega c_k \, (\partial_k u) \, \overline v
   + \lambda \int_\Omega u \, \overline v
,
\end{eqnarray*}
where $\gota$ is as in (\ref{eSdat2;31}).
Define similarly $\ca_\lambda,\cb_\lambda \colon H^1_\loc(\Omega) \to \cd'(\Omega)$
and denote by $B^D$ the operator in $L_2(\Omega)$ associated with the sesquilinear form
$\gotb_0|_{H^1_0(\Omega) \times H^1_0(\Omega)}$.
Ellipticity implies that 
\[
\frac{\mu}{2} \, \|v\|_{H^1(\Omega)}^2
\leq \RRe \gota_{\lambda_0}(v)
\quad \mbox{and} \quad
\frac{\mu}{2} \, \|v\|_{H^1(\Omega)}^2
\leq \RRe \gotb_{\lambda_0}(v)
\]
for all $v \in H^1(\Omega)$, where 
$\lambda_0 = 1 + \|c_0\|_\infty + d \mu^{-1}(\sum_{k=1}^d \|b_k\|_\infty^2 + \|c_k\|_\infty^2)$.

In the next proposition we use that the $c_k$ are real valued.

\begin{prop} \label{pdat263}
Let $\lambda > \lambda_0$ and $p \in (d,\infty]$.
Let $f_0,f_1,\ldots,f_d \in L_p(\Omega)$.
Then there exists a unique $u \in H^1_0(\Omega)$
such that $\cb_\lambda u = f_0 + \sum_{k=1}^d \partial_k f_k$.
This solution $u$ satisfies $u \in C_b(\Omega)$.
Explicitly, there exists a $C > 0$, depending only on the ellipticity constant,
the $L_\infty$-norm of the coefficients, the diameter of $\Omega$ and $\lambda$,
such that 
\begin{equation}
\|u\|_{C_b(\Omega)} \leq C \sum_{k=0}^d \|f_k\|_{L_p(\Omega)}
.  
\label{epdat263;1}
\end{equation}
\end{prop}
\begin{proof}
Since $a_{kl}$ and $c_k$ are real valued for all $k,l \in \{ 1,\ldots,d \} $,
we may assume that $f_0,\ldots,f_d$ are real valued.
The existence and uniqueness follows by 
Lemmas~\ref{ldat202.3}\ref{ldat202.3-2} and \ref{ldat260}\ref{ldat260-1}.
The Nash--De Giorgi theory implies that $u$ is locally H\"older continuous.
The boundedness of the function $u$ and the bound (\ref{epdat263;1})
is proved by Stampacchia \cite{Stam2} Th\'eor\`eme~4.2
if $f_0 = 0$. 

It remains to show how to deal with $f_0$.
We may assume that $p < \infty$.
Let $q$ be the dual exponent of~$p$.
Define $F \in W^{-1,p}(\Omega)$ by $F(v) = \int_\Omega f_0 \, \overline v$.
Then $\|F\|_{W^{-1,p}(\Omega)} \leq (\diam \Omega) \, \|f_0\|_{L_p(\Omega)}$
by Lemma~\ref{ldat202.3}\ref{ldat202.3-2}.
By Lemma~\ref{ldat202.3}\ref{ldat202.3-1} there are $g_1,\ldots,g_d \in L_p(\Omega)$
such that $F = - \sum_{k=1}^d \partial_k g_k$ and 
\[
\Big\| \Big( \sum_{k=1}^d |g_k|^2 \Big)^{1/2} \Big\|_{L_p(\Omega)} 
\leq 2 \|F\|_{W^{-1,p}(\Omega)}
\leq 2 (\diam \Omega) \, \|f_0\|_{L_p(\Omega)}
.  \]
Therefore one can absorb $f_0$ in $f_1,\ldots,f_d$.
\end{proof}

We next wish to solve the Dirichlet problem for~$\ca_\lambda$ if $\lambda$ is large enough.
We need a technical lemma as in \cite{AE9} Lemma~2.7. 
Since we rely on an approximation argument in the proof of Proposition~\ref{pdat808}, 
we need control of the constants, which was not required in \cite{AE9}.
In order to trace the constants we give the full details.

\begin{lemma} \label{ldat207}
Fix $\delta \in [0,\lambda_0 + 1]$.
\mbox{}
\begin{tabel}
\item \label{ldat207-1}
For all $f \in L_2(\Omega)$ and $\lambda > \lambda_0$ 
there exists a unique $u \in H^1_0(\Omega)$ such that
\begin{equation}
\gotb_\lambda(u,v) 
= \sum_{k=1}^d (b_k \, f, \partial_k v)_{L_2(\Omega)} 
   + \, ((c_0 - \delta \, \one_\Omega) \, f, v)_{L_2(\Omega)}
\label{eldat207;1}
\end{equation}
for all $v \in H^1_0(\Omega)$.
\end{tabel}
For all $\lambda > \lambda_0$ define 
$R_\lambda \colon L_2(\Omega) \to L_2(\Omega)$ by $R_\lambda f = u$, where 
$u \in H^1_0(\Omega)$ is as in~{\rm (\ref{eldat207;1})}.
\begin{tabel}
\setcounter{teller}{1}
\item \label{ldat207-1.5}
There exists a $C_1 \geq 1$ such that 
\[
\|R_\lambda f\|_{L_q(\Omega)} \leq C_1 \, (\lambda - \lambda_0)^{-1/4} \, \|f\|_{L_2(\Omega)}
\]
for all $\lambda > \lambda_0$ and $f \in L_2(\Omega)$,
where $\frac{1}{q} = \frac{1}{2} - \frac{1}{4d}$.
The constant $C_1$ can be chosen to be independent of $\Omega$ and to depend only on 
the ellipticity constant and the $L_\infty$-norm of the coefficients.
\item \label{ldat207-2}
There exists a $C_2 \geq 1$ such that 
\[
\|R_\lambda f\|_{L_q(\Omega)} \leq C_2 \, \|f\|_{L_p(\Omega)}
\]
for all $\lambda \in [\lambda_0 + 1,\infty)$,
$p,q \in [2,\infty]$ and $f \in L_p(\Omega)$ with $\frac{1}{q} = \frac{1}{p} - \frac{1}{4d}$.
The constant $C_2$ can be chosen to depend only on the
ellipticity constant, the $L_\infty$-norm of the coefficients and the diameter of $\Omega$.
\item \label{ldat207-3}
If $\lambda > \lambda_0$, 
$p \in (d,\infty]$ and $f \in L_p(\Omega)$, then $R_\lambda f \in C_b(\Omega)$.
\end{tabel}
\end{lemma}
\begin{proof}
`\ref{ldat207-1}'.
Apply the Lax--Milgram theorem.

`\ref{ldat207-1.5}'.
By the Sobolev embedding theorem there exists a $C_3 > 0$ such that 
$\|v\|_{L_{q_1}(\Ri^d)} \leq C_3 \, \|v\|_{H^1(\Ri^d)}$ for all $v \in H^1_0(\Omega)$,
where $\frac{1}{q_1} = \frac{1}{2} - \frac{1}{2d}$.
(The extra factor $2$ is to avoid a separate case for $d=2$.)
Define $M = \|c_0 - \delta \, \one_\Omega\|_{L_\infty(\Omega)} + \sum_{k=1}^d \|b_k\|_{L_\infty(\Omega)}$.
Let $\lambda > \lambda_0$, $f \in L_2(\Omega)$ and set $u = R_\lambda f$.
Then 
\begin{eqnarray*}
\frac{\mu}{2} \, \|u\|_{H^1(\Omega)}^2 + (\lambda - \lambda_0) \|u\|_{L_2(\Omega)}^2
& \leq & \RRe \gotb_{\lambda_0}(u) + (\lambda - \lambda_0) \|u\|_{L_2(\Omega)}^2  \\
& = & \RRe \gotb_\lambda(u)  \\
& = & \RRe \sum_{k=1}^d (b_k \, f, \partial_k u)_{L_2(\Omega)} 
   + \RRe ((c_0 - \delta \, \one_\Omega) \, f, u)_{L_2(\Omega)}  \\
& \leq & M \, \|f\|_{L_2(\Omega)} \, \|u\|_{H^1(\Omega)}
.  
\end{eqnarray*}
Therefore
$\|u\|_{H^1(\Omega)} 
\leq 2 \mu^{-1} \, M \, \|f\|_{L_2(\Omega)}$
and 
\[
\|R_\lambda f\|_{L_2(\Omega)}
= \|u\|_{L_2(\Omega)}
\leq \sqrt{\frac{2}{\mu (\lambda - \lambda_0)} } \, M \, \|f\|_{L_2(\Omega)}
.  \]
The Sobolev embedding gives
$\|R_\lambda f\|_{L_{q_1}(\Omega)} 
\leq 2 \mu^{-1} \, C_3 \, M \, \|f\|_{L_2(\Omega)}$.
Consequently 
\[
\|R_\lambda f\|_{L_q(\Omega)} 
\leq \|R_\lambda f\|_{L_2(\Omega)}^{1/2} \, \|R_\lambda f\|_{L_{q_1}(\Omega)}^{1/2}
\leq C_4 \, (\lambda - \lambda_0)^{-1/4} \, \|f\|_{L_2(\Omega)}
,  \]
where $C_4 = (2/\mu)^{3/4} \, C_3^{1/2} \, M$.

`\ref{ldat207-2}'.
Apply Proposition~\ref{pdat263} with $p = 4d$ and $\lambda = \lambda_0 + 1$.
It follows that $R_{\lambda_0 + 1} f \in C_b(\Omega)$ for all $f \in L_p(\Omega)$.
Moreover, there exists a $C_5 > 0$ such that 
$\|R_{\lambda_0 + 1} f\|_{L_\infty(\Omega)} 
= \|R_{\lambda_0 + 1} f\|_{C_b(\Omega)} 
\leq C_5 \, \|f\|_{L_p(\Omega)}$
for all $f \in L_p(\Omega)$.
The constant $C_5$ can be chosen to depend only on the
ellipticity constant,
the $L_\infty$-norm of the coefficients and the diameter of $\Omega$.

Let $\lambda \geq \lambda_0 + 1$ and $f \in L_2(\Omega)$.
Set $u = R_\lambda f$ and $u_0 = R_{\lambda_0 + 1} f$.
By definition $\gotb_\lambda(u,v) = \gotb_{\lambda_0 + 1}(u_0,v)$
and hence 
$\gotb_\lambda(u - u_0, v) = - (\lambda - \lambda_0 - 1) \, (u_0,v)_{L_2(\Omega)}$
for all $v \in H^1_0(\Omega)$.
Then $u - u_0 \in D(B^D)$ and $(B^D + \lambda \, I) (u - u_0) = - (\lambda - \lambda_0 - 1) \, u_0$.
In particular 
\[
R_\lambda 
= \Big( I - (\lambda - \lambda_0 - 1) \, (B^D + \lambda \, I)^{-1} \Big) R_{\lambda_0 + 1}
\]
for all $\lambda \geq \lambda_0 + 1$.
Since the semigroup generated by $-B^D$ is dominated, using \cite{Ouh5} Proposition~4.23,
by a suitable semigroup on $\Ri^d$ and the latter has Gaussian bounds, there exists a
$C_6 \geq 1$, depending only on the ellipticity constant and 
the $L_\infty$-norm of the coefficients, such that 
$\|(B^D + \lambda \, I)^{-1}\|_{L_\infty(\Omega) \to L_\infty(\Omega)} \leq C_6 \, \lambda^{-1}$ 
for all $\lambda \geq \lambda_0 + 1$.
Then $\|R_\lambda f\|_{L_\infty(\Omega)} \leq 2 C_5 \, C_6 \, \|f\|_{L_p(\Omega)}$
for all $\lambda \geq \lambda_0 + 1$ and $f \in L_p(\Omega)$.
We proved that $R_\lambda$ is bounded from $L_p(\Omega)$ into $L_\infty(\Omega)$.

Finally let $p' \in (2,4d)$ and let $q' \in (2,\infty)$ be such that 
$\frac{1}{q'} = \frac{1}{p'} - \frac{1}{4d}$.
There exists a $\theta \in (0,1)$ such that 
$\frac{1}{p'} = \frac{1-\theta}{2} + \frac{\theta}{p}$.
Then $\frac{1}{q'} = \frac{1-\theta}{q}$, where 
$\frac{1}{q} = \frac{1}{2} - \frac{1}{4d}$.
Let $C_1 > 0$ be as in Statement~\ref{ldat207-1.5}.
Interpolation between the $L_2 \to L_q$ bounds and the $L_p \to L_\infty$ bounds 
gives that $R_\lambda$ is bounded from 
$L_{p'}(\Omega)$ into $L_{q'}(\Omega)$ with norm bounded by 
$C_1^{1-\theta} \, (2 C_5 \, C_6)^\theta \leq C_1 + 2 C_5 \, C_6$,
which gives Statement~\ref{ldat207-2}.

`\ref{ldat207-3}'.
This follows from Proposition~\ref{pdat263}.
\end{proof}

Now we can solve the Dirichlet problem for~$\ca_\lambda$ if $\lambda$ is large enough
as in \cite{AE9} Lemma~2.8.

\begin{lemma} \label{ldat280}
There exist $\lambda > \lambda_0$ and $C > 0$ such that 
the following is valid.
\begin{tabel}
\item \label{ldat280-1}
For all $\Phi \in C(\overline \Omega) \cap H^1(\Omega)$ there exists a unique
$u \in C_b(\Omega) \cap H^1(\Omega)$ such that 
$\ca_\lambda u = 0$ and $\Phi - u \in H^1_0(\Omega)$.
\item \label{ldat280-2}
If $\Phi \in C(\overline \Omega) \cap H^1(\Omega)$, then
\[
\|u\|_{C_b(\Omega)} \leq C \, \|\Phi|_\Gamma\|_{C(\Gamma)}
,  \]
where $u \in C_b(\Omega) \cap H^1(\Omega)$ is such that 
$\ca_\lambda u = 0$ and $\Phi - u \in H^1_0(\Omega)$.
\end{tabel}
The constants $\lambda$ and $C$ can be chosen to depend only on the
ellipticity constant, the $L_\infty$-norm of the coefficients and the diameter of $\Omega$.
\end{lemma}
\begin{proof}
Choose $\delta = 0$ in Lemma~\ref{ldat207}.
Let $C_1$ and $C_2$ be as in Lemma~\ref{ldat207}.
Let $\lambda \in (\lambda_0 + 1,\infty)$ be such that 
$C_1 \, C_2^{2d-1} \, (\lambda - \lambda_0)^{-1/4} \, (1 + |\Omega|) = \frac{1}{2}$.
The rest of the proof is the same as the proof of Lemma~2.8 in \cite{AE9},
with obvious changes.
\end{proof}

Next we remove the $\lambda$ to solve the elliptic problem for $\ca$, even 
allowing a divergence term. 
Note that the existence and uniqueness of $u$ in the next proposition was proved in 
Lemma~\ref{ldat260}\ref{ldat260-1}.

\begin{prop} \label{pdat214}
Let $p \in (d,\infty]$
and let $f_0,f_1,\ldots,f_d \in L_p(\Omega)$.
Let $u \in H^1_0(\Omega)$ be such that 
$\ca u = f_0 + \sum_{k=1}^d \partial_k f_k$.
Then $u \in C_b(\Omega)$.
Moreover, there exists a $C > 0$, depending only on 
the ellipticity constant, the $L_\infty$-norm of the coefficients and the diameter of $\Omega$,
such that 
\[
\|u\|_{C_b(\Omega)} 
\leq C \, \Big( \|u\|_{L_2(\Omega)} + \sum_{k=0}^d \|f_k\|_{L_p(\Omega)} \Big)
.  \]
\end{prop}
\begin{proof}
Without loss of generality we may assume that $p \in (d,4d)$.
Choose $\lambda = \delta = \lambda_0 + 1$ in Lemma~\ref{ldat207}\ref{ldat207-2}
and in Proposition~\ref{pdat263}.
Let $C > 0$ be as in Proposition~\ref{pdat263} and $C_2 \geq 1$ as in Lemma~\ref{ldat207}\ref{ldat207-2}.
By Proposition~\ref{pdat263} there exists a unique $\tilde u \in H^1_0(\Omega) \cap C_b(\Omega)$
such that $\cb_\lambda \tilde u = f_0 + \sum_{k=1}^d \partial_k f_k$.
Then $\|\tilde u\|_{L_\infty(\Omega)} \leq C \sum_{k=0}^d \|f_k\|_{L_p(\Omega)}$.
If $v \in C_c^\infty(\Omega)$, then 
\begin{eqnarray*}
\gotb_\lambda(\tilde u,v)
& = & \langle f_0 + \sum_{k=1}^d \partial_k f_k,v \rangle_{\cd'(\Omega) \times \cd(\Omega)}  \\
& = & \gota(u,v) \\
& = & \gotb_\lambda(u,v) 
   + \sum_{k=1}^d (b_k \, u, \partial_k v)_{L_2(\Omega)} 
   + ((c_0 - \delta \, \one_\Omega) \, u, v)_{L_2(\Omega)}
.  
\end{eqnarray*}
Therefore, by density,
\[
\gotb_\lambda(\tilde u - u,v) 
= \sum_{k=1}^d (b_k \, u, \partial_k v)_{L_2(\Omega)} 
   + ((c_0 - \delta \, \one_\Omega) \, u, v)_{L_2(\Omega)}
\]
for all $v \in H^1_0(\Omega)$.
If $R_\lambda$ is as in Lemma~\ref{ldat207}, then $\tilde u - u = R_\lambda u$.
For all $n \in \{ 0,\ldots,2d \} $ define $p_n = \frac{4d}{2d-n}$.
Then $p_0 = 2$ and $p_{2d} = \infty$.
Let $n \in \{ 1,\ldots,2d \} $.
One estimates from Lemma~\ref{ldat207}\ref{ldat207-2} that 
$\|\tilde u - u\|_{L_{p_n}(\Omega)} \leq C_2 \, \|u\|_{L_{p_{n-1}}(\Omega)}$, 
so 
\[
\|u\|_{L_{p_n}(\Omega)} 
\leq \|\tilde u\|_{L_{p_n}(\Omega)} + C_2 \, \|u\|_{L_{p_{n-1}}(\Omega)}
\leq (1 + |\Omega|) \|\tilde u\|_{L_\infty(\Omega)} + C_2 \, \|u\|_{L_{p_{n-1}}(\Omega)}
.  \]
Iteration gives 
\[
\|u\|_{L_\infty(\Omega)}
= \|u\|_{L_{p_{2d}}(\Omega)} 
\leq C_2^{2d} \, \|u\|_{L_2(\Omega)} 
    + (1 + |\Omega|) \sum_{k=0}^{2d-1} C_2^k \, \|\tilde u\|_{L_\infty(\Omega)}
.  \]
Then Lemma~\ref{ldat207}\ref{ldat207-3} gives  
$\tilde u - u = R_\lambda u \in C_b(\Omega)$ and therefore $u \in C_b(\Omega)$.
\end{proof}

\begin{cor} \label{cdat278}
Let $p \in (d,\infty]$.
Then $(A^D)^{-1} (L_p(\Omega)) \subset C_b(\Omega)$.
\end{cor}

\begin{cor} \label{cdat211} 
There exists a $c' > 0$, depending only on 
the ellipticity constant, the $L_\infty$-norm of the coefficients and the diameter of $\Omega$,
such that 
\[
\|(A^D)^{-1} f\|_{L_\infty(\Omega)} 
\leq c' \, (\|(A^D)^{-1}\|_{L_2(\Omega) \to L_2(\Omega)} + 1) \, \|f\|_{L_\infty(\Omega)}
\]
for all $f \in L_\infty(\Omega)$.
\end{cor}

We now solve the Dirichlet problem for functions of the form $\Phi|_\Gamma$ with 
$\Phi \in C(\overline \Omega) \cap H^1(\Omega)$ and show boundedness of
the solution.
We use the decomposition of Lemma~\ref{ldat260}\ref{ldat260-2}.

\begin{prop} \label{pdat269}
There exists a $c > 0$, depending only on 
the ellipticity constant, the $L_\infty$-norm of the coefficients and the diameter of $\Omega$, 
such that the following holds.
Let $\Phi \in C(\overline \Omega) \cap H^1(\Omega)$ and $u \in \ch_\ca^1(\Omega)$,
such that $\Phi - u \in H^1_0(\Omega)$.
Then $u \in C_b(\Omega)$ and
\[
\|u\|_{L_\infty(\Omega)}
\leq c \, (\|(A^D)^{-1}\|_{L_2(\Omega) \to L_2(\Omega)} + 1) \, \|\Phi|_\Gamma\|_{C(\Gamma)}
.  \]
\end{prop}
\begin{proof}
Let $C,\lambda > 0$ be as in Lemma~\ref{ldat280}
and let $c' > 0$ be as in Corollary~\ref{cdat211}.
Using Lemma~\ref{ldat280} there exists a unique 
$\tilde u \in C_b(\Omega) \cap H^1(\Omega)$
such that $\Phi - \tilde u \in H^1_0(\Omega)$ and $\ca_\lambda \tilde u = 0$.
Next Lemma~\ref{ldat260}\ref{ldat260-1} gives that there exists a unique $w \in H^1_0(\Omega)$ such that 
$\gota(w,v) = \gota(\Phi,v)$ for all $v \in H^1_0(\Omega)$.
Then $\ca(\Phi - w) = 0$, so $u = \Phi - w$ by Lemma~\ref{ldat260}\ref{ldat260-3}.
Define $\widetilde w = \Phi - \tilde u$.
Then 
\begin{eqnarray*}
\gota(w,v)
& = & \gota(\Phi,v)
= \gota_\lambda(\Phi,v) - \lambda \, (\Phi,v)_{L_2(\Omega)}
= \gota_\lambda(\widetilde w,v) - \lambda \, (\Phi,v)_{L_2(\Omega)}  \\
& = & \gota(\widetilde w,v) 
   + \lambda \, (\widetilde w,v)_{L_2(\Omega)}
   - \lambda \, (\Phi,v)_{L_2(\Omega)}
= \gota(\widetilde w,v) 
   - \lambda \, (\tilde u,v)_{L_2(\Omega)}
\end{eqnarray*}
for all $v \in H^1_0(\Omega)$. 
Hence
\[
\gota(\tilde u - u, v) 
= \gota(w - \widetilde w,v)
= - \lambda \, (\tilde u,v)_{L_2(\Omega)}
. \]
Therefore $\tilde u - u \in D(A^D)$ and
$A^D(\tilde u - u) = - \lambda \, \tilde u$.
Then $u = \tilde u + \lambda \, (A^D)^{-1} \tilde u \in C_b(\Omega)$
by Corollary~\ref{cdat278}.
Finally, by Corollary~\ref{cdat211} and Lemma~\ref{ldat280} one estimates
\begin{eqnarray*}
\|u\|_{C_b(\Omega)}
& \leq & \|\tilde u\|_{L_\infty(\Omega)} + \lambda \, \|(A^D)^{-1} \tilde u\|_{L_\infty(\Omega)}  \\
& \leq & \Big( 1 + c' \, \lambda \, (\|(A^D)^{-1}\|_{L_2(\Omega) \to L_2(\Omega)} + 1) \Big) \, \|\tilde u\|_{L_\infty(\Omega)}  \\
& \leq & \Big( 1 + c' \, \lambda \, (\|(A^D)^{-1}\|_{L_2(\Omega) \to L_2(\Omega)} + 1) \Big) \, C \, \|\Phi|_\Gamma\|_{C(\Gamma)}
\end{eqnarray*}
and the proof of Proposition~\ref{pdat269} is complete.
\end{proof}

Define $|||\cdot||| \colon H^1_\loc(\Omega) \to [0,\infty]$ by
\[
|||u||| 
= \sup_{\delta > 0} 
  \sup_{\scriptstyle \Omega_0 \subset \Omega \; {\rm open} \atop
        \scriptstyle d(\Omega_0,\Gamma) = \delta}
    \delta \Big( \int_{\Omega_0} |\nabla u|^2 \Big)^{1/2}
.  \]
Before we can prove Theorem~\ref{tdat120} we need the following Caccioppoli inequality.

\begin{prop} \label{pdat213}
There exists a $c' \geq 1$ such that 
$|||u||| \leq c' \, \|u\|_{L_2(\Omega)}$
for all $u \in H^1(\Omega)$ such that $\ca u = 0$.
\end{prop}
\begin{proof}
See \cite{GiM} Theorem~4.4.
\end{proof}

Theorem~\ref{tdat120} is a special case of the next theorem.

\begin{thm} \label{tdat235}
There exists a unique linear map $T \colon C(\Gamma) \to C_b(\Omega) \cap H^1_\loc(\Omega)$
such that $T$ is continuous from $C(\Gamma)$ into $C_b(\Omega)$,
the function $T \varphi$ is $\ca$-harmonic for all $\varphi \in C(\Gamma)$
and $\Phi|_\Omega - T(\Phi|_\Gamma) \in H^1_0(\Omega)$ for all 
$\Phi \in C(\overline \Omega) \cap H^1(\Omega)$.
Moreover, there exists a $c> 0$ which depends only on 
the ellipticity constant, the $L_\infty$-norm of the coefficients and the diameter of $\Omega$, 
such that 
\begin{equation}
\|T \varphi\|_{C_b(\Omega)}
\leq c \, (\|(A^D)^{-1}\|_{L_2(\Omega) \to L_2(\Omega)} + 1) \, \|\varphi\|_{C(\Gamma)}
\label{eldat821;20}
\end{equation}
for all $\varphi \in C(\Gamma)$.
\end{thm}
\begin{proof}
Let $c > 0$ and $c' \geq 1$ be as in Propositions~\ref{pdat269} and \ref{pdat213}.
Further let 
\[
F = \{ \Phi|_\Gamma : \Phi \in C(\overline \Omega) \cap H^1(\Omega) \}
.  \]
If $\Phi \in C(\overline \Omega) \cap H^1(\Omega)$, then we decompose
\[
\Phi = v + u 
\in H^1_0(\Omega) \oplus \ch_\ca^1(\Omega)
\]
using Lemma~\ref{ldat260}\ref{ldat260-2}.
Then $u \in C_b(\Omega)$ by Proposition~\ref{pdat269}.
We wish to define $T_0(\Phi|_\Gamma) = u$ and for that we need to verify that 
this definition does not depend on $\Phi$.
Let also $\widetilde \Phi \in C(\overline \Omega) \cap H^1(\Omega)$ and suppose that 
$\widetilde \Phi|_\Gamma = \Phi|_\Gamma$.
Let $\tilde u \in \ch_\ca^1(\Omega)$ be such that $\widetilde \Phi - \tilde u \in H^1_0(\Omega)$.
Then $\Phi - \widetilde \Phi - (u - \tilde u) \in H^1_0(\Omega)$
and Proposition~\ref{pdat269} gives
\[
\|u - \tilde u\|_{L_\infty(\Omega)}
\leq c \, (\|(A^D)^{-1}\|_{L_2(\Omega) \to L_2(\Omega)} + 1) \, 
     \|(\Phi - \widetilde \Phi)|_\Gamma\|_{C(\Gamma)}
= 0
.  \]
So $u = \tilde u$.
Therefore we can define an operator $T_0 \colon F \to C_b(\Omega)$ by
$T_0(\Phi|_\Gamma) = u$, where $u \in \ch_\ca^1(\Omega)$ 
is such that $\Phi - u \in H^1_0(\Omega)$.
Then Proposition~\ref{pdat269} gives that 
\[
\|T_0 \varphi\|_{C_b(\Omega)}
\leq c \, (\|(A^D)^{-1}\|_{L_2(\Omega) \to L_2(\Omega)} + 1) \, \|\varphi\|_{C(\Gamma)}
\]
for all $\varphi \in F$.
By the Stone--Weierstra\ss\ theorem $F$ is dense in $C(\Gamma)$.
Hence $T_0$ extends uniquely to a continuous operator $T \colon C(\Gamma) \to C_b(\Omega)$
such that (\ref{eldat821;20}) is valid.

We next show that $T \varphi \in H^1_\loc(\Omega)$ and $T \varphi$ is $\ca$-harmonic
for all $\varphi \in C(\Gamma)$.
If $\varphi \in F$ and $u = T_0 \varphi \in \ch_\ca^1(\Omega)$, then 
\begin{eqnarray}
|||u|||
& \leq & c' \, \|u\|_{L_2(\Omega)} 
\leq (1 + |\Omega|) \, c' \, \|u\|_{C_b(\Omega)}  \nonumber  \\
& \leq & (1 + |\Omega|) \, c \, c' \, (\|(A^D)^{-1}\|_{L_2(\Omega) \to L_2(\Omega)} + 1) \, \|\varphi\|_{C(\Gamma)}
\label{etdat235;1}
\end{eqnarray}
by Proposition~\ref{pdat213}.
Now let $\varphi \in C(\Gamma)$.
There are $\varphi_1,\varphi_2,\ldots \in F$ with $\lim \varphi_n = \varphi$ in $C(\Gamma)$.
Write $u_n = T_0 \varphi_n$ and $u = T \varphi$ for all $n \in \Ni$.
Then $u = \lim u_n$ in $C_b(\Omega)$ and (\ref{etdat235;1}) implies that 
$(u_n)_{n \in \Ni}$ is a Cauchy sequence in $H^1_\loc(\Omega)$.
So $u \in H^1_\loc(\Omega)$.
Since $\ca u_n = 0$  for all $n \in \Ni$, one deduces that $\ca u = 0$.
\end{proof}

In order to make the dependence on the domain or the elliptic operator 
explicit we write $T^\Omega$, or $T^{\ca}$ or $T^{\ca,\Omega}$ instead of~$T$
if confusion is possible.
One can also consider an inhomogeneous equation with a 
divergence term as in Lemma~\ref{ldat202.3}\ref{ldat202.3-2} on the right hand side.

\begin{cor} \label{cdat833}
Let $p \in (d,\infty]$.
Then there exists a unique continuous linear map 
\[
H \colon C(\Gamma) \times W^{-1,p}(\Omega) \to (C_b(\Omega) \cap H^1_\loc(\Omega), \|\cdot\|_{C_b(\Omega)})
\]
such that 
\begin{tabelR}
\item \label{cdat833-1}
$\ca \, H(\varphi, F) = F$
for all $\varphi \in C(\Gamma)$ and $F \in W^{-1,p}(\Omega)$
and
\item \label{cdat833-2}
$\Phi|_\Omega - H(\Phi|_\Gamma,  F) \in H^1_0(\Omega)$
for all $\Phi \in C(\overline \Omega) \cap H^1(\Omega)$ 
and $F \in W^{-1,p}(\Omega)$.
\end{tabelR}
\end{cor}
\begin{proof}
By Proposition~\ref{pdat214} there exists a linear map 
$Z \colon W^{-1,p}(\Omega) \to H^1_0(\Omega) \cap C_b(\Omega)$ such that 
$\ca \, Z(F) = F$
for all $F \in W^{-1,p}(\Omega)$.
Then $Z$ is continuous from $W^{-1,p}(\Omega)$ into $C_b(\Omega)$
by the closed graph theorem.
Consequently $H(\varphi, F) = T \varphi + Z(F)$
gives existence.
The uniqueness is clear.
\end{proof}

We next turn to the proof of Theorem~\ref{tdat825}.
The proof needs some preparation.
If $\Omega' \subset \Omega$ is open, then we
identify $L_2(\Omega')$ with a subspace of $L_2(\Omega)$,
and $H^1_0(\Omega')$ with a subspace of $H^1_0(\Omega)$
by extending functions by zero.

We next prove an approximation result which is of independent interest
(cf.\ \cite{Daners5} Corollary~4.7 for operators with real coefficients, where 
\ref{ldat821-2}--\ref{ldat821-1} are proved).

\begin{lemma} \label{ldat821}
Suppose that $\Omega_1,\Omega_2,\ldots \subset \Ri^d$ are open and bounded.
Assume that $\Omega_n \uparrow \Omega$.
For all $n \in \Ni$ let $A_n^D$ be the elliptic operator with Dirichlet boundary 
conditions in $L_2(\Omega_n)$ and with the coefficients $a_{kl}|_{\Omega_n}$ etc.
Define in the natural way the operator $\ca^{(n)} \colon H^1_\loc(\Omega_n) \to \cd'(\Omega_n)$.
Then one has the following.
\begin{tabel}
\item \label{ldat821-2}
There exists an $N \in \Ni$ such that $0 \not\in \sigma(A_n^D)$ for all $n \in \Ni$
with $n \geq N$.
\item \label{ldat821-3}
There exists an $N \in \Ni$ such that 
\[
\sup \{ \|(A_n^D)^{-1}\|_{L_2(\Omega_n) \to L_2(\Omega_n)} : n \in \Ni \mbox{ and } n \geq N \}
< \infty
.  \]
\item \label{ldat821-1}
$\lim_{n \to \infty} (A_n^D)^{-1}(f|_{\Omega_n}) = (A^D)^{-1} f$ in $L_2(\Omega)$
for all $f \in L_2(\Omega)$.
\item \label{ldat821-4}
Let $f_0,f_1,\ldots,f_d \in L_2(\Omega)$. 
Using Lemma~\ref{ldat260}\ref{ldat260-1}, for all large $n \in \Ni$ there exists
a unique $u_n \in H^1_0(\Omega_n)$ such that 
$\ca^{(n)} u_n = f_0|_{\Omega_n} + \sum_{k=1}^d \partial_k (f_k|_{\Omega_n})$.
Define similarly $u \in H^1_0(\Omega)$.
Then $\lim u_n = u$ in $L_2(\Omega)$.
\end{tabel} 
\end{lemma}
\begin{proof}
`\ref{ldat821-2} and \ref{ldat821-3}'. 
We first show that there exist $c > 0$ and $N \in \Ni$ such that 
\begin{equation}
\|A_n^D u\|_{L_2(\Omega_n)} \geq c \, \|u\|_{L_2(\Omega_n)}
\label{eldat821;6}
\end{equation}
for all $n \in \Ni$ and $u \in D(A_n^D)$ with $n \geq N$.
Suppose not. 
Passing to a subsequence if necessary, then for all $n \in \Ni$ there 
exists a $u_n \in D(A_n^D)$ such that 
$\|A_n^D u_n\|_{L_2(\Omega_n)} < \frac{1}{n} \, \|u_n\|_{L_2(\Omega_n)}$.
We may assume that $\|u_n\|_{L_2(\Omega_n)} = 1$.
Write $f_n = A_n^D u_n \in L_2(\Omega_n)$.
Then $\|f_n\|_{L_2(\Omega_n)} \leq \frac{1}{n}$.
Moreover, 
\[
\frac{\mu}{2} \, \|\nabla u_n\|_{L_2(\Omega)}^2
\leq \RRe \gota(u_n) + \lambda_0 \, \|u_n\|_{L_2(\Omega)}^2
= \RRe (f_n,u_n)_{L_2(\Omega)} + \lambda_0 \, \|u_n\|_{L_2(\Omega)}^2
\leq 1 +\lambda_0
.  \]
So the sequence $(u_n)_{n \in \Ni}$ is bounded in $H^1_0(\Omega)$.
Passing to a subsequence if necessary, there exists a $u \in H^1_0(\Omega)$ 
such that $\lim u_n = u$ weakly in $H^1_0(\Omega)$.
Since the embedding $H^1_0(\Omega) \subset L_2(\Omega)$ is compact, 
one deduces that $\lim u_n = u$ in $L_2(\Omega)$.
Then $\|u\|_{L_2(\Omega)} = 1$ and in particular $u \neq 0$.

Let $v \in C_c^\infty(\Omega)$.
Then $\supp v$ is compact, so $\supp v \subset \Omega_n$ for large $n \in \Ni$.
Hence $\gota(u_n,v) = (f_n,v)_{L_2(\Omega)}$.
Taking the limit $n \to \infty$ it follows that $\gota(u,v) = 0$.
Then by density $\gota(u,v) = 0$ for all $v \in H^1_0(\Omega)$.
Hence $u \in D(A^D)$ and $A^D u = 0$.
Since $0 \not\in \sigma(A^D)$ one deduces that $u = 0$.
This is a contradiction.
So (\ref{eldat821;6}) is valid.

It follows from (\ref{eldat821;6}) that the operator $A_n^D$ is injective for 
large $n \in \Ni$.
Since $A_n^D$ has compact resolvent, this implies that $0 \not\in \sigma(A_n^D)$, 
which is Statement~\ref{ldat821-2}.
Then Statement~\ref{ldat821-3} follows immediately from (\ref{eldat821;6}).

`\ref{ldat821-1}'. 
It is well known that 
$\lim_{n \to \infty} (A_n^D + \lambda_0 \, I)^{-1}(f|_{\Omega_n}) = (A^D + \lambda_0 \, I)^{-1} f$
for all $f \in L_2(\Omega)$.
Then the uniform bound of Statement~\ref{ldat821-3} implies Statement~\ref{ldat821-1}.

`\ref{ldat821-4}'. 
Define $\ca_{\lambda_0}^{(n)}$ in the natural way for all $n \in \Ni$.
For all $n \in \Ni$ there exists by Lemma~\ref{ldat260}\ref{ldat260-1} a unique 
$\tilde u_n \in H^1_0(\Omega_n)$ such that 
$\ca_{\lambda_0}^{(n)} \tilde u_n = f_0|_{\Omega_n} + \sum_{k=1}^d \partial_k (f_k|_{\Omega_n})$.
Define similarly $\tilde u \in H^1_0(\Omega)$.
Then 
\begin{eqnarray*}
\frac{\mu}{2} \, \|\tilde u_n\|_{H^1(\Omega_n)}^2
& \leq & \RRe \gota_{\lambda_0}(\tilde u_n)  \\
& = & \RRe \Big( (f_0, \tilde u_n)_{L_2(\Omega_n)} - \sum_{k=1}^d (f_k, \partial_k \tilde u_n)_{L_2(\Omega_n)} \Big)
\leq \|\tilde u_n\|_{H^1_0(\Omega_n)} \sum_{k=0}^d \|f_k\|_{L_2(\Omega)}
.  
\end{eqnarray*}
Hence $\|\tilde u_n\|_{H^1(\Omega)} \leq \frac{2}{\mu} \sum_{k=0}^d \|f_k\|_{L_2(\Omega)}$
and the sequence $(\tilde u_n)_{n \in \Ni}$ is bounded in $H^1_0(\Omega)$.
Passing to a subsequence if necessary, there exists a $\hat u \in H^1_0(\Omega)$
such that $\lim \tilde u_n = \hat u$ weakly in $H^1_0(\Omega)$.
Because the inclusion of $H^1_0(\Omega)$ in $L_2(\Omega)$ is compact, 
it follows that $\lim \tilde u_n = \hat u$ in $L_2(\Omega)$.

Let $v \in C_c^\infty(\Omega)$.
Then $\supp v$ is compact, so $\supp v \subset \Omega_n$ if $n \in \Ni$ is large enough.
For large $n \in \Ni$ one has
$\gota_{\lambda_0}(\tilde u_n,v) = (f_0, v)_{L_2(\Omega_n)} - \sum_{k=1}^d (f_k, \partial_k v)_{L_2(\Omega_n)}$.
Taking the limit $n \to \infty$ gives
\begin{equation}
\gota_{\lambda_0}(\hat u,v) = (f_0, v)_{L_2(\Omega)} - \sum_{k=1}^d (f_k, \partial_k v)_{L_2(\Omega)}
.
\label{eldat821;1}
\end{equation}
Since $C_c^\infty(\Omega)$ is dense in $H^1_0(\Omega)$, it follows that (\ref{eldat821;1})
is valid for all $v \in H^1_0(\Omega)$.
So $\hat u = \tilde u$.
We proved that $\lim \tilde u_n = \tilde u$ in $L_2(\Omega)$.

Next 
\[
\gota(u,v)
= (f_0, v)_{L_2(\Omega)} - \sum_{k=1}^d (f_k, \partial_k v)_{L_2(\Omega)}
= \gota_{\lambda_0}(\tilde u,v)
= \gota(\tilde u,v) + \lambda_0 \, (\tilde u,v)_{L_2(\Omega)}
\]
for all $v \in H^1_0(\Omega)$.
So $\gota(u - \tilde u,v) = \lambda_0 \, (\tilde u,v)_{L_2(\Omega)}$.
Since $u - \tilde u \in H^1_0(\Omega)$ one deduces that 
$u - \tilde u \in D(A^D)$ and $A^D(u - \tilde u) = \lambda_0 \, \tilde u$.
Then $u = \tilde u + \lambda_0 \, (A^D)^{-1} \tilde u$.
Similarly $u_n = \tilde u_n + \lambda_0 \, (A_n^D)^{-1} \tilde u_n$ for all large $n \in \Ni$.
Because 
\begin{eqnarray*}
\lefteqn{
\|(A_n^D)^{-1} \tilde u_n - (A^D)^{-1} \tilde u\|_{L_2(\Omega)}
} \hspace*{10mm} {} \\*
& \leq & \|(A_n^D)^{-1} (\tilde u_n - \tilde u)|_{\Omega_n}\|_{L_2(\Omega_n)}
   + \|(A_n^D)^{-1} (\tilde u|_{\Omega_n}) - (A^D)^{-1} \tilde u\|_{L_2(\Omega)}  \\
& \leq & \|(A_n^D)^{-1}\|_{L_2(\Omega_n) \to L_2(\Omega_n)} \, \|\tilde u_n - \tilde u|_{\Omega_n}\|_{L_2(\Omega_n)}
   + \|(A_n^D)^{-1} (\tilde u|_{\Omega_n}) - (A^D)^{-1} \tilde u\|_{L_2(\Omega)}
\end{eqnarray*}
the statement follows from Statements~\ref{ldat821-3} and \ref{ldat821-1}.
\end{proof}

We need a technical lemma.
Note that Proposition~\ref{pdat214} implies that 
the function $u$ in the next lemma is continuous.

\begin{lemma} \label{ldat820}
Let $K \subset \Omega$ be a compact set and $p \in (d,\infty)$.
Then there exist $c > 0$ and $\kappa \in (0,1)$, depending only on
the ellipticity constant, the $L_\infty$-norm of the coefficients, $p$ and $d(K,\Gamma)$, 
such that 
\[
|u(x)| \leq c \, (\|u\|_{H^1(\Omega)} + \sum_{k=0}^d \|f_k\|_{L_p(\Omega)})
\]
and 
\[
|u(x) - u(y)| \leq c \, |x-y|^\kappa \, (\|u\|_{H^1(\Omega)} + \sum_{k=0}^d \|f_k\|_{L_p(\Omega)})
\]
for all $x,y \in K$, $u \in H^1_0(\Omega)$ and $f_0,f_1,\ldots,f_d \in L_p(\Omega)$ 
such that $\ca u = f_0 + \sum_{k=1}^d \partial_k f_k$.
\end{lemma} 
\begin{proof}
If $b_k = c_k = c_0 = 0$, then this is a special case of \cite{ERe2} Theorem~1.2
with the choice $\Gamma = \emptyset$, $\Upsilon = K$,
$\zeta = d(K,\Gamma)$ and $\alpha = 1$.
The general case follows by iteration,
using Proposition~3.2 of \cite{ERe2}.
\end{proof}

The next statement is Theorem~\ref{tdat825}\ref{tdat825-3}.

\begin{prop} \label{pdat807}
Let $\Phi \in C(\overline \Omega)$.
Let $\Omega_1,\Omega_2,\ldots \subset \Ri^d$ be open and bounded
with $\Omega_n \uparrow \Omega$.
Let $K \subset \Omega$ be compact.
Then 
\[
\lim_{n \to \infty} T^{\Omega_n}(\Phi|_{\partial \Omega_n}) 
= T^\Omega(\Phi|_{\partial \Omega})
\]
uniformly on $K$.
\end{prop}
\begin{proof} 
We adopt the notation as in Lemma~\ref{ldat821}.
First suppose that there exists a $\widetilde \Phi \in C_c^\infty(\Ri^d)$ such that
$\Phi = \widetilde \Phi|_{\overline \Omega}$.
By Lemma~\ref{ldat821}\ref{ldat821-2} we may assume that 
$0 \not\in \sigma(A_n^D)$ for all $n \in \Ni$.
Then Lemma~\ref{ldat260}\ref{ldat260-2} and Proposition~\ref{pdat269} imply that
for all $n \in \Ni$ there exists a unique 
$u_n \in H^1(\Omega_n) \cap C_b(\Omega_n)$ such that $\ca^{(n)} u_n = 0$ and 
$\Phi|_{\Omega_n} - u_n \in H^1_0(\Omega_n)$.
Write $v_n = \Phi|_{\Omega_n} - u_n$.
Define similarly $u \in C_b(\Omega) \cap H^1(\Omega)$ and $v \in H^1_0(\Omega)$.
Define $f_0 = c_0 \, \Phi|_\Omega + \sum_{k=1}^d c_k \, \partial_k (\widetilde \Phi|_\Omega) \in L_\infty(\Omega)$
and $f_k = - b_k \, \Phi|_\Omega - \sum_{l=1}^d a_{kl} \, \partial_l (\widetilde \Phi|_\Omega) \in L_\infty(\Omega)$
for all $k \in \{ 1,\ldots,d \} $.
Then 
\[
\ca^{(n)} v_n 
= \ca^{(n)}(\Phi|_{\Omega_n}) 
= f_0|_{\Omega_n} + \sum_{k=1}^d \partial_k (f_k|_{\Omega_n})
\]
for all $n \in \Ni$ and a similar identity is valid for $\ca v$.
It follows from Lemma~\ref{ldat821}\ref{ldat821-4} that $\lim v_n = v$ in $L_2(\Omega)$.
In particular the sequence $(v_n)_{n \in \Ni}$ is bounded in $L_2(\Omega)$.
Since
\begin{eqnarray*}
\frac{\mu}{2} \, \|v_n\|_{H^1_0(\Omega)}^2
& = & \frac{\mu}{2} \, \|v_n\|_{H^1_0(\Omega_n)}^2
\leq \RRe \gota^{(n)}_{\lambda_0}(v_n)  \\
& = & \RRe (f_0, v_n)_{L_2(\Omega)} 
   - \sum_{k=1}^d (f_k, \partial_k v_n)_{L_2(\Omega)}
   + \lambda_0 \, \|v_n\|_{L_2(\Omega)}^2
\end{eqnarray*}
for all $n \in \Ni$, it follows that the sequence 
$(v_n)_{n \in \Ni}$ is also bounded in $H^1_0(\Omega)$.
Then Lemma~\ref{ldat820} gives that the sequence 
$(v_n|_K)_{n \in \Ni}$ is equicontinuous on~$K$.
Passing to a subsequence if necessary, there exists a $\tilde v \in C(K)$ such that 
$\lim v_n|_K = \tilde v$ uniformly on~$K$.
But $\lim v_n|_K = v|_K$ in $L_2(K)$.
So $\tilde v = v|_K$ and $\lim v_n|_K = v|_K$ uniformly on~$K$.
Since $u = \Phi|_\Omega - v$ and $u_n = \Phi|_{\Omega_n} - v_n$ for all $n \in \Ni$, 
the statement is proved if
there exists a $\widetilde \Phi \in C_c^\infty(\Ri^d)$ such that
$\Phi = \widetilde \Phi|_{\overline \Omega}$.

Now we prove the general case.
By Lemma~\ref{ldat821}\ref{ldat821-2} and \ref{ldat821-3} we may assume that 
$0 \not\in \sigma(A_n^D)$ for all $n \in \Ni$ and there exists an $M > 0$ such that 
\[
\|(A_n^D)^{-1}\|_{L_2(\Omega_n) \to L_2(\Omega_n)} \leq M
\]
for all $n \in \Ni$.
Then by (\ref{eldat821;20}) there exists a $\hat c > 0$ such that 
\[
\|T^{\Omega_n} \varphi\|_{C_b(\Omega_n)}
\leq \hat c  \, \|\varphi\|_{C(\partial \Omega_n)}
\]
for all $n \in \Ni$ and $\varphi \in C(\partial \Omega_n)$, and, in addition
\[
\|T \varphi\|_{C_b(\Omega)}
\leq \hat c  \, \|\varphi\|_{C(\Gamma)}
\]
for all $\varphi \in C(\Gamma)$.

Let $\varepsilon > 0$.
There exists a $\widetilde \Phi \in C_c^\infty(\Ri^d)$ such that 
\[
\|\widetilde \Phi|_{\overline \Omega} - \Phi\|_{C(\overline \Omega)} \leq \varepsilon
.  \]
Let $n \in \Ni$ and suppose that $K \subset \Omega_n$.
Then 
\begin{eqnarray*}
\lefteqn{
\|T^{\Omega_n}(\Phi|_{\partial \Omega_n}) - T(\Phi|_\Gamma)\|_{C(K)}
} \hspace*{20mm}  \\*
& \leq & \|T^{\Omega_n}(\Phi|_{\partial \Omega_n}) 
             - T^{\Omega_n}(\widetilde \Phi|_{\partial \Omega_n})\|_{C(K)}
   + \|T^{\Omega_n}(\widetilde \Phi|_{\partial \Omega_n}) - T(\widetilde \Phi|_{\partial \Omega})\|_{C(K)}
       \\*
& & \hspace*{10mm} {}   
+ \|T(\widetilde \Phi|_{\partial \Omega}) - T(\Phi|_\Gamma)\|_{C(K)}  \\*
& \leq & \|T^{\Omega_n}(\Phi|_{\partial \Omega_n} - \widetilde \Phi|_{\partial \Omega_n})\|_{C_b(\Omega_n)}
   + \|T^{\Omega_n}(\widetilde \Phi|_{\partial \Omega_n}) - T(\widetilde \Phi|_{\partial \Omega})\|_{C(K)}
       \\*
& & \hspace*{10mm} {}   
   + \|T(\widetilde \Phi|_\Gamma - \Phi|_\Gamma)\|_{C_b(\Omega)}  \\*
& \leq & \hat c \, \varepsilon 
   + \|T^{\Omega_n}(\widetilde \Phi|_{\partial \Omega_n}) - T(\widetilde \Phi|_{\partial \Omega})\|_{C(K)}
   + \hat c \, \varepsilon 
.
\end{eqnarray*}
Taking the limsup $n \to \infty$ and using the first case
one deduces that 
\[
\limsup_{n \to \infty} \|T^{\Omega_n}(\Phi|_{\partial \Omega_n}) - T(\Phi|_\Gamma)\|_{C(K)} 
\leq 2 \hat c \, \varepsilon
\]
and the proposition follows.
\end{proof}

Next we prove that $T$ behaves well on classical solutions.
This is Statement~\ref{tdat825-4} of Theorem~\ref{tdat825}.

\begin{prop} \label{pdat808}
Let $\Phi \in C(\overline \Omega) \cap H^1_\loc(\Omega)$ and suppose that $\ca \Phi = 0$.
Then $T(\Phi|_\Gamma) = \Phi|_\Omega$.
\end{prop}
\begin{proof}
There exist open and bounded $\Omega_1,\Omega_2,\ldots \subset \Ri^d$
such that $\Omega_n \uparrow \Omega$ and $\overline{\Omega_n} \subset \Omega$
for all $n \in \Ni$.
We adopt the notation as in Lemma~\ref{ldat821}.
By Lemma~\ref{ldat821}\ref{ldat821-2} and \ref{ldat821-3} we may assume that 
$0 \not\in \sigma(A_n^D)$ for all $n \in \Ni$ and there exists an $M > 0$ such that 
$\|(A_n^D)^{-1}\|_{L_2(\Omega_n) \to L_2(\Omega_n)} \leq M$ for all $n \in \Ni$.
Then by (\ref{eldat821;20}) there exists a $\hat c > 0$ such that 
\begin{equation}
\|T^{\Omega_n} \varphi\|_{C_b(\Omega_n)}
\leq \hat c  \, \|\varphi\|_{C(\partial \Omega_n)}
\label{epdat808;2}
\end{equation}
for all $n \in \Ni$ and $\varphi \in C(\partial \Omega_n)$,
and, in addition
\[
\|T \varphi\|_{C_b(\Omega)}
\leq \hat c  \, \|\varphi\|_{C(\Gamma)}
\]
for all $\varphi \in C(\Gamma)$.

Let $K \subset \Omega$ be compact.
Let $\varepsilon > 0$.
By density there exists a $\widetilde \Phi \in C_c^\infty(\Ri^d)$ such that 
$\|\widetilde \Phi|_{\overline \Omega} - \Phi\|_{C(\overline \Omega)} \leq \varepsilon$.
Let $n \in \Ni$ and suppose that $K \subset \Omega_n$.
Then 
\begin{eqnarray}
\|T(\Phi|_\Gamma) - \Phi\|_{C(K)}
& \leq & \|T(\Phi|_\Gamma) - T(\widetilde \Phi|_\Gamma)\|_{C(K)}
   + \|T(\widetilde \Phi|_{\partial \Omega}) 
         - T^{\Omega_n}(\widetilde \Phi|_{\partial \Omega_n})\|_{C(K)}  \nonumber  \\*
& & \hspace*{10mm} {}
   + \|T^{\Omega_n}(\widetilde \Phi|_{\partial \Omega_n}) - \Phi\|_{C(K)} 
.
\label{epdat808;1}
\end{eqnarray}
We estimate the first and last terms separately.

For the first term one estimates
\[
\|T(\Phi|_\Gamma) - T(\widetilde \Phi|_\Gamma)\|_{C(K)}
\leq \hat c \, \|\Phi|_\Gamma - \widetilde \Phi|_\Gamma\|_{C(\Gamma)}
\leq \hat c \, \varepsilon
.  \]
Consider the last term in (\ref{epdat808;1}).
One has $\Phi|_{\overline{\Omega_n}} \in C(\overline{\Omega_n}) \cap H^1(\Omega_n)$
and $\ca^{(n)}(\Phi|_{\Omega_n}) = 0$.
Hence $T^{\Omega_n}(\Phi|_{\partial \Omega_n}) = \Phi|_{\Omega_n}$
by Theorem~\ref{tdat120}.
Consequently 
\begin{eqnarray*}
\|T^{\Omega_n}(\widetilde \Phi|_{\partial \Omega_n}) - \Phi\|_{C(K)}
& = & \|T^{\Omega_n}(\widetilde \Phi|_{\partial \Omega_n}) 
     - T^{\Omega_n}(\Phi|_{\partial \Omega_n})\|_{C(K)}  \\
& \leq & \|T^{\Omega_n}(\widetilde \Phi|_{\partial \Omega_n}) 
     - T^{\Omega_n}(\Phi|_{\partial \Omega_n})\|_{C_b(\Omega_n)}  \\
& \leq & \hat c \, \|\widetilde \Phi|_{\partial \Omega_n} - \Phi|_{\partial \Omega_n}\|_{C(\partial \Omega_n)}
\leq \hat c \, \varepsilon
\end{eqnarray*}
by the uniform estimates of (\ref{epdat808;2}).

Then (\ref{epdat808;1}) gives
\[
\|T(\Phi|_\Gamma) - \Phi\|_{C(K)}
\leq 2 \hat c \, \varepsilon 
   + \|T(\widetilde \Phi|_{\partial \Omega}) 
     - T^{\Omega_n}(\widetilde \Phi|_{\partial \Omega_n})\|_{C(K)} 
.  \]
This is for all $n \in \Ni$ with $K \subset \Omega_n$.
Taking the limit $n \to \infty$ and using Proposition~\ref{pdat807}
one deduces that 
$\|T(\Phi|_\Gamma) - \Phi\|_{C(K)} \leq 2 \hat c \, \varepsilon$.
So $\|T(\Phi|_\Gamma) - \Phi\|_{C(K)} = 0$ 
and the proposition follows.
\end{proof}

The proof of Theorem~\ref{tdat825} is complete.

\begin{remark} \label{rdat236}
For the Laplacian Keldy\v{s} proved that $T^{-\Delta}$ is characterised
as the positive bounded linear operator $T \colon C(\Gamma) \to C_b(\Omega)$
that maps to the classical solution whenever it exists, see
\cite{Netuka} Section~3.
By \cite{KreuterThesis} Example~8.14 one cannot simply omit the condition on positivity.
Since in our setting the maximum principle can be violated, the operator $T^\ca$
is not positive in general.
\end{remark}

\section{Regular points} \label{Sdat3}

Throughout this section we adopt the notation and assumptions as in 
the first part of the introduction.
Recall that a point $z \in \Gamma$ is called {\bf $\ca$-regular} if 
\[
\lim_{x \to z} (T\varphi)(x) = \varphi(z)
\]
for all $\varphi \in C(\Gamma)$.
If $-\ca$ is the Laplacian, then we simply call $z$ {\bf regular}.
The domain $\Omega$ is called {\bf $\ca$-regular} if $z$ is $\ca$-regular for all $z \in \Gamma$.
So $\Omega$ is $\ca$-regular if and only if for each $\varphi \in C(\Gamma)$ 
the Dirichlet problem $(D_\varphi)$ has a unique classical solution.
In particular, $\Omega$ is Wiener regular if and only if every point of the boundary is regular.
By Wiener's theorem \cite{Wiener} a point $z \in \Gamma$ is regular 
if and only if 
\begin{equation}
\int_0^1 \frac{\capp( B(z,r) \setminus \Omega )}{r^{d-2}} \, \frac{dr}{r} 
= \infty
.  
\label{eSdat1;2}
\end{equation}
By Kellogg's theorem (see \cite{AdH} Theorems~6.3.2 and 6.3.4) 
the set 
\begin{equation}
S = \{ z \in \Gamma : z \mbox{ is not a regular point} \}
\label{eSdat3;1}
\end{equation}
is a polar set, that is $\capp(S) = 0$.

We need a basic approximation result as in \cite{LSW} Lemma~3.1.

\begin{lemma} \label{ldat830}
Let $z \in \Gamma$.
Then $z$ is $\ca$-regular if and only if 
\[
\lim_{x \to z} \Big( T(\Phi|_\Gamma) \Big)(x) = \Phi(z)
\]
for all $\Phi \in C^1(\Ri^d)$.
\end{lemma}
\begin{proof}
This follows from the fact that the space $ \{ \Phi|_\Gamma : \Phi \in C^1(\Ri^d) \} $
is dense in $C(\Gamma)$.
\end{proof}

The proof of the following theorem is inspired by \cite{Chicco1}.

\begin{thm} \label{tdat831}
Let $z \in \Gamma$ and let $p \in (d,\infty]$.
Then the following are equivalent.
\begin{tabeleq}
\item \label{tdat831-2}
$z$ is regular.
\item \label{tdat831-1}
$z$ is $\ca$-regular.
\item \label{tdat831-3}
Let $H$ be as in Corollary~{\rm \ref{cdat833}}.
Then $\lim_{x \to z} \Big( H(\varphi, F) \Big) (x) = \varphi(z)$
for all $\varphi \in C(\Gamma)$ and $F \in W^{-1,p}(\Omega)$.
\end{tabeleq}
\end{thm}
\begin{proof}
`\ref{tdat831-3}$\Rightarrow$\ref{tdat831-1}'.
Trivial.

`\ref{tdat831-1}$\Rightarrow$\ref{tdat831-2}'.
Suppose $z$ is $\ca$-regular.
We shall show that the point $z$ is $\cb_1$-regular.
Let $\Phi \in C(\overline \Omega) \cap H^1(\Omega)$.
Write $u = T^{\cb_1}(\Phi|_\Gamma)$.
Then $u \in C_b(\Omega) \cap H^1(\Omega)$, $u - \Phi \in H^1_0(\Omega)$
and $\gotb_1(u,v) = 0$ for all $v \in H^1_0(\Omega)$ by Theorem~\ref{tdat120}.
There exists an open and bounded $\widetilde \Omega \subset \Ri^d$
such that $\overline \Omega \subset \widetilde \Omega$.
Extend $a_{kl}$ on $\widetilde \Omega \setminus \Omega$ by 
$a_{kl}(x) = \delta_{kl}$ for all $x \in \widetilde \Omega \setminus \Omega$.
Further extend $b_k$, $c_k$ and $c_0$ by $0$ on $\widetilde \Omega \setminus \Omega$.
Let $\widetilde{\gota}$ be the associated form on $\widetilde \Omega$.
By Lemma~\ref{ldat260}\ref{ldat260-1} and Proposition~\ref{pdat214} there exists a 
$w \in C_b(\widetilde \Omega) \cap H^1_0(\widetilde \Omega)$ such that 
\[
\widetilde{\gota}(w,v)
= \int_\Omega \sum_{k=1}^d b_k \, u \, \overline{\partial_k v}
   + \int_\Omega (c_0 - \one) \, u \, \overline v
\]
for all $v \in H^1_0(\widetilde \Omega)$.
Then $w|_{\overline \Omega} \in C(\overline \Omega) \cap H^1(\Omega)$.
If $v \in H^1_0(\Omega)$, then
\begin{eqnarray*}
\gota(u - w|_\Omega,v)
& = & \gota(u,v) - \widetilde{\gota}(w,v)  \\
& = & \gota(u,v) - \int_\Omega \sum_{k=1}^d b_k \, u \, \overline{\partial_k v}
   - \int_\Omega (c_0 - \one) \, u \, \overline v  \\
& = & \gotb_1(u,v)
= 0
.
\end{eqnarray*}
Now $\Phi - w|_{\overline \Omega} \in C(\overline \Omega) \cap H^1(\Omega)$
and 
\[
\Phi - w|_{\overline \Omega} - (u - w|_{\overline \Omega}) 
= \Phi - u
\in H^1_0(\Omega)
.  \]
Hence $T^{\ca,\Omega}((\Phi - w)|_\Gamma) = u - w|_\Omega$
by Theorem~\ref{tdat120}.
Since $z$ is $\ca$-regular one establishes that 
$\lim_{x \to z} (u - w|_\Omega)(x) = (\Phi - w)(z)$.
Because $w|_{\overline \Omega} \in C(\overline \Omega)$ this implies that 
$\lim_{x \to z} u(x) = \Phi(z)$.
By Lemma~\ref{ldat830} one deduces that $z$ is $\cb_1$-regular.
Hence by \cite{Stam2} Th\'eor\`eme~10.2 the point $z$ is regular.

`\ref{tdat831-2}$\Rightarrow$\ref{tdat831-3}'.
Suppose $z$ is regular.
Then by \cite{Stam2} Th\'eor\`eme~10.2 the point $z$ is $\cb_1$-regular.
Let $\Phi \in C(\overline \Omega) \cap H^1(\Omega)$
and $F \in W^{-1,p}(\Omega)$.
This time write $u = H(\Phi|_\Gamma,F)$, using the operator $\ca$.
Then $u \in C_b(\Omega) \cap H^1(\Omega)$, $u - \Phi \in H^1_0(\Omega)$
and 
\begin{equation}
F(v)
= \gota(u,v)
= \gotb_1(u,v) + \int_\Omega \sum_{k=1}^d b_k \, u \, \overline{\partial_k v}
   + \int_\Omega (c_0 - \one) \, u \, \overline v
\label{etdat831;1}
\end{equation}
for all $v \in H^1_0(\Omega)$.
There exists an open and bounded $\widetilde \Omega \subset \Ri^d$
such that $\overline \Omega \subset \widetilde \Omega$.
Extend $a_{kl}$ on $\widetilde \Omega \setminus \Omega$ by 
$a_{kl}(x) = \delta_{kl}$ for all $x \in \widetilde \Omega \setminus \Omega$.
Further extend $c_k$ by $0$ on $\widetilde \Omega \setminus \Omega$.
Let $\widetilde{\gotb}_1$ be the associated form on $\widetilde \Omega$.
By Lemma~\ref{ldat202.3}\ref{ldat202.3-1}
there exist $f_1,\ldots,f_d \in L_2(\Omega)$ such that 
$F = - \sum_{k=1}^d \partial_k f_k$ as distribution on $\Omega$.
By Lemma~\ref{ldat260}\ref{ldat260-1} and Proposition~\ref{pdat214} there exists a 
$w \in C_b(\widetilde \Omega) \cap H^1_0(\widetilde \Omega)$ such that 
\begin{equation}
\widetilde{\gotb}_1(w,v)
= \sum_{k=1}^d \int_\Omega f_k \, \overline{\partial_k v}
   - \int_\Omega \sum_{k=1}^d b_k \, u \, \overline{\partial_k v}
   - \int_\Omega (c_0 - \one) \, u \, \overline v
\label{etdat831;2}
\end{equation}
for all $v \in H^1_0(\widetilde \Omega)$.
Then $w|_{\overline \Omega} \in C(\overline \Omega) \cap H^1(\Omega)$.
Using (\ref{etdat831;1}) and (\ref{etdat831;2}) one deduces that 
\[
\gotb_1(u,v) 
= \widetilde{\gotb}_1(w,v)
= \gotb_1(w|_\Omega,v)
\]
for all $v \in H^1_0(\Omega)$.
Hence $\gotb_1(u - w|_\Omega,v) = 0$ for all $v \in H^1_0(\Omega)$.
Now $\Phi - w|_{\overline \Omega} \in C(\overline \Omega) \cap H^1(\Omega)$
and 
\[
\Phi - w|_{\overline \Omega} - (u - w|_{\overline \Omega}) 
= \Phi - u
\in H^1_0(\Omega)
.  \]
Hence $T^{\cb_1,\Omega}((\Phi - w)|_\Gamma) = u - w|_\Omega$
by Theorem~\ref{tdat120}.
Since $z$ is $\cb_1$-regular one establishes that 
$\lim_{x \to z} (u - w|_\Omega)(x) = (\Phi - w)(z)$.
Using again that $w|_{\overline \Omega} \in C(\overline \Omega)$ this implies that 
$\lim_{x \to z} u(x) = \Phi(z)$.
Arguing as in the proof of Lemma~\ref{ldat830} one deduces that $z$ is $\ca$-regular.
\end{proof}

Theorem~\ref{tdat831} was mentioned in the introduction.

\begin{cor} \label{tdat834}
The set $\Omega$ is $\ca$-regular if and only if $\Omega$ is Wiener regular.
\end{cor}

As a consequence the Dirichlet problem for a general elliptic operator
admits a classical solution if $\Omega$ is Wiener regular.

\begin{cor} \label{cdat320}
Suppose that $\Omega$ is Wiener regular.
Then for all $\varphi \in C(\Gamma)$ the Dirichlet problem $(D_\varphi)$
has a classical solution.
\end{cor}
\begin{proof}
Define $u \colon \overline \Omega \to \Ci$ by 
\[
u(x) 
= \left\{ \begin{array}{ll}
   (T \varphi)(x) & \mbox{if } x \in \Omega,  \\[5pt]
   \varphi(x) & \mbox{if } x \in \Gamma .
          \end{array} \right.
\]
Then $u \in C(\overline \Omega)$ since $\Omega$ is $\ca$-regular.
So $u$ is a classical solution.
\end{proof}

This gives a mostly new proof of the main result in \cite{AE9}.
Theorem~\ref{tdat831}  also shows that the function $u$ in Proposition~\ref{pdat214} tends to zero at regular 
points of the boundary.
More precisely, the following is valid.

\begin{cor} \label{cdat832}
Let $p \in (d,\infty]$ and $F \in W^{-1,p}(\Omega)$.
Let $u \in H^1_0(\Omega)$ be such that 
$\ca u = F$.
Then $u \in C_b(\Omega)$ and $\lim_{x \to z} u(x) = 0$ for every regular $z \in \Gamma$.
\end{cor} 
\begin{proof}
One has $u = H(0, F)$.
\end{proof}

\begin{cor} \label{cdat834}
Suppose that $\Omega$ is Wiener regular.
Let $p \in (d,\infty]$ and $F \in W^{-1,p}(\Omega)$.
Let $u \in H^1_0(\Omega)$ be such that 
$\ca u = F$.
Then $u \in C_0(\Omega)$.
\end{cor}

This is Proposition~2.9 in \cite{AE9}.

\medskip

In the preceding corollaries it is important that $u$ is a solution.
In fact, for arbitrary functions in $C_b(\Omega) \cap H^1_0(\Omega)$ the 
pointwise behaviour at the boundary may be very bad.
We give an elementary example
which relies on the following lemma on compact sets with 
capacity zero, which we use several more times in this paper.

\begin{lemma} \label{ldat834.2}
Let $K \subset \Ri^d$ be a compact set with $\capp K = 0$.
Let $U \subset \Ri^d$ be open with $K \subset U$ and let $\varepsilon > 0$.
Then there exist $\chi \in C_c^\infty(\Ri^d)$ and open $V \subset \Ri^d$ such that
$K \subset V \subset U$, $\chi(x) = 1$ for all $x \in V$, $0 \leq \chi \leq \one$,
$\supp \chi \subset U$ and $\|\chi\|_{H^1(\Ri^d)} \leq \varepsilon$.
\end{lemma}
\begin{proof}
There exists a $\tau \in C_c^\infty(\Ri^d)$ such that 
$\tau(x) = 1$ for all $x$ in a neighbourhood of $K$, $0 \leq \tau \leq 1$
and $\supp \tau \subset U$.
By \cite{FOT} Lemma~2.2.7(ii) there exists a $\chi_1 \in C_c^\infty(\Ri^d)$ such that 
$\chi_1(x) \geq 1$ for all $x \in K$ and $\|\chi_1\|_{H^1(\Ri^d)} \leq \varepsilon$.
Define $\chi_2 = 0 \vee 2 \chi_1 \wedge \one$.
Then $\chi_2 \in W^{1,\infty}(\Ri^d)$, $0 \leq \chi_2 \leq \one$,
$\chi_2(x) = 1$ for all $x$ in a neighbourhood of $K$
and $\|\chi_2\|_{H^1(\Ri^d)} \leq 2 \varepsilon$.
Regularising $\chi_2$ one obtains that there exists a 
$\chi_3 \in C_c^\infty(\Ri^d)$ such that
$\chi_3(x) = 1$ for all $x$ in a neighbourhood of $K$, $0 \leq \chi_3 \leq \one$
and $\|\chi_3\|_{H^1(\Ri^d)} \leq 2 \varepsilon$.
Finally define $\chi = \tau \, \chi_3$. 
Then $\|\chi\|_{H^1(\Ri^d)} \leq 2 \varepsilon (2 (1 + \|\nabla \tau\|_\infty^2))^{1/2}$,
$\chi(x) = 1$ for all $x$ in a neighbourhood of $K$, $0 \leq \chi \leq \one$
and $\supp \chi \subset U$.
\end{proof}

\begin{exam} \label{xdat834.4}
Let $\Omega = \{ x \in \Ri^d : |x| < 1 \} $ be the unit ball.
We shall show that there exists a $u \in C_b(\Omega) \cap H^1_0(\Omega)$ such that 
$u(x)$ does not converge as $x \to z$ for any $z \in \Gamma$.

Choose a sequence $(x_n)_{n \in \Ni}$ in $\Omega$ such that 
$x_n \neq x_m$ for all $n,m \in \Ni$ with $n \neq m$ and 
such that $\Gamma$ is the set of all accumulation points of $(x_n)_{n \in \Ni}$.
For all $n \in \Ni$ there is an $r_n \in (0,2^{-n})$ such that 
$B(x_n, 2 r_n) \subset \Omega$ and 
$B(x_n, 2 r_n) \cap B(x_m, 2 r_m) = \emptyset$ for all $n,m \in \Ni$
with $n \neq m$.
Let $n \in \Ni$. 
Then $\capp \{ x_n \} = 0$ since $d \geq 2$.
Hence by Lemma~\ref{ldat834.2} there exists a $u_n \in C_c^\infty(\Ri^d)$ 
such that $\supp u_n \subset B(x_n, r_n)$, $u_n(x_n) = 1$, $0 \leq u_n \leq \one$ and 
$\|u_n\|_{H^1_0(\Omega)} \leq 2^{-n}$.
Define $u = \sum_{n=1}^\infty u_n$.
Then $u \in C^\infty(\Omega) \cap H^1_0(\Omega)$ and 
$0 \leq u \leq \one$.
But $u(x_n) = 1$ and $u(x_n + r_n \, e_1) = 0$ for all $n \in \Ni$,
where $e_1$ is the first basis vector in $\Ri^d$.
Hence $u(x)$ does not converge as $x \to z$ for any $z \in \Gamma$.
\end{exam}

There is also a converse of Corollary~\ref{cdat832}.
The following is a remarkable extension of \cite{AE9} Lemma~2.2
(where $u \in C_0(\Omega) \cap H^1_\loc(\Omega)$ and $\chi = \one_\Omega$).

\begin{thm} \label{tdat243}
Let $u \in C_b(\Omega) \cap H^1_\loc(\Omega)$ and $\chi \in C_b(\Omega) \cap H^1(\Omega)$
with $0 \leq \chi \leq \one$.
Suppose that $\ca u \in H^{-1}(\Omega)$
and $\limsup_{x \to z} (\chi \, \RRe u)(x) \leq 0$ for quasi every $z \in \Gamma$.
Then $(\chi \, \RRe u)^+ \in H^1_0(\Omega)$.
\end{thm}

The proof requires quite some work and the main step is first to prove
the theorem with $\chi^2$ instead of $\chi$ and as an 
intermediate step that $(\chi^2 \, \RRe u)^+ \in H^1(\Omega)$.

\begin{lemma} \label{ldat243.1}
Let $u \in C_b(\Omega) \cap H^1_\loc(\Omega)$ and $\chi \in C_b(\Omega) \cap H^1(\Omega)$
with $0 \leq \chi \leq \one$.
Suppose that $\ca u \in H^{-1}(\Omega)$
and $\limsup_{x \to z} (\chi^2 \, \RRe u)(x) \leq 0$ for quasi every $z \in \Gamma$.
Then $(\chi \, \RRe u)^+ \in C_b(\Omega) \cap H^1(\Omega)$ and 
$(\chi^2 \, \RRe u)^+ \in H^1_0(\Omega)$.
\end{lemma}
\begin{proof}
For all $\varepsilon > 0$ define the set
\[
K_\varepsilon 
= \{ z \in \Gamma : \limsup_{x \to z} (\chi^2 \, \RRe u)(x) \geq \varepsilon \} 
.  \]
We first show that $K_\varepsilon$ is closed.
Let $z_1,z_2,\ldots \in K_\varepsilon$, $z \in \Gamma$ and suppose that 
$\lim z_n = z$.
For all $n \in \Ni$ there exists an $x_n \in \Omega$ such that 
$|x_n - z_n| < \frac{1}{n}$ and $(\chi^2 \, \RRe u)(x_n) \geq \varepsilon - \frac{1}{n}$.
Then $\limsup_{x \to z} (\chi^2 \, \RRe u)(x) 
\geq \limsup_{n \to \infty} (\chi^2 \, \RRe u)(x_n) \geq \varepsilon$,
so $z \in K_\varepsilon$.
Hence $K_\varepsilon$ is closed and therefore compact. 
Obviously $\capp(K_\varepsilon) = 0$.
Note that $u \in L_2(\Omega)$, since $\Omega$ and $u$ are bounded.

Let $\delta > 0$ and $n \in \Ni$ with $\frac{1}{n} < \delta$.
We emphasise that all the following estimates involving constants $M_1,M_2,\ldots$ 
are uniform in both $\delta$ and~$n$.
By Lemma~\ref{ldat834.2} there exists a $\chi_n \in C_c^\infty(\Ri^d)$
such that $0 \leq \chi_n \leq \one$, $\|\chi_n\|_{H^1(\Ri^d)} \leq \frac{1}{n}$ and 
$\chi_n(z) = 1$ for all $z \in K_{1/n}$.
Passing to a subsequence if necessary, we may assume that $\lim \chi_n = 0$ 
almost everywhere on $\Ri^d$.
Define 
\[
u_n = \Big( (\one - \chi_n)^2 \, \chi^2 \, \RRe u - \delta \, \one \Big)^+
.  \]
(In order to avoid clutter we drop the dependence of $\delta$ in the notation 
of $u_n$.)
Then $\supp u_n \subset \Omega$ since $\frac{1}{n} < \delta$.
Hence there is an open $\Omega_1 \subset \Ri^d$ such that 
$\supp u_n \subset \Omega_1 \subset \overline{\Omega_1} \subset \Omega$.
Then $u_n \in H^1_0(\Omega_1)$.
By assumption there exists $M_1 > 0$ such that 
$|\langle \ca u, v \rangle| \leq M_1 \|v\|_{H^1(\Omega)}$ for all $v \in C_c^\infty(\Omega)$.
Hence
\begin{eqnarray}
\lefteqn{
\Big| \sum_{k,l=1}^d \int_{\Omega_1} a_{kl} \, (\partial_l u) \, \overline{\partial_k v}
   + \sum_{k=1}^d \int_{\Omega_1} b_k \, u \, \overline{\partial_k v}
   + \sum_{k=1}^d \int_{\Omega_1} c_k (\partial_k u) \, \overline v
   + \int_{\Omega_1} c_0 \, u \, \overline v \Big|
} \hspace*{120mm}    \label{etdat243;1}   \\*
& \leq & M_1 \, \|v\|_{H^1(\Omega)}
\nonumber
\end{eqnarray}
for all $v \in C_c^\infty(\Omega_1)$.
Since $u|_{\Omega_1} \in H^1(\Omega_1)$ it follows that (\ref{etdat243;1})
is valid for all $v \in H^1_0(\Omega_1)$ and in particular for $v = u_n$.
Therefore
\[
\Big| \sum_{k,l=1}^d \int_\Omega a_{kl} \, (\partial_l u) \, \overline{\partial_k u_n}
   + \sum_{k=1}^d \int_\Omega b_k \, u \, \overline{\partial_k u_n}
   + \sum_{k=1}^d \int_\Omega c_k (\partial_k u) \, \overline{u_n}
   + \int_\Omega c_0 \, u \, \overline{u_n} \Big|
\leq M_1 \, \|u_n\|_{H^1(\Omega)}
.  \]
Note that 
\[
\|u_n\|_{L_2(\Omega)}
\leq \|(\one - \chi_n)^2 \, \chi^2 \, \RRe u\|_{L_2(\Omega)}
\leq \|u\|_{L_2(\Omega)}
.  \]
Secondly
\[
\partial_k u_n
= \Big( (\one - \chi_n)^2 \, \chi^2 \partial_k \RRe u 
        - 2 (\RRe u) (\one - \chi_n) \, \chi^2 \, \partial_k \chi_n 
        + 2 (\RRe u) (\one - \chi_n)^2 \, \chi \, \partial_k \chi \Big)
 \, \one_{[u_n > 0]}
\]
for all $k \in \{ 1,\ldots,d \} $.
Hence
\begin{eqnarray*}
\|\partial_k u_n\|_{L_2(\Omega)}
& \leq & \|\one_{[u_n > 0]} \, (\one - \chi_n) \, \chi \, \partial_k \RRe u\|_{L_2(\Omega)}
   + 2 \|u\|_{C_b(\Omega)} \, (\|\partial_k \chi_n\|_{L_2(\Omega)} + \|\partial_k \chi\|_{L_2(\Omega)} )  
\end{eqnarray*}
for all $k \in \{ 1,\ldots,d \} $.
Then 
\[
\|\nabla u_n\|_2^2
\leq 3 \|\one_{[u_n > 0]} \, (\one - \chi_n) \, \chi \, \nabla \RRe u\|_2^2
   + 12 (1 + \|\nabla \chi\|_2^2) \, \|u\|_{C_b(\Omega)}^2
\]
and 
\[
\|u_n\|_{H^1(\Omega)}
\leq 2 \|\one_{[u_n > 0]} \, (\one - \chi_n) \, \chi \, \nabla \RRe u\|_2
   + 4 (1 + \|\nabla \chi\|_2) \, \|u\|_{C_b(\Omega)} + \|u\|_{L_2(\Omega)}
.  \]
We wish to estimate $\|\one_{[u_n > 0]} \, (\one - \chi_n) \, \chi \, \nabla \RRe u\|_2^2$.

Clearly 
\begin{eqnarray*}
M_1 \, \|u_n\|_{H^1(\Omega)}
& \geq &
\Big| \sum_{k,l=1}^d \int_\Omega a_{kl} \, (\partial_l u) \, \overline{\partial_k u_n}
   + \sum_{k=1}^d \int_\Omega b_k \, u \, \overline{\partial_k u_n}
   + \sum_{k=1}^d \int_\Omega c_k (\partial_k u) \, \overline{u_n}
   + \int_\Omega c_0 \, u \, \overline{u_n} \Big|  \\
& \geq & \RRe \sum_{k,l=1}^d \int_\Omega a_{kl} \, (\partial_l u) \, \overline{\partial_k u_n}
      \\*
& & \hspace{10mm} {}
   + \RRe \sum_{k=1}^d \int_\Omega b_k \, u \, \overline{\partial_k u_n}
   + \RRe\sum_{k=1}^d \int_\Omega c_k (\partial_k u) \, \overline{u_n}
   + \RRe \int_\Omega c_0 \, u \, \overline{u_n}  \\
& = & \int_\Omega \sum_{k,l=1}^d \one_{[u_n > 0]} \, (\one - \chi_n)^2 \, \chi^2 \, a_{kl} \, 
      (\partial_l \RRe u) \, \partial_k \RRe u  \\
& & \hspace{10mm} {}
   - 2 \int_\Omega \sum_{k,l=1}^d \one_{[u_n > 0]} \, a_{kl} \, 
     (\partial_l \RRe u) \, (\RRe u) \, (\one - \chi_n) \, \chi^2 \, \partial_k \chi_n  \\
& & \hspace{10mm} {}
   + 2 \int_\Omega \sum_{k,l=1}^d \one_{[u_n > 0]} \, a_{kl} \, 
     (\partial_l \RRe u) \, (\RRe u) \, (\one - \chi_n)^2 \, \chi \, \partial_k \chi  \\
& & \hspace{10mm} {}
   + \RRe \int_\Omega \sum_{k=1}^d b_k \, u \, \overline{\partial_k u_n}  \\
& & \hspace{10mm} {}
   + \int_\Omega \sum_{k=1}^d c_k \, (\partial_k \RRe u) \, 
        \Big( (\one - \chi_n)^2 \, \chi^2 \, \RRe u - \delta \, \one \Big)^+  \\
& & \hspace{10mm} {}
   + \RRe \int_\Omega c_0 \, u \, 
        \Big( (\one - \chi_n)^2 \, \chi^2 \, \RRe u - \delta \, \one \Big)^+   \\
& = & I_1 + I_2 + I_3 + I_4 + I_5 + I_6
\geq I_1 - |I_2| - |I_3| - |I_4| - |I_5| - |I_6|
.
\end{eqnarray*}
We estimate the six terms separately.

Ellipticity gives
\begin{eqnarray*}
I_1 
& = & \int_\Omega \sum_{k,l=1}^d \one_{[u_n > 0]} \, (\one - \chi_n)^2 \, \chi^2 \, a_{kl} \, 
      (\partial_l \RRe u) \, \partial_k \RRe u  \\
& \geq & \mu \int_\Omega \one_{[u_n > 0]} \, (\one - \chi_n)^2 \, \chi^2 
       \sum_{k=1}^d |\partial_k \RRe u|^2
= \mu \, \|\one_{[u_n > 0]} \, (\one - \chi_n) \, \chi \, \nabla \RRe u\|_2^2
.  
\end{eqnarray*}
Using the Cauchy--Schwarz inequality one obtains
\begin{eqnarray*}
|I_2|
& = & \Big| - 2 \int_\Omega \sum_{k,l=1}^d \one_{[u_n > 0]} \, a_{kl} \, 
     (\partial_l \RRe u) \, (\RRe u) \, (\one - \chi_n) \, \chi^2 \, \partial_k \chi_n \Big|   \\
& \leq & M_2 \, \|u\|_{C_b(\Omega)} \, \|\one_{[u_n > 0]} \, (\one - \chi_n) \, \chi \, \nabla \RRe u\|_2 \, \|\nabla \chi_n\|_2  \\
& \leq & M_2 \, \|u\|_{C_b(\Omega)} \, \|\one_{[u_n > 0]} \, (\one - \chi_n) \, \chi \, \nabla \RRe u\|_2
,
\end{eqnarray*}
where $M_2 = 2 \sum_{k,l=1}^d \|a_{kl}\|_\infty$.
Similarly
\begin{eqnarray*}
|I_3|
& = & \Big| 2 \int_\Omega \sum_{k,l=1}^d \one_{[u_n > 0]} \, a_{kl} \, 
     (\partial_l \RRe u) \, (\RRe u) \, (\one - \chi_n)^2 \, \chi \, \partial_k \chi \Big|   \\
& \leq & M_2 \, \|u\|_{C_b(\Omega)} \, \|\nabla \chi\|_2 \, \|\one_{[u_n > 0]} \, (\one - \chi_n) \, \chi \, \nabla \RRe u\|_2
\end{eqnarray*}
and
\begin{eqnarray*}
|I_4|
& = & \Big| \RRe \int_\Omega \sum_{k=1}^d b_k \, u \, \overline{\partial_k u_n} \Big|   \\
& \leq & M_3 \, \|u\|_{L_2(\Omega)} \, \Big( \|\one_{[u_n > 0]} \, (\one - \chi_n) \, \nabla \RRe u\|_2
                               + 2 (1 + \|\nabla \chi\|_2) \, \|u\|_{C_b(\Omega)} \Big)  \\
& = & M_3 \, \|u\|_{L_2(\Omega)} \, \|\one_{[u_n > 0]} \, (\one - \chi_n) \, \nabla \RRe u\|_2
   + 2 M_3 \, (1 + \|\nabla \chi\|_2) \, \|u\|_{L_2(\Omega)} \, \|u\|_{C_b(\Omega)}
,
\end{eqnarray*}
where $M_3 = \sum_{k=1}^d \|b_k\|_\infty$.
For the next term we use the inequality $(t - \delta)^+ \leq t^+$ for all $t \in \Ri$ 
to obtain
\begin{eqnarray*}
|I_5|
& = & \Big| \int_\Omega \sum_{k=1}^d c_k \, (\partial_k \RRe u) \, 
        \Big( (\one - \chi_n)^2 \, \chi^2 \, \RRe u - \delta \, \one \Big)^+ \Big|   \\
& \leq & \int_\Omega \sum_{k=1}^d |c_k| \, |\partial_k \RRe u| \, 
        \one_{[u_n > 0]} \, \Big( (\one - \chi_n)^2 \, \chi^2 \RRe u - \delta \, \one \Big)^+   \\
& \leq & \int_\Omega \sum_{k=1}^d |c_k| \, |\partial_k \RRe u| \, 
         \one_{[u_n > 0]} \, (\one - \chi_n)^2 \, \chi^2 \, |\RRe u|   \\
& \leq & M_4 \, \|\one_{[u_n > 0]} \, (\one - \chi_n) \, \chi \, \nabla \RRe u\|_2 \, \|u\|_{L_2(\Omega)}  
,  
\end{eqnarray*}
where $M_4 = \sum_{k=1}^d \|c_k\|_\infty$.
Finally
\[
|I_6|
= \Big| \RRe \int_\Omega c_0 \, u \, 
        \Big( (\one - \chi_n)^2 \, \chi^2 \, \RRe u - \delta \, \one \Big)^+ \Big|
\leq \|c_0\|_\infty \, \|u\|_{L_2(\Omega)}^2
.  \]

Since 
\[
\mu \, \|\one_{[u_n > 0]} \, (\one - \chi_n) \, \chi \, \nabla \RRe u\|_2^2
\leq I_1
\leq M_1 \, \|u_n\|_{H^1(\Omega)} + |I_2| + |I_3| + |I_4| + |I_5| + |I_6|
\]
we obtain
\begin{eqnarray*}
\mu \, \|\one_{[u_n > 0]} \, (\one - \chi_n) \, \chi \, \nabla \RRe u\|_2^2
& \leq & M_5 \, \|\one_{[u_n > 0]} \, (\one - \chi_n) \, \chi \, \nabla \RRe u\|_2 
   + M_6  \\
& \leq & \frac{\mu}{2} \, \|\one_{[u_n > 0]} \, (\one - \chi_n) \, \chi \, \nabla \RRe u\|_2^2
   + \frac{1}{2 \mu} \, M_5^2 + M_6,
\end{eqnarray*}
where 
$M_5 = 2 M_1 + M_2 \, (1 + \|\nabla \chi\|_2) \, \|u\|_{C_b(\Omega)}
       + (M_3 + M_4) \, \|u\|_{L_2(\Omega)}$ and 
$M_6 = M_1 \, \|u\|_{L_2(\Omega)} + 4 M_1 \, (1 + \|\nabla \chi\|_2) \, \|u\|_{C_b(\Omega)}
   + 2 M_3 \, (1 + \|\nabla \chi\|_2) \, \|u\|_{L_2(\Omega)} \, \|u\|_{C_b(\Omega)} 
   + \|c_0\|_\infty \, \|u\|_{L_2(\Omega)}^2$.
Therefore 
\begin{equation}
\|\one_{[u_n > 0]} \, (\one - \chi_n) \, \chi \, \nabla \RRe u\|_2^2
\leq M_7
,  
\label{eldat243.1;1}
\end{equation}
where $M_7 = \mu^{-2} \, M_5^2 + 2 \mu^{-1} \, M_6$.
We will used this estimate in two different ways.

First note that 
\begin{equation}
|(\one_{[\chi^2 \RRe u - \delta > 0]} \, \chi \, \nabla \RRe u)(x)|^2
\leq \liminf_{n \to \infty} |(\one_{[u_n > 0]} \, (\one - \chi_n) \, \chi \, \nabla \RRe u)(x)|^2
\label{eldat243.1;2}
\end{equation}
for almost every $x \in \Omega$ with $(\chi^2 \, \RRe u)(x) - \delta > 0$.
Actually one has an equality and the liminf is a limit.
On the other hand, if $(\chi^2 \, \RRe u)(x) - \delta \leq 0$, then 
(\ref{eldat243.1;2}) is trivially valid since the left hand side vanishes.
So (\ref{eldat243.1;2}) is valid for almost every $x \in \Omega$.
Then by Fatou's lemma one deduces that 
\[
\int_\Omega |\one_{[\chi^2 \RRe u - \delta > 0]} \, \chi \, \nabla \RRe u|^2
\leq M_7
.  \]
This is for all $\delta > 0$.
Then the monotone convergence theorem (or again Fatou's lemma) gives
\[
\int_\Omega |\one_{[\chi \RRe u > 0]} \, \chi \, \nabla \RRe u|^2
= \int_\Omega |\one_{[\chi^2 \RRe u > 0]} \, \chi \, \nabla \RRe u|^2
\leq M_7
.  \]
Let $k \in \{ 1,\ldots,d \} $.
Then 
\[
\partial_k (\chi \, \RRe u)^+
= \one_{[\chi \RRe u > 0]} \Big( \chi \, \partial_k \RRe u + (\RRe u) \, \partial_k \chi \Big)
\in L_2(\Omega)
.  \]
Hence $(\chi \, \RRe u)^+ \in H^1(\Omega)$.

Alternatively, returning to the estimate (\ref{eldat243.1;1}) 
one deduces that  
\[
\|u_n\|_{H^1_0(\Omega)}^2
\leq \|u\|_{L_2(\Omega)}^2 
   + 3 \|\one_{[u_n > 0]} \, (\one - \chi_n) \, \chi \, \nabla \RRe u\|_2^2
   + 12 \, (1 + \|\nabla \chi\|_2^2) \, \|u\|_{C_b(\Omega)}^2
\leq M_8
,  \]
where 
$M_8 = \|u\|_{L_2(\Omega)}^2 + 3 M_7 + 12 \, (1 + \|\nabla \chi\|_2^2) \, \|u\|_{C_b(\Omega)}^2$.
We emphasise that this bound is uniform in $\delta$ and~$n$.

The sequence $(u_n)_{n \in \Ni, \; n > 1/\delta}$ is bounded in $H^1_0(\Omega)$.
Passing to a subsequence if necessary, there exists a $v \in H^1_0(\Omega)$
such that $\lim u_n = v$ weakly in $H^1_0(\Omega)$.
Since $\lim u_n = (\chi^2 \, \RRe u - \delta \, \one)^+$ in $L_2(\Omega)$
it follows that $(\chi^2 \, \RRe u - \delta \, \one)^+ = v \in H^1_0(\Omega)$.
Moreover, 
\[
\|(\chi^2 \, \RRe u - \delta \, \one)^+\|_{H^1_0(\Omega)}
\leq \liminf_{n \to \infty} \|u_n\|_{H^1_0(\Omega)}
\leq M_8^{1/2}
.  \]
Hence the sequence $((\chi^2 \, \RRe u - 2^{-m} \, \one)^+)_{m \in \Ni}$ is bounded in $H^1_0(\Omega)$.
Now $\lim (\chi^2 \, \RRe u - 2^{-m} \, \one)^+ = (\chi^2 \, \RRe u)^+$ in $L_2(\Omega)$.
Arguing again with a weakly convergent subsequence one deduces that 
$(\chi^2 \, \RRe u)^+ \in H^1_0(\Omega)$.
The proof of Lemma~\ref{ldat243.1} is complete.
\end{proof}

\begin{proof}[{\bf Proof of Theorem~\ref{tdat243}.}]
First note that $\limsup_{x \to z} (\chi^2 \, \RRe u)(x) \leq 0$
for quasi every $z \in \Gamma$.
So by Lemma~\ref{ldat243.1} one deduces that $(\chi \, \RRe u)^+ \in H^1(\Omega)$.
Therefore $\ca ( (\chi \, \RRe u)^+ ) \in H^{-1}(\Omega)$ by Lemma~\ref{ldat260}\ref{ldat260-3}.
Further $\limsup_{x \to z} (\one_\Omega^2 \RRe (\chi \, \RRe u)^+ )(x) \leq 0$
for quasi every $z \in \Gamma$.
Hence we can apply Lemma~\ref{ldat243.1} with $u$ replaced by $(\chi \, \, \RRe u)^+$ and 
$\chi$ by $\one_\Omega$ to conclude that 
$(\chi \, \RRe u)^+ = \one_\Omega^2 \RRe (\chi \, \RRe u)^+ \in H^1_0(\Omega)$.
This completes the proof of Theorem~\ref{tdat243}.
\end{proof}

Theorem~\ref{tdat243} has many consequences.

\begin{cor} \label{cdat243.3}
Let $u \in C_b(\Omega) \cap H^1_\loc(\Omega)$ and $\chi \in C_b(\Omega) \cap H^1(\Omega)$
with $0 \leq \chi \leq \one$.
Suppose that $\ca u \in H^{-1}(\Omega)$
and $\lim_{x \to z} (\chi \, u)(x) = 0$ for quasi every $z \in \Gamma$.
Then $\chi \, u \in H^1_0(\Omega)$.
\end{cor}
\begin{proof}
Replacing $u$ by $-u$ or $\pm i u$ gives 
$(\chi \, \RRe u)^-, (\chi \, \IIm u)^+, (\chi \, \IIm u)^- \in H^1_0(\Omega)$.
Hence $\chi \, u \in H^1_0(\Omega)$.
\end{proof}

We emphasise that one may choose $\chi = \one_\Omega$ in Theorem~\ref{tdat243}
or Corollary~\ref{cdat243.3}.
The next corollary is Theorem~\ref{tdat834.6}, which we writing again 
for the convenience of the reader.

\begin{cor} \label{cdat244}
Let $\varphi \in C(\Gamma)$.
Let $u \in C_b(\Omega) \cap H^1_\loc(\Omega)$ be $\ca$-harmonic.
Then the following are equivalent.
\begin{tabeleq}
\item \label{cdat244-1}
$u = T \varphi$. 
\item \label{cdat244-2}
There exists a polar set $P \subset \Gamma$ such that 
$
\lim\limits_{x \to z} u(x) = \varphi(z)
$
for all $z \in \Gamma \setminus P$.
\end{tabeleq}
\end{cor}
\begin{proof}
Let $S$ be the set of non-regular points of $\Gamma$, see (\ref{eSdat3;1}).
Then $\capp S = 0$.

`\ref{cdat244-1}$\Rightarrow$\ref{cdat244-2}'. 
This follows from Theorem~\ref{tdat831} and the choice $P = S$.

`\ref{cdat244-2}$\Rightarrow$\ref{cdat244-1}'. 
Let $\widehat P = P \cup S$ and $\hat u = u - T \varphi$.
Then $\capp \widehat P = 0$ and 
$\lim_{x \to z} \hat u(x) = 0$ for all $z \in \Gamma \setminus \widehat P$
by Theorem~\ref{tdat831}.
Moreover, $\ca \hat u = 0$.
Therefore $\hat u \in H^1_0(\Omega)$ by Corollary~\ref{cdat243.3} with the choice $\chi = \one_\Omega$.
Hence $A^D \hat u = 0$.
Since $0 \not\in \sigma(A^D)$ one concludes that $\hat u = 0$.
Therefore $u = T \varphi$.
\end{proof}

One also obtains a generalisation of \cite{Keldys} Theorem~IX, 
where the following corollary was proved for the Laplacian.

\begin{cor} \label{cdat246}
Let $u,v \in C_b(\Omega) \cap H^1_\loc(\Omega)$ be $\ca$-harmonic.
Suppose that 
\[
\lim_{x \to z} u(x) 
= \lim_{x \to z} v(x) 
\]
for all regular $z \in \Gamma$
(the limits exist and are equal).
Then $u = v$.
\end{cor}

Corollary~\ref{cdat243.3} can be used to obtain a new description of 
the Perron solution.

\begin{cor} \label{cdat347}
Let $\varphi \in C(\Gamma)$ and $u \in C_b(\Omega) \cap H^1_\loc(\Omega)$.
Suppose that $\ca u \in H^{-1}(\Omega)$ and 
$\lim_{x \to z} u(x) = \varphi(z)$ for quasi every $z \in \Gamma$.
By Lemma~\ref{ldat260}\ref{ldat260-1} there exists a unique $v \in H^1_0(\Omega)$ such that 
$\ca v = \ca u$.
Then $v \in C_b(\Omega)$ and $T \varphi = u - v$.
\end{cor}
\begin{proof}
Define $\tilde v = u - T \varphi$.
Then $\tilde v \in C_b(\Omega) \cap H^1_\loc(\Omega)$ and 
$\ca \tilde v = \ca u \in H^{-1}(\Omega)$.
Moreover, $\lim_{x \to z} \tilde v(x) = 0$ for quasi every $z \in \Gamma$.
Hence by Corollary~\ref{cdat243.3} one deduces that $\tilde v \in H^1_0(\Omega)$.
Then the uniqueness of $v$ implies that $v = \tilde v$.
\end{proof}

For Proposition~\ref{pdat241} we need a kind of dual map of $\ca$, so 
that we do not need to assume that $u \in H^1_\loc(\Omega)$.
Define the map $\ca^t \colon H^1_\loc(\Omega) \to \cd'(\Omega)$ by
\[
\langle \ca^t u,v \rangle_{\cd'(\Omega) \times \cd(\Omega)}
= \sum_{k,l=1}^d \int_\Omega a_{lk} \, (\partial_l u) \, \overline{\partial_k v}
   - \sum_{k=1}^d \int_\Omega c_k \, u \, \overline{\partial_k v}
   - \sum_{k=1}^d \int_\Omega \overline{b_k} \, (\partial_k u) \, \overline v
   + \int_\Omega \overline{c_0} \, u \, \overline v
\]
for all $u \in H^1_\loc(\Omega)$ and $v \in C_c^\infty(\Omega)$.

\begin{prop} \label{pdat241}
Suppose that $a_{kl}, c_k \in W^{1,\infty}(\Omega)$ for all 
$k,l \in \{ 1,\ldots,d \} $.
Let $\varphi \in C(\Gamma)$ and $u \in C_b(\Omega)$.
Suppose $\lim_{x \to z} u(x) = \varphi(z)$ for quasi every $z \in \Gamma$ and
that there exists a $v \in H^1_0(\Omega)$ such that 
\[
\int_\Omega u \, \overline{\ca^t \tau}
= \gota(v,\tau)
\]
for all $\tau \in C_c^\infty(\Omega)$.
Then $u \in H^1_\loc(\Omega)$, $v \in C_b(\Omega)$ and $T \varphi = u - v$.
\end{prop}
\begin{proof}
By assumption one has $\int_\Omega (u - v) \, \overline{\ca^t \tau} = 0$
for all $\tau \in C_c^\infty(\Omega)$.
Hence $u - v \in H^1_\loc(\Omega)$ by elliptic regularity, see \cite{AEG} Proposition~A.1.
So $u \in H^1_\loc(\Omega)$ and 
\[
\langle \ca u, \tau \rangle_{\cd'(\Omega) \times \cd(\Omega)}
= \int_\Omega u \, \overline{\ca^t \tau}
= \gota(v,\tau)
= \langle \ca v, \tau \rangle_{\cd'(\Omega) \times \cd(\Omega)}
\]
for all $\tau \in C_c^\infty(\Omega)$.
Therefore $\ca u = \ca v \in H^{-1}(\Omega)$ by Lemma~\ref{ldat260}\ref{ldat260-3}
and the result follows from Corollary~\ref{cdat347}.
\end{proof}

As a special case of Proposition~\ref{pdat241} we obtain a new result for the Laplacian.

\begin{cor} \label{cdat242}
Let $\varphi \in C(\Gamma)$ and $u \in C_b(\Omega)$.
Suppose $\lim_{x \to z} u(x) = \varphi(z)$ for quasi every $z \in \Gamma$ and
$\Delta u \in H^{-1}(\Omega)$.
Let $v \in H^1_0(\Omega)$ be such that $\Delta v = \Delta u$ as distribution.
Then $v \in C_b(\Omega)$ and $T^{-\Delta} \varphi = u - v$.
\end{cor}

This corollary extends \cite{AD2} Theorem~1.1, where it was assumed that 
$\lim_{x \to z} u(x) = \varphi(z)$ for all $z \in \Gamma$.

\section{Approximation with $C(\overline \Omega)$-functions} \label{Sdat4}

The aim of this section is to prove Theorem~\ref{tdat121}.
In particular, we show that for all $\varphi \in C(\Gamma)$ the Perron solution 
$T \varphi$ is an element of $H^1(\Omega)$ if and only if $\varphi$
has an extension in $C(\overline \Omega) \cap H^1(\Omega)$.
The main tool is the next theorem, which is even valid for $H^1_\loc$-functions
instead of $H^1$-functions.

Throughout this section we adopt the notation and assumptions as in 
the first part of the introduction.

\begin{thm} \label{tdat401}
Let $\varphi \in C(\Gamma)$ and $u \in C_b(\Omega) \cap H^1_\loc(\Omega)$
be such that $\ca u \in H^{-1}(\Omega)$ and 
$\lim_{x \to z} u(x) = \varphi(z)$ for quasi every $z \in \Gamma$.
Then for all $\varepsilon > 0$ there exists a $\Phi \in C(\overline \Omega) \cap H^1_\loc(\Omega)$
such that $\Phi|_\Gamma = \varphi$,
$u - \Phi \in H^1_0(\Omega)$ and $\|u - \Phi\|_{H^1(\Omega)} < \varepsilon$.
\end{thm}

The proof requires a delicate induction procedure. 
It relies on the next lemma.

\begin{lemma} \label{ldat402}
Let $\varphi \in C(\Gamma)$ and $w \in C_b(\Omega) \cap H^1_\loc(\Omega)$
be both real valued.
Let $z_0 \in \Gamma$ and $r_0,\beta,\delta > 0$.
Suppose that 
\begin{eqnarray*}
& & \ca w \in H^{-1}(\Omega),  \\
& & \lim_{x \to z} w(x) = \varphi(z) \quad \mbox{for quasi every } z \in \Gamma, \mbox{ and}  \\
& & |\varphi(z) - \varphi(z_0)| \leq \beta \quad \mbox{for all }z \in \Gamma \cap B(z_0,4 r_0) .
\end{eqnarray*}
Then there exists a real valued $\widetilde w \in C_b(\Omega) \cap H^1_\loc(\Omega)$ such that 
the following is valid.
\begin{tabel}
\item \label{ldat402-1}
$\widetilde w - w \in H^1_0(\Omega)$.
Consequently also $\ca \widetilde w \in H^{-1}(\Omega)$.
\item \label{ldat402-2}
$\|\widetilde w - w\|_{H^1_0(\Omega)} \leq \delta$.
\item \label{ldat402-3}
$\limsup_{x \to z} |\widetilde w(x) - \varphi(z)|
\leq \limsup_{x \to z} |w(x) - \varphi(z)|$
for all $z \in \Gamma$.
\item \label{ldat402-4}
$\lim_{x \to z} \widetilde w(x) = \varphi(z)$ for quasi every $z \in \Gamma$.
\item \label{ldat402-5}
$\limsup_{x \to z} |\widetilde w(x) - \varphi(z)|
\leq 2 \beta$ for all $z \in \Gamma \cap B(z_0,r_0)$.
\item \label{ldat402-6}
$\widetilde w(x) = w(x)$ for all $x \in \Omega \setminus B(z_0,2 r_0)$.
\item \label{ldat402-7}
For all $z_1 \in \Gamma$, $R_1 \in [4 r_0, \infty)$, $\gamma \in [\beta,\infty)$
and $R \in (0,\frac{1}{2} \, R_1]$ with 
\begin{eqnarray*}
& & |\varphi(z) - \varphi(z_1)| \leq \gamma \quad \mbox{for all } z \in \Gamma \cap B(z_1,R_1), \mbox{ and}  \\
& & |w(x) - \varphi(z_1)| \leq 2 \gamma \quad \mbox{for all } x \in \Omega \cap B(z_1,R)
\end{eqnarray*}
it follows that 
\[
|\widetilde w(x) - \varphi(z_1)| \leq 2 \gamma \quad \mbox{for all } x \in \Omega \cap B(z_1,R)
.  \]
\end{tabel} 
\end{lemma}
\begin{proof}
Define 
\[
K = \{ z \in \overline{B(z_0,r_0)} \cap \Gamma : \limsup_{x \to z} |w(x) - \varphi(z)| \geq \beta \} 
.  \]
Then $K$ is compact by the same argument as at the beginning of the proof of 
Lemma~\ref{ldat243.1}.
Moreover, $\capp K = 0$ since $\lim_{x \to z} w(x) = \varphi(z)$ for quasi every 
$z \in \Gamma$.
In particular $|K| = 0$.
Let $\chi \in C_c^\infty(\Ri^d)$ be such that
such that $\chi(x) = 1$ for all $x$ in a neighbourhood of $K$,
$0 \leq \chi \leq \one$ 
and $\supp \chi \subset B(z_0,2 r_0)$.
There exist many such $\chi$.
At the moment we need the existence of one such $\chi$ and at a later 
stage we show with the aid of Lemma~\ref{ldat834.2}
that there exists one such that all properties, including~\ref{ldat402-2},
are valid for the function that we define next.

Define 
\begin{equation}
\widetilde w  
= (\one - \chi^2) \, w
   + \chi^2 \, \Big( (\varphi(z_0) - \beta) \vee w \wedge (\varphi(z_0) + \beta) \Big)
.  
\label{eldat402;3}
\end{equation}
Then $\widetilde w  \in C_b(\Omega) \cap H^1_\loc(\Omega)$.

Clearly $\widetilde w (x) = w(x)$ for all $x \in \Omega \setminus B(z_0,2 r_0)$,
which gives~\ref{ldat402-6}.

Let $z \in \Gamma \cap B(z_0,4 r_0)$.
Then $\varphi(z) \in [\varphi(z_0) - \beta, \varphi(z_0) + \beta]$.
If $x \in \Omega$, then
\[
\widetilde w (x) - \varphi(z)
= (1 - \chi^2(x)) \, (w(x) - \varphi(z))
   + \chi^2(x) \, \Big( (\varphi(z_0) - \beta) \vee w(x) \wedge (\varphi(z_0) + \beta) 
                         - \varphi(z) \Big)
.  \]
Since $|a \vee t \wedge b - c| \leq |t-c|$ for all $c \in [a,b]$ and $t \in \Ri$,
one deduces that 
\[
|\widetilde w (x) - \varphi(z)|
\leq (1 - \chi^2(x)) \, |w(x) - \varphi(z)|
   + \chi^2(x) \, |w(x) - \varphi(z)|
= |w(x) - \varphi(z)|
.  \]
Hence $\limsup_{x \to z} |\widetilde w (x) - \varphi(z)| \leq \limsup_{x \to z} |w(x) - \varphi(z)|$.
Next let $z \in \Gamma \setminus B(z_0,4 r_0)$.
If $x \in B(z,r_0)$, then $x \not\in B(z_0,2 r_0)$ and hence $\widetilde w (x) = w(x)$.
Therefore 
\[
\limsup_{x \to z} |\widetilde w (x) - \varphi(z)| = \limsup_{x \to z} |w(x) - \varphi(z)|
.  \]
This gives~\ref{ldat402-3}.
Since $\lim_{x \to z} w(x) = \varphi(z)$ for quasi every $z \in \Gamma$,
this implies that also $\lim_{x \to z} \widetilde w (x) = \varphi(z)$ for quasi every $z \in \Gamma$,
which gives~\ref{ldat402-4}.
Moreover, $\lim_{x \to z} (\widetilde w  - w)(x) = 0$ for quasi every $z \in \Gamma$.

Clearly 
\begin{eqnarray}
\widetilde w  - w
& = & \chi^2 \, \Big( (\varphi(z_0) - \beta) \vee w \wedge (\varphi(z_0) + \beta) 
                         - w \Big)  \nonumber \\
& = & \chi^2 \, \Big( (w - (\varphi(z_0) - \beta))^- -  (w - (\varphi(z_0) + \beta))^+ \Big)
\label{eldat402;1}
\end{eqnarray}
and 
\begin{equation}
\nabla (\widetilde w  - w)
= 2 \chi \, (\nabla \chi) \, \Big( (\varphi(z_0) - \beta) \vee w \wedge (\varphi(z_0) + \beta) 
                         - w \Big)
   + \chi^2 \, \one_{[|w - \varphi(z_0)| > \beta]} \, \nabla w
.
\label{eldat402;2}
\end{equation}
Hence $(\widetilde w - w)^- = \chi^2 \, (w - (\varphi(z_0) + \beta))^+$
and 
\[
\lim_{x \to z} \Big( \chi^2 \, (w - (\varphi(z_0) + \beta))^+ \Big)(x)
= \lim_{x \to z} (\widetilde w - w)^-(x) = 0
\]
for quasi every $z \in \Gamma$.
Therefore $\lim_{x \to z} \Big( \chi^2 \, (w - (\varphi(z_0) + \beta)) \Big)(x) \leq 0$
for quasi every $z \in \Gamma$.
Since $\ca w \in H^{-1}(\Omega)$ by assumption and also $\one_\Omega \in H^1(\Omega)$,
one deduces that $\ca(w - (\varphi(z_0) + \beta)) \in H^{-1}(\Omega)$.
It follows from Theorem~\ref{tdat243} that 
\[
\Big( \chi^2 \, (w - (\varphi(z_0) + \beta)) \Big)^+ \in H^1_0(\Omega)
.  \]
Similarly $\Big( \chi^2 \, (w - (\varphi(z_0) - \beta)) \Big)^- \in H^1_0(\Omega)$
and then by (\ref{eldat402;1}) also $\widetilde w - w \in H^1_0(\Omega)$, which gives~\ref{ldat402-1}.

Fix $\chi_0 \in C_c^\infty(\Ri^d)$ such that
such that $\chi_0(x) = 1$ for all $x$ in a neighbourhood $V_0$ of $K$,
$0 \leq \chi_0 \leq \one$ 
and $\supp \chi_0 \subset B(z_0,2 r_0)$.
We emphasise that such a $\chi_0$ exists.
Define $\widetilde w_0 \: (= \widetilde w)$ as in (\ref{eldat402;3}) with the choice $\chi = \chi_0$.
Then (\ref{eldat402;2}) implies that 
\[
\int_{V_0} \one_{[|w - \varphi(z_0)| > \beta]} \, |\nabla w|^2
\leq \|\widetilde w_0 - w\|_{H^1_0(\Omega)}^2
< \infty
.  \]
Recall that $|K| = 0$ and $K \subset V_0$.
Then it follows from the Lebesgue dominated convergence theorem
that there is an open $V \subset \Ri^d$ with $K \subset V \subset V_0$ such that 
\[
\int_V \one_{[|w - \varphi(z_0)| > \beta]} \, |\nabla w|^2 < \frac{1}{4} \delta^2
.  \]
Because $K$ is compact and $\capp K = 0$, one deduces from Lemma~\ref{ldat834.2}
that there exist $\chi \in C_c^\infty(\Ri^d)$ and an open neighbourhood $U$ of $K$
such that 
\[
\|\chi\|_{H^1(\Omega)}
\leq \frac{\delta}{8 (|\varphi(z_0)| + \beta + \|w\|_\infty)}
,  \]
$0 \leq \chi \leq \one$ and $\chi(x) = 1$ for all $x \in U$
and $\supp \chi \subset B(z_0,2 r_0) \cap V$.
With this choice of $\chi$ we prove the lemma.
Define $\widetilde w$ as in (\ref{eldat402;3}).
We already showed~\ref{ldat402-1}, \ref{ldat402-3}, \ref{ldat402-4} and \ref{ldat402-6}.

It follows from (\ref{eldat402;1}) that 
\[
\|\widetilde w  - w\|_2
\leq (|\varphi(z_0)| + \beta + \|w\|_\infty) \, \|\chi\|_2
\leq \frac{\delta}{8} 
\]
and since $\supp \chi \subset V$ one deduces from (\ref{eldat402;2}) that 
\begin{eqnarray*}
\|\nabla (\widetilde w  - w)\|_2
& \leq & 2 (|\varphi(z_0)| + \beta + \|w\|_\infty) \, \|\nabla \chi\|_2
   + \Big( \int_V \one_{[|w - \varphi(z_0)| > \beta]} \, |\nabla w|^2 \Big)^{1/2}  \\
& \leq & \frac{\delta}{4} + \frac{\delta}{2}
.  
\end{eqnarray*}
Hence 
$\|\widetilde w  - w\|_{H^1(\Omega)} \leq \delta$, which gives~\ref{ldat402-2}.

Let $z \in K$.
If $x \in \Omega \cap U$, then $\chi(x) = 1$ and $\varphi(z) \in [\varphi(z_0) - \beta, \varphi(z_0) + \beta]$.
So 
\[
|\widetilde w (x) - \varphi(z)|
= \Big| (\varphi(z_0) - \beta) \vee w(x) \wedge (\varphi(z_0) + \beta) 
                         - \varphi(z) \Big|
\leq 2 \beta 
.  \]
So $\limsup_{x \to z} |\widetilde w (x) - \varphi(z)| \leq 2 \beta$.
If $z \in \Gamma \cap B(z_0, r_0) \setminus K$, then by~\ref{ldat402-3} and the definition of $K$
we obtain that $\limsup_{x \to z} |\widetilde w (x) - \varphi(z)| \leq \beta \leq 2 \beta$.
That gives~\ref{ldat402-5}.

Finally we show~\ref{ldat402-7}.
Let $z_1 \in \Gamma$, $R_1 \in [4 r_0, \infty)$, $\gamma \in [\beta,\infty)$ 
and $R \in (0,\frac{1}{2} \, R_1]$.
Suppose that 
\begin{eqnarray*}
& & |\varphi(z) - \varphi(z_1)| \leq \gamma \quad \mbox{for all } z \in \Gamma \cap B(z_1,R_1), \mbox{ and}  \\
& & |w(x) - \varphi(z_1)| \leq 2 \gamma \quad \mbox{for all } x \in \Omega \cap B(z_1,R) .
\end{eqnarray*}
Let $x \in \Omega \cap B(z_1,R)$.
If $x \not\in B(z_0,2 r_0)$, then $\widetilde w (x) = w(x)$ and hence
\[
|\widetilde w (x) - \varphi(z_1)| 
= |w(x) - \varphi(z_1)| 
\leq 2 \gamma
.  \]
Now suppose that $x \in B(z_0,2 r_0)$.
Then 
$|z_0 - z_1| \leq |z_0 - x| + |x - z_1| < 2 r_0 + R \leq R_1$.
So $|\varphi(z_0) - \varphi(z_1)| \leq \gamma$.
Hence $\varphi(z_0) + \beta \leq \varphi(z_1) + \gamma + \beta \leq \varphi(z_1) + 2 \gamma$.
Similarly $\varphi(z_0) - \beta \geq \varphi(z_1) - 2 \gamma$.
Therefore 
\[
(\varphi(z_0) - \beta) \vee w(x) \wedge (\varphi(z_0) + \beta) 
\in [ \varphi(z_1) - 2 \gamma , \varphi(z_1) + 2 \gamma ]
\]
and
\[
| (\varphi(z_0) - \beta) \vee w(x) \wedge (\varphi(z_0) + \beta) 
  - \varphi(z_1) |
\leq 2 \gamma
.  \]
Then 
\begin{eqnarray*}
|\widetilde w (x) - \varphi(z_1)|
& \leq & (1 - \chi^2(x)) \, |w(x) - \varphi(z_1)|  \\*
& & \hspace*{10mm} {}
   + \chi^2(x) \, | (\varphi(z_0) - \beta) \vee w(x) \wedge (\varphi(z_0) + \beta) 
  - \varphi(z_1) |  \\
& \leq & (1 - \chi^2(x)) \cdot 2 \gamma + \chi^2(x) \cdot 2 \gamma
= 2 \gamma
.  
\end{eqnarray*}
This proves~\ref{ldat402-7} and completes the proof of the lemma.
\end{proof}

Now we are able to prove Theorem~\ref{tdat401}.

\begin{proof}[{\bf Proof of Theorem~\ref{tdat401}.}]
We first show that $\ca \RRe u \in H^{-1}(\Omega)$.
Since $\ca u \in H^{-1}(\Omega)$ there are $f_1,\ldots,f_d \in L_2(\Omega)$
such that 
\begin{equation}
\langle \ca u, v \rangle_{\cd'(\Omega) \times \cd(\Omega)}
= \sum_{k=1}^d \int_\Omega f_k \, \overline{\partial_k v}
\label{etdat401;2}
\end{equation}
for all $v \in C_c^\infty(\Omega)$.
Let $v \in C_c^\infty(\Omega)$ be real valued.
Taking the real part in (\ref{etdat401;2}) it follows that 
\begin{equation}
\langle \ca \RRe u, v \rangle_{\cd'(\Omega) \times \cd(\Omega)}
= \sum_{k=1}^d \int_\Omega (\RRe f_k) \, \overline{\partial_k v}
   + \sum_{k=1}^d \int_\Omega i (\IIm b_k) \, \overline u \, \overline{\partial_k v}
   + \int_\Omega i (\IIm c_0) \, \overline u \, \overline  v
.
\label{etdat401;3}
\end{equation}
Then by linearity (\ref{etdat401;3}) is valid for all $v \in C_c^\infty(\Omega)$.
Since $i (\IIm b_k) \, \overline u \in L_2(\Omega)$ and $i (\IIm c_0) \, \overline u \in L_2(\Omega)$
for all $k \in \{ 1,\ldots,d \} $ it follows that $\ca \RRe u \in H^{-1}(\Omega)$.

So without loss of generality we may assume that $u$ and $\varphi$ are real valued.
For all $\eta > 0$ define 
\[
K_\eta = \{ z \in \Gamma : \limsup_{x \to z} |u(x) - \varphi(z)| \geq \eta \} 
.  \]
Then $K_\eta$ is compact and $\capp K_\eta = 0$ since 
$\lim_{x \to z} u(x) = \varphi(z)$ for quasi every $z \in \Gamma$.

Define $u_0 = u$.
We shall show that there exist real valued $u_1,u_2,\ldots \in C_b(\Omega) \cap H^1_\loc(\Omega)$
such that for all $n \in \Ni$ one has
\begin{eqnarray*}
& & \ca u_n \in H^{-1}(\Omega),  \\
& & u_n - u_{n-1} \in H^1_0(\Omega) ,  \\
& & \|u_n - u_{n-1}\|_{H^1_0(\Omega)} \leq \varepsilon \, 2^{-n} ,  \\
& & \limsup_{x \to z} |u_n(x) - \varphi(z)| \leq \limsup_{x \to z} |u_{n-1}(x) - \varphi(z)| ,  \\
& & \lim_{x \to z} u_n(x) = \varphi(z) \quad \mbox{for quasi every } z \in \Gamma ,  \\
& & \limsup_{x \to z} |u_n(x) - \varphi(z)| \leq \tfrac{2}{n} \quad \mbox{for all } 
    z \in K_{1/n} ,  \\
& & u_n(x) = u_{n-1}(x) \quad \mbox{for all } x \in \Omega \mbox{ with } 
     d(x,\Gamma) \geq \tfrac{2}{n} , 
\end{eqnarray*}
and for all $z_1 \in \Gamma$, $R_1 \in [\frac{4}{n}, \infty)$, $\gamma \in [\frac{1}{n},\infty)$
and $R \in (0,\frac{1}{2} \, R_1]$ with 
\begin{equation}
\left[
\begin{array}{l}
|\varphi(z) - \varphi(z_1)| \leq \gamma \quad \mbox{for all } z \in \Gamma \cap B(z_1,R_1), \mbox{ and}  \\[5pt]
|u_{n-1}(x) - \varphi(z_1)| \leq 2 \gamma \quad \mbox{for all } x \in \Omega \cap B(z_1,R)
\end{array}
\right.
\label{etdat401;1}
\end{equation}
it follows that 
\[
|u_n(x) - \varphi(z_1)| \leq 2 \gamma \quad \mbox{for all } x \in \Omega \cap B(z_1,R)
.  \]
The proof is by induction.
Let $n \in \Ni$ and suppose that $u_{n-1}$ is defined.
Since $K_{1/n}$ is compact, there is an $N \in \Ni$ such that for all 
$k \in \{ 1,\ldots,N \} $ there exist $z_k \in K_{1/n}$ and 
$r_k \in (0,\frac{1}{n}]$ such that 
$\varphi(z) \in ( \varphi(z_k) - \frac{1}{n} , \varphi(z_k) + \frac{1}{n} )$
for all $z \in B(z_k, 4 r_k)$ and 
$K_{1/n} \subset \bigcup_{k=1}^N B(z_k, r_k)$.
Define $w_0 = u_{n-1}$.
We shall define $w_1,\ldots,w_N$ by induction.
Let $k \in \{ 1,\ldots,N \} $ and suppose that $w_{k-1}$ is defined. 
Apply Lemma~\ref{ldat402} with $w = w_{k-1}$, $z_0 = z_k$, $r_0 = r_k$,
$\beta = \frac{1}{n}$ and $\delta = \frac{\varepsilon}{N \, 2^n}$
and define $w_k = \widetilde w$.
Then define $u_n = w_N$.
It is easy to verify that $u_n$ satisfies the required properties.

Let $x \in \Omega$.
Let $N \in \Ni$ be such that $\frac{2}{N} < d(x,\Gamma)$.
Then $u_n(x) = u_N(x)$ for all $n \in \Ni$ with $n \geq N$.
Hence we can define $\Phi_0 \colon \Omega \to \Ri$ by 
\[
\Phi_0(x) = \lim_{n \to \infty} u_n(x)
.  \]
Since $u_n \in C(\Omega)$ for all $n \in \Ni$, one deduces that 
$\Phi_0 \in C(\Omega)$.
The series $\sum (u_n - u_{n-1})$ is absolutely convergent in $H^1(\Omega)$.
Hence $\lim (u_n - u_0)$ exists in $H^1(\Omega)$.
Therefore $\Phi_0 - u \in H^1(\Omega)$ and 
\[
\|\Phi_0 - u\|_{H^1(\Omega)}
= \|\sum_{n=1}^\infty (u_n - u_{n-1})\|_{H^1(\Omega)}
\leq \varepsilon
.  \]
Moreover, $\Phi_0 - u = \sum_{n=1}^\infty (u_n - u_{n-1}) \in H^1_0(\Omega)$.
Since $u \in H^1_\loc(\Omega)$, also $\Phi_0 \in H^1_\loc(\Omega)$.

Note that if $n \in \Ni$, then 
\[
\limsup_{x \to z} |u_n(x) - \varphi(z)|
\leq \tfrac{2}{n}
\]
for all $z \in K_{1/n}$ and hence, by definition of $K_{1/n}$
and the construction of $u_n$,
for all $z \in \Gamma$.
Fix $z_1 \in \Gamma$.
We shall show that $\lim_{x \to z_1} \Phi_0(x) = \varphi(z_1)$.
Let $\varepsilon > 0$.
There exists an $R_1 > 0$ such that 
$|\varphi(z) - \varphi(z_1)| \leq \varepsilon$ for all $z \in \Gamma \cap B(z_1, R_1)$.
Let $N \in \Ni$ be such that $\frac{4}{N} \leq R_1$ and 
$\frac{2}{N} \leq \varepsilon$.
Since $\limsup_{x \to z_1} |u_N(x) - \varphi(z_1)| \leq \frac{2}{N} \leq \varepsilon$,
there exists an $R \in (0,\frac{1}{2} \, R_1]$ such that 
$|u_N(x) - \varphi(z_1)| \leq 2 \varepsilon$ for all $x \in \Omega \cap B(z_1,R)$.
Choose $\gamma = \varepsilon$ in (\ref{etdat401;1}).
Then it follows by induction that $|u_n(x) - \varphi(z_1)| \leq 2 \varepsilon$
for all $x \in \Omega \cap B(z_1,R)$ and $n \in \Ni$ with $n \geq N$.
Consequently
\[
|\Phi_0(x) - \varphi(z_1)|
= \lim_{n \to \infty} |u_n(x) - \varphi(z_1)|
\leq 2 \varepsilon
\]
for all $x \in \Omega \cap B(z_1,R)$.
Hence $\lim_{x \to z_1} \Phi_0(x) = \varphi(z_1)$.

Finally define $\Phi \colon \overline \Omega \to \Ri$ by 
\[
\Phi(x) 
= \left\{ \begin{array}{ll}
   \Phi_0(x) & \mbox{if } x \in \Omega ,  \\[5pt]
   \varphi(x) & \mbox{if } x \in \Gamma .
           \end{array} \right.
\]
Then $\Phi \in C(\overline \Omega)$ and the proof is complete.
\end{proof}

\begin{cor} \label{cdat405}
Let $\varphi \in C(\Gamma)$.
Then there exists a $\Phi \in C(\overline \Omega) \cap H^1_\loc(\Omega)$
such that $\Phi|_\Gamma = \varphi$ and $\ca \Phi \in H^{-1}(\Omega)$.
\end{cor}
\begin{proof}
Let $u = T \varphi$.
Then $u \in C_b(\Omega) \cap H^1_\loc(\Omega)$, $\ca u = 0 \in H^{-1}(\Omega)$
and $\lim_{x \to z} u(x) = \varphi(z)$
for quasi every $z \in \Gamma$ by Corollary~\ref{cdat244}\ref{cdat244-1}$\Rightarrow$\ref{cdat244-2}.
Hence by Theorem~\ref{tdat401}
there exists a $\Phi \in C(\overline \Omega) \cap H^1_\loc(\Omega)$ 
such that $\Phi|_\Gamma = \varphi$ and $T \varphi - \Phi \in H^1_0(\Omega)$.
Then use Lemma~\ref{ldat260}\ref{ldat260-3}.
\end{proof}

Corollary~\ref{cdat405} allows us to describe the Perron solution,
as stated in Theorem~\ref{tdat121} in the introduction.

\begin{proof}[{\bf Proof of Theorem~\ref{tdat121}.}]
Let $\varphi \in C(\Gamma)$.
By Corollary~\ref{cdat405} there exists a 
$\Phi \in C(\overline \Omega) \cap H^1_\loc(\Omega)$ such that 
$\Phi|_\Gamma = \varphi$ and $\ca \Phi \in H^{-1}(\Omega)$.
Then by Lemma~\ref{ldat260}\ref{ldat260-1} there exists a unique
$v \in H^1_0(\Omega)$ such that $\ca v = \ca \Phi$.
Now apply Corollary~\ref{cdat347} with $u = \Phi|_\Omega$.
One deduces that $T \varphi = \Phi|_\Omega - v$.
\end{proof}

Now we can also determine for which $\varphi \in C(\Gamma)$ 
one has $T \varphi \in H^1(\Omega)$.

\begin{cor} \label{cdat403}
Let $\varphi \in C(\Gamma)$.
Then the following are equivalent.
\begin{tabeleq}
\item \label{cdat403-1}
$T \varphi \in H^1(\Omega)$.
\item \label{cdat403-2}
There exists a $\Phi \in C(\overline \Omega) \cap H^1(\Omega)$ 
such that $\Phi|_\Gamma = \varphi$.
\item \label{cdat403-3}
There exists a $u \in C_b(\Omega) \cap H^1(\Omega)$ such that 
$\lim_{x \to z} u(x) = \varphi(z)$ for quasi every $z \in \Gamma$.
\item \label{cdat403-4}
For all $\Phi \in C(\overline \Omega) \cap H^1_\loc(\Omega)$ 
such that $\ca \Phi \in H^{-1}(\Omega)$ and $\Phi|_\Gamma = \varphi$
one has $\Phi \in H^1(\Omega)$.
\item \label{cdat403-5}
For all $u \in C_b(\Omega) \cap H^1_\loc(\Omega)$ such that 
$\ca u \in H^{-1}(\Omega)$ and
$\lim_{x \to z} u(x) = \varphi(z)$ for quasi every $z \in \Gamma$
one has $u \in H^1(\Omega)$.
\end{tabeleq}
\end{cor}
\begin{proof}
`\ref{cdat403-1}$\Leftrightarrow$\ref{cdat403-2}'.
This is a direct consequence of the last statement
in Theorem~\ref{tdat121}.

`\ref{cdat403-2}$\Rightarrow$\ref{cdat403-3}'.
Trivial.

`\ref{cdat403-3}$\Rightarrow$\ref{cdat403-5}'.
There exists a $\tilde u \in C_b(\Omega) \cap H^1(\Omega)$ such that 
$\lim_{x \to z} \tilde u(x) = \varphi(z)$ for quasi every $z \in \Gamma$.
Then $u - \tilde u \in C_b(\Omega) \cap H^1_\loc(\Omega)$, 
$\ca (u - \tilde u) \in H^{-1}(\Omega)$ and
$\lim_{x \to z} (u - \tilde u)(x) = 0$ for quasi every $z \in \Gamma$.
By Corollary~\ref{cdat243.3} it follows that $u - \tilde u \in H^1_0(\Omega)$,
whence $u \in H^1(\Omega)$.

`\ref{cdat403-5}$\Rightarrow$\ref{cdat403-4}'.
Trivial.

`\ref{cdat403-4}$\Rightarrow$\ref{cdat403-2}'.
Clear by Corollary~\ref{cdat405}.
\end{proof}

A famous example by Hadamard shows that there exists a $\varphi \in C(\Gamma)$ 
such that the classical solution of the Dirichlet problem $(D_\varphi)$ is 
not in $H^1(\Omega)$, where $\Omega = \{ x \in \Ri^2 : |x| < 1 \} $.
Hence the classical solution of the Dirichlet problem 
cannot be obtained by variational methods in general.
See \cite{MazyaShaposhnikova} Section~12.3 and 
\cite{ArendtUrbanEng} Example~6.68.
Therefore it is interesting to know for which $\varphi \in C(\Gamma)$ one has 
$T \varphi \in H^1(\Omega)$.

\medskip

The property $T \varphi \in H^1(\Omega)$ is independent of the 
elliptic operator.

\begin{cor} \label{cdat404}
Let $\varphi \in C(\Gamma)$.
Then the following are equivalent.
\begin{tabeleq}
\item \label{cdat404-1}
$T \varphi \in H^1(\Omega)$.
\item \label{cdat404-2}
$T^{-\Delta} \varphi \in H^1(\Omega)$.
\end{tabeleq}
\end{cor}
\begin{proof}
Condition~\ref{cdat403-2} in Corollary~\ref{cdat403} is independent of the operator.
\end{proof}

One might wonder whether the property $\ca u \in H^{-1}(\Omega)$ for a 
function $u \in C(\overline \Omega) \cap H^1_\loc(\Omega)$ depends on the operator.
We show in Example~\ref{xdat410} that this is the case, in general.
Nevertheless, one has the following characterisation for $\ca u \in H^{-1}(\Omega)$.

\begin{prop} \label{pdat411}
Let $u \in C(\overline \Omega) \cap H^1_\loc(\Omega)$ and $\varphi \in C(\Gamma)$.
Suppose that $\lim_{x \to z} u(x) = \varphi(z)$ for quasi every $z \in \Gamma$.
Then the following are equivalent.
\begin{tabeleq}
\item \label{pdat411-1}
$\ca u \in H^{-1}(\Omega)$.
\item \label{pdat411-2}
$u - T \varphi \in H^1_0(\Omega)$.
\item \label{pdat411-3}
For all $\Phi \in C(\overline \Omega) \cap H^1_\loc(\Omega)$ with 
$\Phi|_\Gamma = \varphi$ and $\ca \Phi \in H^{-1}(\Omega)$ 
one has $u - \Phi \in H^1_0(\Omega)$.
\end{tabeleq}
\end{prop}
\begin{proof}
`\ref{pdat411-1}$\Rightarrow$\ref{pdat411-3}'.
Let $\Phi \in C(\overline \Omega) \cap H^1_\loc(\Omega)$ and suppose that  
$\Phi|_\Gamma = \varphi$ and $\ca \Phi \in H^{-1}(\Omega)$.
Then $\ca(u - \Phi) \in H^{-1}(\Omega)$ and 
$\lim_{x \to z} (u - \Phi)(x) = 0$ for quasi every $z \in \Gamma$.
Hence Corollary~\ref{cdat243.3} implies that $u - \Phi \in H^1_0(\Omega)$.

`\ref{pdat411-3}$\Rightarrow$\ref{pdat411-2}'.
By Theorem~\ref{tdat121} there exists a 
$\Phi \in C(\overline \Omega) \cap H^1_\loc(\Omega)$ such that 
$\Phi|_\Gamma = \varphi$ and $\ca \Phi \in H^{-1}(\Omega)$.
Note that $T \varphi$ satisfies the same assumptions as the function $u$ 
by Corollary~\ref{cdat244}.
Then the implication \ref{pdat411-1}$\Rightarrow$\ref{pdat411-3} with 
$u$ replaced by $T \varphi$ gives $T \varphi - \Phi \in H^1_0(\Omega)$.
Consequently $u - T \varphi = (u - \Phi) - (T \varphi - \Phi) \in H^1_0(\Omega)$.

`\ref{pdat411-2}$\Rightarrow$\ref{pdat411-1}'.
Since $u = (u - T \varphi) + T \varphi$ one deduces that
$\ca u \in H^{-1}(\Omega)$.
\end{proof}

We give an example of an open bounded set $\Omega \subset \Ri^2$,
an elliptic operator and a function $u \in C(\overline \Omega) \cap H^1_\loc(\Omega)$
such that $\Delta u \in H^{-1}(\Omega)$, but $\ca u \not\in H^{-1}(\Omega)$.

\begin{exam} \label{xdat410}
For all $n \in \Ni$ define $I_n = (2^{-n}, 2^{-n} + 2 \cdot 4^{-n})$.
Then the sets $\overline{I_n}$ are 
pairwise disjoint.
Set $\Omega_1 = \bigcup_{n=1}^\infty I_n$.
Define $w \colon \overline{\Omega_1} \to \Ci$ by
\[
w(x)
= \left\{ \begin{array}{ll}
   0 & \mbox{if } x = 0 ,  \\[5pt]
   \frac{1}{n} \, 4^n (x - 2^{-n}) & \mbox{if } n \in \Ni \mbox{ and } x \in \overline{I_n} .
          \end{array} \right.
\]
Then $w \in C(\overline{\Omega_1}) \cap H^1_\loc(\Omega_1)$ 
and $\Delta_1 w = 0$, where $\Delta_1$ is the Laplacian on $\Omega_1$.
Next for all $n \in \Ni$ let $v_n \colon I_n \to \Ci$ be the tent on $I_n$
with height~$4^{-n}$. 
So 
\[
v_n(x) = 4^{-n} - |x- (2^{-n} + 4^{-n})|
.  \]
Then $v_n \in H^1_0(I_n)$ and 
$\|v_n\|_{H^1_0(I_n)}^2 = \frac{2}{3} \cdot 4^{-3n} + 2 \cdot 4^{-n}$.
Hence 
\[
v = \sum_{n=1}^\infty v_n \in H^1_0(\Omega_1)
.  \]
Next define $\Omega = \Omega_1 \times \Omega_1 \subset \Ri^2$.
Then $\Omega$ is open and bounded.
Consider $u = w \otimes w \in C(\overline \Omega) \cap H^1_\loc(\Omega)$.
Then $\Delta u = 0 \in H^{-1}(\Omega)$, where $\Delta$ is the Laplacian on $\Omega$.
In order to define the coefficients of the operator $\ca$ it is 
convenient to split the interval $I_n$ in the left part
$I_n^l = (2^{-n}, 2^{-n} + 4^{-n}]$ and right part
$I_n^r = (2^{-n} + 4^{-n}, 2^{-n} + 2 \cdot 4^{-n})$ for all $n \in \Ni$.
Define $\alpha \colon \Omega_1 \to [1,2]$ by 
\[
\alpha(x)
= \left\{ \begin{array}{ll}
   2 & \mbox{if } n \in \Ni \mbox{ and } x \in I_n^l ,  \\[5pt]
   1 & \mbox{if } n \in \Ni \mbox{ and } x \in I_n^r .
          \end{array} \right.
\]
For all $k,l \in \{ 1,2 \} $ define $a_{kl} \colon \Omega \to \Ri$ by 
$a_{kl} = (\alpha \otimes \alpha) \, \delta_{kl}$.
Moreover, set $b_k = c_k = 0$ and $c_0 = 0$, and let $\ca$ be the 
corresponding elliptic operator.

Now suppose that $\ca u \in H^{-1}(\Omega)$.
Then there is an $M > 0$ such that 
\[
|\langle \ca u, \tau \rangle_{\cd'(\Omega) \times \cd(\Omega)}|
\leq M \, \|\tau\|_{H^1_0(\Omega)}
\]
for all $\tau \in C_c^\infty(\Omega)$.
Let $N \in \Ni$.
For all $n \in \{ 1,\ldots,N \} $ let $f_n \in C_c^\infty(I_n)$.
Then $f = \sum_{n=1}^N f_n \in C_c^\infty(\Omega_1)$ and 
\begin{eqnarray*}
2 \Big| \Big( \int_{\Omega_1} \alpha \, w' \, \overline{f'} \Big)
        \Big( \int_{\Omega_1} \alpha \, w \, \overline f \Big) \Big|
& = & \Big| \sum_{k,l=1}^2 \int_\Omega a_{kl} \, 
         (\partial_k (w \otimes w)) \, \overline{ \partial_l (f \otimes f) } \Big|  \\
& = & |\langle \ca u, f \otimes f \rangle_{\cd'(\Omega) \times \cd(\Omega)}|
\leq M \, \|f \otimes f\|_{H^1_0(\Omega)}
.
\end{eqnarray*}
Hence by density, approximating $v_n$ by elements in $C_c^\infty(I_n)$,
one deduces that 
\[
2 \Big| \Big( \int_{\Omega_1} \sum_{n=1}^N \alpha \, w' \, \overline{v_n'} \Big)
        \Big( \int_{\Omega_1} \sum_{n=1}^N \alpha \, w \, \overline{v_n} \Big) \Big|
\leq M \, \Big\| \Big( \sum_{n=1}^N v_n \Big) \otimes \Big( \sum_{n=1}^N v_n \Big) \Big\|_{H^1_0(\Omega)}
\leq M \, \|v \otimes v\|_{H^1_0(\Omega)}
.  \]
But 
\[
\int_{I_n} \alpha \, w' \, v_n'
= \int_{I_n^l} 2 \cdot \tfrac{1}{n} \, 4^n  \cdot 1
   + \int_{I_n^r} 1 \cdot \tfrac{1}{n} \, 4^n  \cdot (-1)
= \frac{1}{n}
\]
and $\int_{I_n} \alpha \, w \, v_n > 0$ for all $n \in \Ni$.
Therefore 
\[
2 \Big( \int_{I_1} \alpha \, w \, v_1 \Big) \sum_{n=1}^N \frac{1}{n}
\leq M \, \|v \otimes v\|_{H^1_0(\Omega)}
< \infty
\]
for all $N \in \Ni$.
This is a contradiction.
\end{exam}

\section{Dirichlet problem with the approximative trace} \label{Sdat5}

Adopt the notation and assumptions as in the first part of the introduction.
In addition we provide $\Gamma$ with the $(d-1)$-dimensional Hausdorff measure, 
denoted by $\sigma$.
Throughout this section we assume that $\sigma(\Gamma) < \infty$.
Instead of data in $C(\Gamma)$ we wish to consider data in $L_2(\Gamma)$.

We recall from the introduction the notion of a trace.
Given $u \in H^1(\Omega)$ and $\varphi \in L_2(\Gamma)$ we say that 
$\varphi$ is an {\bf approximative trace}, or simply $\varphi$ is a {\bf trace} 
of $u$, or $u$ has {\bf trace} $\varphi$, if there are 
$u_1,u_2,\ldots \in H^1(\Omega) \cap C(\overline \Omega)$ such that
$\lim u_n = u$ in $H^1(\Omega)$ and $\lim u_n|_\Gamma = \varphi$ in $L_2(\Gamma)$.
If $u \in H^1(\Omega)$, then the set 
\[
\{ \varphi \in L_2(\Gamma) : \varphi \mbox{ is a trace of } u \}
\]
of all traces of~$u$ may be empty, and also it may contain two and hence infinitely many elements.
If $\Omega$ has a Lipschitz boundary, then for each
$u \in H^1(\Omega)$ there is a unique $\varphi \in L_2(\Gamma)$ 
such that $\varphi$ is a trace of~$u$.
Of course, we are mainly interested in more general domains.
Set 
\[
\Tr(\Omega)
= \{ \varphi \in L_2(\Gamma) : \mbox{there exists a $u \in H^1(\Omega)$ such that 
$\varphi$ is a trace of $u$} \}
.  \]
Recall from Lemma~\ref{ldat260}\ref{ldat260-2} that 
\[
H^1(\Omega) = H^1_0(\Omega) \oplus \ch_\ca^1(\Omega)
.  \]
Hence for all $w \in H^1(\Omega)$ there are $v \in H^1_0(\Omega)$ and 
$u \in \ch_\ca^1(\Omega)$ such that $w = v + u$.
Consequently, for all $\varphi \in \Tr(\Omega)$ there exists a $u \in \ch_\ca^1(\Omega)$
such that $\varphi$ is a trace of $u$.
We call $u$ a {\bf variational solution of the Dirichlet problem $(D_\varphi)$}.
In general there are several such variational solutions.

We next consider uniqueness of the variational solution.
We introduce 
\[
V(\Omega) = \{ u \in H^1(\Omega) : 0 \mbox{ is a trace of } u \} 
,  \]
the space of all $H^1(\Omega)$ functions with vanishing approximative trace.
Clearly $H^1_0(\Omega) \subset V(\Omega)$.

\begin{remark} \label{rdat540}
By \cite{Sau1} Corollary~7.39 one has 
\[
V(\Omega)
\subset \{ u|_\Omega : u \in H^1(\Ri^d) \mbox{ and } u = 0 \mbox{ a.e.\ on } 
     \Ri^d \setminus \overline \Omega \} 
.  \]
Consequently $V(\Omega) = H^1_0(\Omega)$ if $\Omega$ is stable for the 
Dirichlet problem, see \cite{AD} Theorem~4.7.
In particular, this is the case if $\Omega$ has a continuous boundary 
in the sense of graphs, see \cite{AD} Proposition~2.2.
\end{remark}

The characterisation for the uniqueness of the variational solution
is as follows.

\begin{thm} \label{tdat202}
The following conditions are equivalent.
\begin{tabeleq}
\item \label{tdat202-1}
For all $\varphi \in \Tr(\Omega)$ there is a unique variational solution
of the Dirichlet problem $(D_\varphi)$.
\item \label{tdat202-2}
$V(\Omega) = H^1_0(\Omega)$.
\item \label{tdat202-3}
$V(\Omega) \cap \ch^1_\ca(\Omega) = \{ 0 \} $.
\end{tabeleq}
\end{thm}
\begin{proof}
`\ref{tdat202-1}$\Rightarrow$\ref{tdat202-2}'.
Suppose~\ref{tdat202-1}.
Let $w \in V(\Omega)$.
Using Lemma~\ref{ldat260}\ref{ldat260-2} we can write $w = v + u$
with $v \in H^1_0(\Omega)$ and $u \in \ch^1_\ca(\Omega)$.
Then $0$ is a trace of $u$ and $u$ is $\ca$-harmonic.
Clearly also $0 \in L_2(\Gamma)$ is a trace of $0 \in H^1(\Omega)$
and $0$ is $\ca$-harmonic.
Since the variational solution is unique, one deduces that $u = 0$.
Then $w = v \in H^1_0(\Omega)$.

`\ref{tdat202-2}$\Rightarrow$\ref{tdat202-3}'.
$V(\Omega) \cap \ch^1_\ca(\Omega) = H^1_0(\Omega) \cap \ch^1_\ca(\Omega) = \{ 0 \} $.

`\ref{tdat202-3}$\Rightarrow$\ref{tdat202-1}'.
Let $\varphi \in L_2(\Gamma)$ and $u_1,u_2 \in \ch^1_\ca(\Omega)$ 
be such that $\varphi$ is a trace of both $u_1$ and $u_2$.
Then $0$ is a trace of $u_1 - u_2$.
So $u_1 - u_2 \in V(\Omega)$ and $u_1 - u_2 \in \ch^1_\ca(\Omega)$.
So $u_1 - u_2 \in V(\Omega) \cap \ch^1_\ca(\Omega) = \{ 0 \} $ and $u_1 = u_2$.
\end{proof}

It follows from Daners \cite{Daners2} Proposition~3.3
that there exists a measurable set $\Gamma_s \subset \Gamma$ such that 
\[
\{ \varphi \in L_2(\Gamma) : \varphi \mbox{ is a trace of } 0 \}
= L_2(\Gamma_s)
.  \]
For an elementary proof, see \cite{AE2} Lemma~4.14.
We define $\Gamma_r = \Gamma \setminus \Gamma_s$.
We say that {\bf the trace on $\Omega$ is unique} if every $u \in H^1(\Omega)$ 
has at most one (approximative) trace.
So the trace on $\Omega$ is unique if and only if $\sigma(\Gamma_s) = 0$.
The trace is unique if $\Omega$ has a continuous boundary (in the sense of graphs).
If $d = 2$ then it suffices that $\Omega$ is connected.
See \cite{Sau2} for these two and other results on the trace.

We next give an example of a bounded connected open set $\Omega \subset \Ri^2$ that is 
Wiener regular, on which the trace is unique, but where uniqueness of
a variational solution fails. 
We even construct an ($\ca$-harmonic) function $u \in C(\overline \Omega) \cap \ch_\ca^1(\Omega) \cap V(\Omega)$
such that $u \not\in H^1_0(\Omega)$.
Hence for every $\varphi \in \Tr(\Omega)$ and $u_0 \in \ch_\ca^1(\Omega)$ 
such that $\varphi$ is a trace of $u_0$, we obtain that also $u_0 + c \, u \in \ch_\ca^1(\Omega)$
and $\varphi$ is a trace of $u_0 + c \, u$ for all $c \in \Ci$.
In particular, for each $\varphi \in \Tr(\Omega)$ 
there are infinitely many variational solutions
of the Dirichlet problem $(D_\varphi)$.

\begin{exam} \label{xdat203}
Let $\cc \subset \Ri$ be the usual middle third Cantor set and set $S = \cc \times \{ 0 \} $.
Define $\Omega = B(0,2) \setminus S$.
Clearly $S$ is compact, so $\Omega$ is open and bounded in $\Ri^2$.
It is not hard to see that $\Omega$ is also connected.
Obviously $\sigma(S) = 0$.

We recall that if $E \subset \Ri^d$ and $a \in \Ri^d$, then the set $E$ is 
called {\bf thin at $a$} if 
\[
\int_0^1 \frac{\capp( E \cap B(a,r) )}{r^{d-2}} \, \frac{dr}{r} 
< \infty
.  \]
If $U \subset \Ri^d$ is open, then a point $a \in \partial U$ is 
a regular point for $U$ if and only if $U^{\rm c}$ is not thin at $a$
by the theorem on page~143 in \cite{Wiener}.
It follows from \cite{Tsuji} Theorem~1 that every element of $S$ is a 
regular point for $\Omega$.
Hence every element of $\Gamma$ is a regular point for~$\Omega$ and $\Omega$ is 
Wiener regular.
In particular $\capp(B(a,r) \setminus \Omega) > 0$ for all $a \in S$ and $r \in (0,1)$.
Since $B(a,r) \setminus \Omega = B(a,r) \cap S$ one establishes that $\capp S > 0$.

Define $\varphi \in C(\Gamma)$ by
\[
\varphi(z) 
= \left\{ \begin{array}{ll}
    1 & \mbox{if } z \in S,  \\[5pt]
    0 & \mbox{if } z \in \partial B(0,2) .
          \end{array} \right.
\]
Since $\Omega$ is Wiener regular, there exists a unique classical solution
of the Dirichlet problem with boundary data $\varphi$.
Let $u \in C(\overline \Omega) \cap H^1_\loc(\Omega)$ be the unique $\ca$-harmonic function
such that $u|_\Gamma = \varphi$.
So $u|_\Omega = T \varphi$.
We next show that $u \in H^1(\Omega)$.
Let $w \in C_c^\infty(B(0,2))$ be such that $w|_{B(0,1)} = \one$.
Then $\varphi = w|_\Gamma$
and $T \varphi = T(w|_\Gamma) \in H^1(\Omega)$ by Theorem~\ref{tdat120}.
Therefore $u \in H^1(\Omega) \cap C(\overline \Omega)$, so by definition
$u|_\Gamma$ is a trace of~$u$.
Since $\sigma(S) = 0$ one deduces that $u|_\Gamma = \varphi = 0$ 
$\sigma$-almost everywhere.
So $0$ is a trace of $u$ and $u \in V(\Omega)$.
Because $\capp S > 0$, the function $u$ does not vanish quasi-everywhere
on $\Gamma$ and hence $u \not\in H^1_0(\Omega)$.
So in this example $V(\Omega) \neq H^1_0(\Omega)$.
Note that $u \in C(\overline \Omega) \cap \ch_\ca^1(\Omega) \cap V(\Omega)$, 
but $u \not\in H^1_0(\Omega)$ and $u \not\in C_0(\Omega)$.

Note that $\|T \varphi\|_{L_2(\Omega)} = \|u\|_{L_2(\Omega)} > 0$
and $\|\varphi\|_{L_2(\Gamma)} = 0$.
Hence there does not exist a $c > 0$ such that 
$\|T \psi\|_{L_2(\Omega)} \leq c \, \|\psi\|_{L_2(\Gamma)}$ for all 
$\psi \in C(\Gamma)$.
In particular, the map $T \colon C(\Gamma) \to C_b(\Omega)$
does not extend to a bounded linear map from $L_2(\Gamma)$ into $L_2(\Omega)$.
\end{exam}

The example shows that the classical solution (and in particular 
the Perron solution) is a much finer notion than the 
variational solution for a given continuous approximative trace.
In the example one has $\varphi \in C(\Gamma) \cap \Tr(\Omega) \setminus \{ 0 \}$ and 
the Dirichlet problem $(D_\varphi)$ has a unique classical solution $u \in C(\overline \Omega)$.
But since $\varphi = 0$ $\sigma$-a.e.\ on $\Gamma$, the zero-function
is another variational solution.
The problem does not occur if $V(\Omega) = H^1_0(\Omega)$.

Suppose $V(\Omega) = H^1_0(\Omega)$, that is for all $\varphi \in \Tr(\Omega)$ the variational solution
of the Dirichlet problem $(D_\varphi)$ is unique.
Define the map $\gamma \colon \Tr(\Omega) \to \ch_\ca^1(\Omega)$ by 
\[
\gamma(\varphi) = u
,  \]
where $u \in \ch_\ca^1(\Omega)$ is the unique element
with trace~$\varphi$.

We next wish to prove that $\gamma$ maps bounded functions in $\Tr(\Omega)$
into bounded functions.
We need a lemma.

\begin{lemma} \label{ldat215}
Let $\varphi \in \Tr(\Omega) \cap L_\infty(\Gamma)$.
Then there exist $u_1,u_2,\ldots \in H^1(\Omega) \cap C(\overline \Omega)$
and $u \in H^1(\Omega)$ such that $\lim u_n = u$ in $H^1(\Omega)$,
$\lim u_n|_\Gamma = \varphi$ in $L_2(\Gamma)$ and 
$\|u_n\|_{L_\infty(\Omega)} \leq 2 \|\varphi\|_{L_\infty(\Gamma)}$
for all $n \in \Ni$.
\end{lemma}
\begin{proof}
First suppose that $\varphi$ is real valued.
There are $u_1,u_2,\ldots \in H^1(\Omega) \cap C(\overline \Omega)$
and $u \in H^1(\Omega)$ such that $\lim u_n = u$ in $H^1(\Omega)$ and 
$\lim u_n|_\Gamma = \varphi$ in $L_2(\Gamma)$.
Taking the real part we may assume that $u_n$ and $u$ are
real valued for all $n \in \Ni$.
Define $v_n = (-\|\varphi\|_{L_\infty(\Gamma)} ) \vee u_n \wedge \|\varphi\|_{L_\infty(\Gamma)}$
for all $n \in \Ni$.
Then $v_n \in H^1(\Omega) \cap C(\overline \Omega)$ for all $n \in \Ni$.
Moreover, 
$\lim v_n = (-\|\varphi\|_{L_\infty(\Gamma)} ) \vee u \wedge \|\varphi\|_{L_\infty(\Gamma)}
\in H^1(\Omega)$ in $H^1(\Omega)$ 
and $\lim v_n|_\Gamma = \varphi$ in $L_2(\Gamma)$.
Clearly $\|v_n\|_{L_\infty(\Omega)} \leq \|\varphi\|_{L_\infty(\Gamma)}$
for all $n \in \Ni$.

If $\varphi$ is general, then write $\varphi = \RRe \varphi + i \IIm \varphi$
and apply the above to $\RRe \varphi$ and $\IIm \varphi$.
\end{proof}

\begin{prop} \label{pdat216}
Suppose that $V(\Omega) = H^1_0(\Omega)$.
Let $\varphi \in \Tr(\Omega) \cap L_\infty(\Gamma)$.
Then $\gamma(\varphi) \in L_\infty(\Omega)$ and 
$\|\gamma(\varphi)\|_{L_\infty(\Omega)} \leq 2 c \, \|\varphi\|_{L_\infty(\Gamma)}$,
where $c = \|T\|_{C(\Gamma) \to C_b(\Omega)}$.
\end{prop}
\begin{proof}
By Lemma~\ref{ldat215} there are $w_1,w_2,\ldots \in H^1(\Omega) \cap C(\overline \Omega)$
and $w \in H^1(\Omega)$ such that $\lim w_n = w$ in $H^1(\Omega)$,
$\lim w_n|_\Gamma = \varphi$ in $L_2(\Gamma)$ and 
$\|w_n\|_{L_\infty(\Omega)} \leq 2 \|\varphi\|_{L_\infty(\Gamma)}$
for all $n \in \Ni$.
For all $n \in \Ni$ write $w_n = v_n + u_n$ and $w = v + u$,
where $v_n,v \in H^1_0(\Omega)$ and $u_n,u \in \ch_\ca^1(\Omega)$.
Since $H^1_0(\Omega)$ and $\ch_\ca^1(\Omega)$ are closed in $H^1(\Omega)$,
it follows from the closed graph theorem that $\lim u_n = u$ in $H^1(\Omega)$.

We know that $\gamma(\varphi) = u$.
Then $u_n = T(w_n|_\Gamma)$ for all $w \in \Ni$ by Theorem~\ref{tdat120}.
So 
\[
\|u_n\|_{L_\infty(\Omega)} 
\leq c \, \|w_n|_\Gamma\|_{L_\infty(\Gamma)}
\leq 2 c \, \|\varphi\|_{L_\infty(\Gamma)}
\]
for all $n \in \Ni$.
Now $\lim u_n = u$ in $H^1(\Omega)$
and hence also in $L_2(\Omega)$.
Passing to a subsequence if necessary we may assume that 
$\lim u_n = u$ almost everywhere.
Hence 
$\|\gamma(\varphi)\|_{L_\infty(\Omega)} 
= \|u\|_{L_\infty(\Omega)} 
\leq 2 c \, \|\varphi\|_{L_\infty(\Gamma)}$ as required.
\end{proof}

\begin{thm} \label{tdat910}
Suppose that $V(\Omega) = H^1_0(\Omega)$.
Then $T \varphi = \gamma(\varphi)$ for all $\varphi \in C(\Gamma) \cap \Tr(\Omega)$.
\end{thm}
\begin{proof}
Clearly $T (\Phi|_\Gamma) = \gamma(\Phi|_\Gamma)$ for all $\Phi \in H^1(\Omega) \cap C(\overline \Omega)$
by Theorem~\ref{tdat120}.

Let $\varphi \in C(\Gamma) \cap \Tr(\Omega)$.
By Stone--Weierstra\ss\ there exist $\Phi_1,\Phi_2,\ldots \in H^1(\Omega) \cap C(\overline \Omega)$ such that 
$\lim \Phi_n|_\Gamma = \varphi$ in $C(\Gamma)$.
Then $T \varphi = \lim T (\Phi_n|_\Gamma)$ in $C_b(\Omega)$ and in particular pointwise 
on $\Omega$.
If $n \in \Ni$, then 
\[
\|\gamma(\varphi) - \gamma(\Phi_n|_\Gamma)\|_{L_\infty(\Omega)} 
\leq 2 c \, \|\varphi - \Phi_n|_\Gamma\|_{C(\Gamma)}
\]
by Proposition~\ref{pdat216}, where $c = \|T\|_{C(\Gamma) \to C_b(\Omega)}$.
Hence $\lim \gamma(\Phi_n|_\Gamma) = \gamma(\varphi)$ almost everywhere
on $\Omega$.
So $\lim T (\Phi_n|_\Gamma) = \gamma(\varphi)$ almost everywhere
on $\Omega$.
Therefore $T \varphi = \gamma(\varphi)$ almost everywhere.
\end{proof}

One would expect that if $\varphi \in C(\Gamma)$ and 
$T \varphi \in H^1(\Omega)$,
then $\varphi$ is a trace of $T \varphi$.
The proof relies on Theorem~\ref{tdat401}, whose proof
was quite involved and depended on capacity methods.

\begin{thm} \label{tdat206}
Let $\varphi \in C(\Gamma)$.
Suppose that $T \varphi \in H^1(\Omega)$.
Then $\varphi$ is a trace of $T \varphi$.
\end{thm}
\begin{proof}
It follows from Corollary~\ref{cdat244} and the choice $u = T \varphi$ in Theorem~\ref{tdat401} 
that for all $n \in \Ni$ there exists a $\Phi_n \in C(\overline \Omega) \cap H^1_\loc(\Omega)$
such that $\Phi_n|_\Gamma = \varphi$, $T \varphi - \Phi_n \in H^1_0(\Omega)$
and in addition $\|T \varphi - \Phi_n\|_{H^1_0(\Omega)} \leq \frac{1}{n}$ for all $n \in \Ni$.
Since $T \varphi \in H^1(\Omega)$ also $\Phi_n \in H^1(\Omega)$ 
for all $n \in \Ni$.
So $\Phi_n \in C(\overline \Omega) \cap H^1(\Omega)$ for all $n \in \Ni$.
Hence $\lim \Phi_n = T \varphi$ in $H^1(\Omega)$ and 
obviously $\lim \Phi_n|_\Gamma = \varphi$ in $L_2(\Gamma)$.
Therefore $\varphi$ is a trace of $T \varphi$.
\end{proof}

Theorems~\ref{tdat910} and  \ref{tdat206}  allow us to give another characterisation
of those $\varphi \in C(\Gamma)$ for which the Perron solution $T \varphi$ is 
an element of $H^1(\Omega)$.
See also Corollary~\ref{cdat403}.

\begin{thm} \label{tdat222}
Suppose that $V(\Omega) = H^1_0(\Omega)$.
Let $\varphi \in C(\Gamma)$.
Then the following are equivalent.
\begin{tabeleq}
\item \label{tdat222-1}
$T \varphi \in H^1(\Omega)$.
\item \label{tdat222-2}
$\varphi \in \Tr(\Omega)$.
\end{tabeleq}
\end{thm}
\begin{proof}
`\ref{tdat222-1}$\Rightarrow$\ref{tdat222-2}'.
It follows from Theorem~\ref{tdat206} that $\varphi$ is a trace of $T \varphi$.
Hence $\varphi \in \Tr(\Omega)$.

`\ref{tdat222-2}$\Rightarrow$\ref{tdat222-1}'.
By assumption $\varphi \in C(\Gamma) \cap \Tr(\Omega)$.
Hence $T \varphi = \gamma(\varphi) \in H^1(\Omega)$ by Theorem~\ref{tdat910}.
\end{proof}

The next example shows that the condition $V(\Omega) = H^1_0(\Omega)$ in Theorem~\ref{tdat222}
cannot be omitted.

\begin{exam} \label{xdat222.5}
Let $\Omega$ be the Wiener regular bounded open set as in Example~\ref{xdat203}.
We shall show that there exists a $\varphi \in C(\Gamma)$ 
such that $\varphi|_{\partial B(0,2)} = 0$ and $T \varphi \not\in H^1(\Omega)$.
Hence $\varphi = 0$ almost everywhere on $\Gamma$, 
so $\varphi \in C(\Gamma) \cap \Tr(\Omega)$.

Set $d_0 = \frac{\log2}{\log 3}$.
Recall that $S = \cc \times \{ 0 \} $.
Then $S$ is a $d_0$-set in the sense of Jonsson--Wallin by 
\cite{JW} Example~II.2.
Fix $s \in (0,\frac{d_0}{2})$.
Since $s < 1 \wedge d_0$ it follows from 
\cite{CarvalhoCaetano} Theorem~3.2 that there exists a Weierstra\ss-type function $\psi \in C(S)$ such that 
$\dim_H G(\psi) = d_0 + 1 - s$,
where $\dim_H$ denotes the Hausdorff dimension and $G(\psi)$ is the graph of $\psi$.
Define $\varphi \in C(\Gamma)$ by 
\[
\varphi(z) 
= \left\{ \begin{array}{ll}
    \psi(z) & \mbox{if } z \in S,  \\[5pt]
    0 & \mbox{if } z \in \partial B(0,2) .
          \end{array} \right.
\]
Since $\Omega$ is Wiener regular, it follows from Corollary~\ref{cdat320}
that the Dirichlet problem $(D_\varphi)$ has a (unique) classical 
solution $u \in C(\overline \Omega) \cap H^1_\loc(\Omega)$.
Then $T \varphi = u|_\Omega$
and obviously $\varphi \in C(\Gamma) \cap \Tr(\Omega)$.

Now suppose that $T \varphi \in H^1(\Omega)$.
Then $u \in C(\overline \Omega) \cap H^1(\Omega) \subset H^1(B(0,2))$
by \cite{BW2} Proposition~3.6, since $\sigma(S) < \infty$.
Because every element in $H^1(B(0,2))$ can be extended 
to an element of $H^1(\Ri^2) = B^{22}_1(\Ri^2)$, a Besov space,
it follows from \cite{JW} Theorem~VII.1 that 
$\psi = u|_S \in B^{22}_{d_0/2}(S)$.
Then \cite{CarvalhoCaetano} Theorem~3.3 implies that 
$\dim_H G(\psi) \leq d_0 + 1 - \frac{d_0}{2}$.
But $\dim_H G(\psi) = d_0 + 1 - s$.
So $\frac{d_0}{2} \leq s$, which is a contradiction.
Hence $T \varphi \not\in H^1(\Omega)$.
\end{exam}

If $\Omega$ is Wiener regular, then it follows from 
Biegert--Warma \cite{BW} Theorem~3.2(c)$\Rightarrow$(b) and the 
Wiener condition (\ref{eSdat1;2}) that
each $u \in H^1_0(\Omega) \cap C(\overline \Omega)$ vanishes 
(pointwise) on~$\Gamma$.
This is no longer true if $u \in V(\Omega) \cap C(\overline \Omega)$.
In fact, in Example~\ref{xdat203} the set $\Omega$ is Wiener regular, 
but nevertheless there exists a $u \in V(\Omega) \cap C(\overline \Omega)$
such that $u|_\Gamma \neq 0$ in $C(\Gamma)$.
Next we show that any such $u$ vanishes $\sigma$-a.e.\ on~$\Gamma$.

\begin{prop} \label{pdat208}
Let $u \in V(\Omega) \cap C(\overline \Omega)$.
Then $u|_\Gamma = 0$ $\sigma$-a.e.
\end{prop}
\begin{proof}
Write $\varphi = u|_\Gamma \in C(\Gamma)$.
Recall that $\Gamma_r$ and $\Gamma_s$ are Borel measurable subsets of $\Gamma$.
Then $\varphi|_{\Gamma_r} = 0$ $\sigma$-a.e.\ by construction/definition.
Define 
\[
Z = \{ z \in \Gamma_s : \limsup_{r \downarrow 0} 
    \frac{|\Omega \cap B(z,r)|}{r^d} > 0 \}
.  \]
Then $\sigma(Z) = 0$ by \cite{Sau2} Theorem~4.15.
We shall show that $\varphi(z) = 0$ for all $z \in \Gamma_s \setminus Z$.

Let $z \in \Gamma_s \setminus Z$ and let $\delta > 0$.
We first prove that $\sigma(\Gamma_r \cap B(z,\delta)) > 0$.
Indeed, suppose that $\sigma(\Gamma_r \cap B(z,\delta)) = 0$.
By \cite{AFP} (3.43) there exists a $\gamma > 0$
such that 
\begin{equation} \label{epdat208;1}
\Big( |\Omega \cap B(z,\delta)| \wedge |\Omega^{\rm c} \cap B(z,\delta)| \Big)^{(d-1)/d}
\leq \gamma \, P(\Omega, B(z,\delta))
,
\end{equation}
where $P(\Omega, B(z,\delta))$ denotes the perimeter of $\Omega$ in $B(z,\delta)$.
Now 
\[
P(\Omega, B(z,\delta)) = \sigma( (\partial^* \Omega) \cap B(z,\delta))
\]
by \cite{Federer} Theorem~4.5.11 and \cite{AFP} (3.62), where 
$\partial^* \Omega$ is the measure theoretic boundary of $\Omega$.
Note that 
\[
\Ri^d \setminus \partial^* \Omega
= \{ p \in \Ri^d : \lim_{r \downarrow 0} \frac{|\Omega \cap B(p,r)|}{r^d} = 0
   \mbox{ or} \lim_{r \downarrow 0} \frac{|\Omega^{\rm c} \cap B(p,r)|}{r^d} = 0
  \} 
.  \]
Clearly $\partial^* \Omega \subset \Gamma$.
Moreover, $\Gamma_s \setminus Z \subset \Gamma \setminus \partial^* \Omega$.
Hence $\partial^* \Omega \subset \Gamma_r \cup Z$.
Then 
\[
P(\Omega, B(z,\delta))
= \sigma( (\partial^* \Omega) \cap B(z,\delta)) 
\leq \sigma( \Gamma_r \cap B(z,\delta)) + \sigma(Z)
= 0 
.  \]
On the other hand, $|\Omega \cap B(z,\delta)| > 0$ since $z \in \partial \Omega$
and $\Omega$ is open.
Also $\limsup_{r \downarrow 0} \frac{|\Omega \cap B(z,r)|}{r^d} = 0$
since $z \in \Gamma_s \setminus Z$.
Then $\lim_{r \downarrow 0} \frac{|\Omega \cap B(z,r)|}{r^d} = 0$
and consequently 
$\lim_{r \downarrow 0} \frac{|\Omega^{\rm c} \cap B(z,r)|}{|B(z,r)|} = 1$.
So $|\Omega^{\rm c} \cap B(z,\delta)| > 0$.
This contradicts (\ref{epdat208;1}).

Let $z \in \Gamma_s \setminus Z$.
We proved that $\sigma(\Gamma_r \cap B(z,\delta)) > 0$
for all $\delta > 0$.
Since $\varphi|_{\Gamma_r} = 0$ $\sigma$-a.e.\ and $\varphi$ is continuous,
one deduces that $\varphi(z) = 0$.
\end{proof}

We conclude with a sufficient condition which implies 
for a Perron solution to vanish.

\begin{prop} \label{pdat209}
Suppose that $\sigma(\Gamma \cap B(z,\delta)) > 0$ for all $z \in \Gamma$ 
and $\delta > 0$.
Let $\varphi \in C(\Gamma)$ and suppose that $T \varphi \in V(\Omega)$.
Then $\varphi = 0$ and $T \varphi = 0$.
\end{prop} 
\begin{proof}
It follows from Corollary~\ref{cdat244} and Theorem~\ref{tdat401}
that there exists a  $\Phi \in C(\overline \Omega) \cap H^1(\Omega)$
such that $\Phi|_\Gamma = \varphi$ and $T \varphi - \Phi \in H^1_0(\Omega)$.
Then $\Phi \in V(\Omega)$.
By Proposition~\ref{pdat208} there exists a measurable set 
$N \subset \Gamma$ such that $\sigma(N) = 0$ and 
$\Phi(z) = 0$ for all $z \in \Gamma \setminus N$.
Let $z \in \Gamma$.
If $n \in \Ni$, then 
$\sigma((\Gamma \setminus N) \cap B(z,\frac{1}{n})) > 0$, hence there 
exists a $z_n \in (\Gamma \setminus N) \cap B(z,\frac{1}{n})$.
Then $\varphi(z) = \lim \varphi(z_n) = 0$.
\end{proof}

By the next lemma, the first condition in the previous proposition is satisfied if
$\Omega$ is regular in topology, that is if $\cov{\overline \Omega} = \Omega$.
On the other hand, it is not sufficient to suppose that $\Omega$ is 
regular in capacity by Example~\ref{xdat203}.

\begin{lemma} \label{ldat210}
Suppose $\Omega$ is regular in topology.
Then $\sigma(\Gamma \cap B(z,\delta)) > 0$ for all $z \in \Gamma$ and $\delta > 0$.
\end{lemma}
\begin{proof}
Let $z \in \Gamma$ and $\delta > 0$.
Then $|\Omega \cap B(z,\delta)| > 0$ and $|\Omega^c \cap B(z,\delta)| > 0$.
Hence by (\ref{epdat208;1}) one deduces that $P(\Omega, B(z,\delta)) > 0$.
Arguing as in the proof of Proposition~\ref{pdat208} one obtains
\[
0 
< P(\Omega, B(z,\delta))
= \sigma( (\partial^* \Omega) \cap B(z,\delta)) 
\leq \sigma(\Gamma \cap B(z,\delta))
\]
as required.
\end{proof}

\begin{cor} \label{cdat531}
If $\Omega$ is regular in topology, then 
\[
C(\overline \Omega) \cap \ch^1_\ca(\Omega) \cap V(\Omega) = \{ 0 \} 
.  \]
\end{cor}
\begin{proof}
Let $u \in C(\overline \Omega) \cap \ch^1_\ca(\Omega) \cap V(\Omega)$.
Then $u = T(u|_\Gamma)$ by Proposition~\ref{pdat808}.
Now use Lemma~\ref{ldat210} and Proposition~\ref{pdat209}.
\end{proof}

Applying the arguments in \cite{Sau1} Example~7.55 to the domain in Example~\ref{xdat203},
one can construct a topologically and Wiener regular bounded set
$\Omega$ with $\sigma(\Gamma) < \infty$
such that $H^1_0(\Omega)\subsetneqq V(\Omega)$.
By Theorem~\ref{tdat202} the variational solution of the Dirichlet problem 
is not unique on this set $\Omega$.
Still, for every $\varphi\in C(\Gamma)$ with $T\varphi\in H^1(\Omega)$ there exists a unique
$u \in C(\overline \Omega) \cap \ch^1_\ca(\Omega)$
with trace $\varphi$ by Theorem~\ref{tdat206} and Corollary~\ref{cdat531}.

\subsection*{Acknowledgements}
The second-named author is most grateful for the hospitality extended
to him during a fruitful stay at the University of Ulm.
He wishes to thank the University of Ulm for financial support.
The third-named author had a wonderful time during his visit at the 
University of Auckland in early 2020.
He is most grateful for the hospitality, the productive collaboration and the 
financial support.
This work is supported by the Marsden Fund Council from Government funding,
administered by the Royal Society of New Zealand.

\end{document}